\documentclass[10pt]{amsart}
\usepackage{eepic}
\usepackage{epsfig}
\usepackage{graphicx}
\usepackage{eufrak}
\usepackage{mathrsfs}
\usepackage[english]{babel}
\usepackage{color}
\usepackage{mathabx}

\usepackage{palatino, mathpazo}
\usepackage{amscd}      
\usepackage{amssymb}
\usepackage{amsmath}
\usepackage{dsfont}
\usepackage{xypic}      
\LaTeXdiagrams          
\usepackage[all,v2]{xy}
\xyoption{2cell} \UseAllTwocells \xyoption{frame} \CompileMatrices
\usepackage{latexsym}
\usepackage{epsfig}
\usepackage{amsfonts}
\usepackage{enumerate}

\allowdisplaybreaks[3]

\newtheorem{prop}{Proposition}[section]
\newtheorem{lem}[prop]{Lemma}

\newtheorem{cor}[prop]{Corollary}
\newtheorem{thm}[prop]{Theorem}
\newtheorem{rmk}[prop]{Remark}
\newtheorem{example}{Example}

\newtheorem{conjecture}[prop]{Conjecture}

\newtheorem{defn}[prop]{Definition}

            {\nolinebreak $\Box$ \end{trivlist}}

\newcommand{\noprint}[1]{}

\renewcommand{\tilde}{\widetilde}

\newcommand{\Ext}{\mbox{Ext}}
\newcommand{\spf}{\mbox{Spf}}
\newcommand{\Hom}{\mbox{Hom}}

\newcommand{\an}{{\mbox{\tiny an}}}

\newcommand{\et}{{\mbox{\tiny \'{e}t}}}

\newcommand{\alg}{\mbox{\tiny alg}}

\newcommand{\virt}{\mbox{\tiny virt}}
\newcommand{\cone}{\mbox{cone}}

\newcommand{\XX}{{\mathfrak X}}

\renewcommand{\SS}{{\mathfrak S}}
\newcommand{\TT}{{\mathfrak T}}

\newcommand{\YY}{{\mathfrak Y}}
\newcommand{\UU}{{\mathfrak U}}

\newcommand{\ZZ}{{\mathfrak Z}}

\newcommand{\FF}{{\mathfrak F}}

\newcommand{\MM}{{\mathfrak M}}

\newcommand{\RR}{{\mathfrak R}}

\newcommand{\VV}{{\mathfrak V}}
\newcommand{\WW}{{\mathfrak W}}

\newcommand{\Tt}{{\mathfrak t}}

\newcommand{\CC}{{\mathfrak C}}
\newcommand{\zz}{{\mathbb Z}}

\newcommand{\kk}{{\mathbb K}}

\newcommand{\aaa}{{\mathbb A}}

\newcommand{\nn}{{\mathbb N}}

\renewcommand{\ll}{{\mathbb L}}
\newcommand{\qq}{{\mathbb Q}}
\newcommand{\pp}{{\mathbb P}}
\newcommand{\cc}{{\mathbb C}}
\newcommand{\rr}{{\mathbb R}}

\newcommand{\Gm}{{{\mathbb G}_{\mbox{\tiny\rm m}}}}

\newcommand{\eE}{{\mathcal E}}
\newcommand{\iI}{{\mathcal I}}

\newcommand{\lL}{{\mathcal L}}
\newcommand{\ii}{{\mathcal I}}

\newcommand{\sS}{{\mathcal S}}
\newcommand{\sT}{{\mathcal T}}
\newcommand{\pP}{{\mathcal P}}

\newcommand{\oO}{{\mathcal O}}

\newcommand{\sX}{{\mathcal X}}

\newcommand{\mM}{{\mathcal M}}
\newcommand{\hH}{{\mathcal H}}
\newcommand{\yY}{{\mathcal Y}}

\newcommand{\fF}{{\mathcal F}}
\newcommand{\aA}{{\mathcal A}}
\newcommand{\stft}{\mbox{stft}}
\newcommand{\Coh}{\mbox{Coh}}
\newcommand{\Fsch}{\mathscr{F}sch}
\newcommand{\sSch}{\mathscr{S}ch}
\newcommand{\FP}{\mathscr{P}}
\newcommand{\mMF}{\mathscr{MF}}
\newcommand{\PV}{\mathscr{PV}}

\newcommand{\bC}{{\mathbf C}}
\newcommand{\GBSRig}{\text{GBSRig}}
\newcommand{\BSRig}{\text{BSRig}}
\DeclareMathOperator{\MV}{MV}

\DeclareMathOperator{\Def}{Def}
\DeclareMathOperator{\RDef}{RDef}
\DeclareMathOperator{\Var}{Var}
\DeclareMathOperator{\loc}{loc}
\DeclareMathOperator{\val}{val}
\DeclareMathOperator{\id}{id}
\DeclareMathOperator{\Crit}{Crit}

\DeclareMathOperator{\Sch}{Sch}

\DeclareMathOperator{\DT}{DT}
\DeclareMathOperator{\Hilb}{Hilb}
\DeclareMathOperator{\pic}{Pic}
\DeclareMathOperator{\Aut}{Aut}
\DeclareMathOperator{\SP}{Sp}
\DeclareMathOperator{\An}{An}
\DeclareMathOperator{\Isom}{Isom}
\DeclareMathOperator{\Div}{Div}
\DeclareMathOperator{\Sim}{Sim}
\DeclareMathOperator{\LinLoc}{LinLoc}
\DeclareMathOperator{\ConvLoc}{ConvLoc}
\DeclareMathOperator{\SConvLoc}{SConvLoc}
\DeclareMathOperator{\Lin}{Lin}
\DeclareMathOperator{\Conv}{Conv}
\DeclareMathOperator{\SConv}{SConv}
\DeclareMathOperator{\Tor}{Tor}
\DeclareMathOperator{\rig}{rig}
\DeclareMathOperator{\Rig}{Rig}
\DeclareMathOperator{\ind}{ind}

\newcommand{\F}{\ensuremath{\mathscr{F}}}

\newcommand{\C}{\ensuremath{\mathscr{C}}}

\newcommand{\Gr}{\mathop{\rm Gr}\nolimits}

\newcommand{\tr}{\mathop{\rm tr}\nolimits}

\renewcommand{\top}{\mathop{\rm top}}

\newcommand{\ord}{\mathop{\rm ord}\nolimits}

\newcommand{\red}{\mathop{\rm red}\nolimits}

\newcommand{\spec}{\mathop{\rm Spec}\nolimits}

\numberwithin{equation}{subsection}

\title[Moduli space of stable sheaves via non-archimedean geometry]{The moduli space of stable coherent sheaves via non-archimedean geometry}
\author{Yunfeng Jiang}
\address{Department of Mathematics\\ University of Kansas\\ 405 Snow Hall 1460 Jayhawk Blvd\\Lawrence KS 66045 USA} 
\email{y.jiang@ku.edu}

\begin{document}
\sloppy \maketitle
\begin{abstract}
We provide a construction of the moduli space of stable coherent sheaves in the world of  non-archimedean geometry, where we use the notion of Berkovich non-archimedean analytic spaces.   The motivation for our construction is Tony Yue Yu's non-archimedean enumerative geometry in Gromov-Witten theory.  The construction of the moduli space of stable sheaves using Berkovich analytic spaces will give rise to the non-archimedean version of Donaldson-Thomas invariants. 

In this paper we give the moduli construction over a non-archimedean field $\kk$. We use the machinery of formal schemes, that is, we define and construct the formal moduli stack of (semi)-stable 
coherent sheaves over a discrete valuation ring $R$, and  taking generic fiber we get the  non-archimedean analytic moduli  of semistable coherent sheaves over the fractional non-archimedean field $\kk$.  For a moduli space of stable sheaves of 
an algebraic variety $X$ over an algebraically closed field $\kappa$, the analytification of such a moduli space gives an example of the non-archimedean moduli space.    

We generalize Joyce's $d$-critical scheme structure in \cite{Joyce} or Kiem-Li's virtual critical manifolds in \cite{KL} to the world of formal schemes, and Berkovich non-archimedean analytic spaces.  As an application, we provide a proof for the motivic localization formula for a $d$-critical non-archimedean $\kk$-analytic space using global motive of vanishing cycles  and motivic integration on oriented formal  $d$-critical schemes. This generalizes  Maulik's motivic localization formula for the motivic Donaldson-Thomas invariants.
\end{abstract}

\maketitle

\tableofcontents

\section{Introduction}

\subsection{Structure of the paper}
This paper contains two parts.  The first part is a construction of the moduli space (stack) of (semi)-stable coherent sheaves in the world of non-archimedean geometry.   The motivation for our research is Tony Yue Yu's study of Gromov compactness of the moduli space of stable maps in non-archimedean analytic geometry in \cite{Tony_Yu1}.  The central part of his theory, motivated by the work of Kontsevich-Soibelman, is to define the open version of Gromov-Witten invariants using non-archimedean analytic geometry.   Tony Yu has made several progresses along this direction, see \cite{Tony_Yu2}, \cite{Tony_Yu3}.  

The advantage of T. Yu's theory is that the invariants he defined satisfy the Kontsevich-Soibelman wall crossing formula as in \cite{KS}.  So this theory should give the right invariants, and these invariants can be used to construct mirrors for some log Calabi-Yau geometries, see \cite{Tony_Yu3}. This is parallel to the  Gross-Siebert program in the series of works \cite{GS1}, \cite{GPKS1}, \cite{GPKS2},  \cite{GS}, where their goal is also to construct mirrors using open or punctured version of Gromov-Witten invariants. 
All of these achievements provide deep evidences in mirror symmetry.

One of the important conjectures in enumerative geometry is the Gromov-Witten/Donaldson-Thomas (GW/DT) correspondence by MNOP \cite{MNOP1}, \cite{MNOP2}. The GW/DT-correspondence conjecture holds for any smooth projective threefold. We only restrict to Calabi-Yau threefolds.   Roughly speaking, for a projective Calabi-Yau threefold $Y$, GW/DT-correspondence states that the Gromov-Witten partition function of the curve counting invariants of $Y$ via stable maps is equivalent to the Donaldson-Thomas (in \cite{Thomas}) partition function of the curve counting invariants  via ideal sheaves after a change of variables.  
T. Yu's theory \cite{Tony_Yu1} is the Gromov-Witten like invariants on non-archimedean analytic spaces. 
So there should exist a theory of the sheaf counting in  non-archimedean analytic  geometry.  In this paper we provide the first step for the sheaf counting in  the non-archimedean sense, i.e., we construct the moduli space (stack) of the (semi)-stable coherent sheaves for non-archimedean analytic spaces.  We work in the category of formal schemes and Berkovich non-archimedean analytic spaces for this construction.  Of course it is ambitious at the moment to see if the counting sheaf theory provides better sights on the construction of mirrors than Gromov-Witten like invariants. 

We define the formal moduli functor of the semistable sheaves over a {\em stft} formal scheme, and prove that the functor is represented by an algebraic stack using Artin's criterions for the representability of algebraic stacks in \cite{Artin}. 

The Berkovich non-archimedean analytic space is natural to the study of degenerations of algebraic varieties.  Over filed of character zero,  every Berkovich analytic space  $X$ has a simple normal crossing (SNC) formal model $\XX$ over $R$,  which is a {\em stft} formal scheme over a discrete valuation ring $R$ such that its special fiber $\XX_s$ is a $\kappa$-scheme and  has only simple normal crossing divisors and the generic fiber $\XX_\eta\cong X$. The  non-archimedean analytic space  $X$ is independent to the formal model we choose, i.e., if we have a different formal model $\XX^\prime$, then the generic fiber $\XX^\prime_{\eta}$ is also isomorphic to $X$.  For instance, elliptic curves can be degenerated into a circle of projective lines, called the skeleton of the degeneration.  One can blow-up the special fiber so that there are some branches on the circle, but the skeleton does not change.   Using this idea we construct the universal stack $\MM$ of SNC formal models for a fixed pair $(X,\XX)$; and the formal moduli stack of stable ideal sheaves of $\MM$. 

The Berkovich non-archimedean analytic spaces can be used to study the degeneration of stable coherent sheaves for the variety $Y$ with simple normal crossing divisors.   
In \cite{Li1}, \cite{LW}, the moduli space of relative stable coherent sheaves for a pair $(Y,D)$ is defined by the so called ``expanded degenerations" and the central idea is to solve the normality property of the stable coherent sheaf with the divisor $D$.  Geometrically this means that the underlying curve associated with the stable sheaf  intersects with the divisor $D$ transversally.   This normality property can also be studied by the techniques of logarithmic  structures as in \cite{AC}, \cite{GS2}.   Using the SNC formal model of the non-archimedean analytic space $X$,   the normality property of stable coherent sheaves on a $\kappa$-scheme $\XX_s$ with respect to a SNC divisor $D$ is automatically satisfied, since the  non-archimedean analytic space $X$ does not depends on the SNC formal model and one can do admissible formal blow-ups along the special fiber to keep the transversality property.  We prove that the formal  moduli stack of stable ideal sheaves of 
$\MM$ is the formal completion of the moduli  stack of ideal sheaves of the stack of expanded degenerations. 
It is very interesting to see if one can use non-archimedean analytic spaces to study the degeneration formula for both 
Gromov-Witten and Donaldson-Thomas invariants in \cite{Li2}, \cite{LW}.

The second part is motivic Donaldson-Thomas theory. We outline the basic materials and questions here, and more details can be found in \S  \ref{sec_motivic_DT_introduction}    and in the sections of Part II. 

Let $Y$ be a smooth Calabi-Yau threefold.  The moduli space $X$ of stable coherent sheaves over $Y$ with fixed topological invariants admits a symmetric obstruction theory \cite{Behrend} and the Donaldson-Thomas invariant of $Y$ is the weighted Euler characteristic of $X$ weighted by the Behrend function $\nu_{X}$ \cite[Theorem 4.18]{Behrend}, which coincides with the Donaldson-Thomas invariant of Thomas \cite{Thomas} by using virtual fundamental classes. 
This proves that Donaldson-Thomas invariants are motivic invariants. 

The general construction of motivic Donaldson-Thomas invariants are given by Kontsevich-Soibelman  \cite{KS} for any Calabi-Yau category, and in degree zero by Behrend-Bryan-Szendroi \cite{BBS}.  The key part in the motivic Donaldson-Thomas theory is to construct a global motive, or a global vanishing cycle sheaf  for $X$ such that taking Euler characteristic of the cohomology of such a sheaf we get the weighted Euler characteristic of $X$.  In a series of papers \cite{BBJ}, \cite{BBBJ}, \cite{BJM}, Joyce etc achieved this goal by using the symplectic derived schemes or stacks, since the moduli space $X$ can be extended naturally to a $(-1)$-shifted symplectic derived scheme $\mathbf{X}$ in \cite{PTVV}.  

The underlying scheme $X$ of a $(-1)$-shifted symplectic derived scheme $\mathbf{X}$ is a $d$-critical locus in the sense of \cite{Joyce} or a virtual critical manifold in the sense of Kiem-Li in \cite{KL} .  In this paper we generalize Joyce's definition of $d$-critical schemes to the setting of formal schemes and Berkovich  non-archimedean analytic spaces.  This at least gives a formal and analytic version of Joyce's $d$-critical scheme structures, and will have applications in motivic Donaldson-Thomas theory.   It is hoped that the formal and analytic $d$-critical schemes or non-archimedean analytic spaces will have more applications.    Similar to the case of $d$-critical schemes and  $(-1)$-shifted symplectic derived schemes, we hope that the $d$-critical formal schemes and $d$-critical non-archimedean $\kk$-analytic spaces are the underlying schemes (spaces) of the corresponding 
$(-1)$-shifted symplectic derived formal schemes and derived non-archimedean analytic spaces, see the corresponding results in this direction \cite[Chapter 8]{Lurie}, \cite{Porta_Yu1},\cite{Porta_Yu2}.  

We also generalize Joyce's orientation of $d$-critical schemes to $d$-critical formal schemes and $d$-critical non-archimedean analytic spaces. This provide a global motive $\mMF_{X,s}^{\phi}$ of vanishing cycles for the  $d$-critical non-archimedean analytic spaces and $d$-critical formal schemes.    This global motive $\mMF_{X,s}^{\phi}$ lies in the localized 
Grothendieck ring $\overline{\mM}_{\kappa}$ of varieties over $\kappa$. 
If the $d$-critical non-archimedean analytic space $(X,s)$ admits a good $\Gm$-action, which is circle-compact, we generalize Maulik's motivic localization formula for the global motive $\mMF_{X,s}^{\phi}$ of $(X,s)$ by using motivic integration for formal schemes in \cite{Nicaise}, \cite{Jiang4}, see Theorem \ref{intro_thm_Maulik}. 

\subsection{Main results of the construction of moduli spaces}
We list our main results in this section on the moduli construction. 

\subsubsection{Main results}

Let $Y$ be a projective scheme over $\kappa$ with a polarization $\oO_Y(1)$.  Let us fix a Hilbert polynomial $P$.  We first recall the construction theorem of moduli space of (semi)-stable sheaves in \cite[Theorem 4.3.4]{HL}.

\begin{thm}
There is a projective scheme $M(P)$ which universally represents the moduli functor $\mM(P)$.  Closed points in 
$M(P)$ are in bijection with equivalence classes of semistable sheaves with Hilbert polynomial $P$.  Moreover there is an open subset $M^s(P)$ that  universally represents the moduli functor $\mathcal{M}^s(P)$ of family of stable sheaves.
\end{thm}

We prove a similar theorem in non-archimedean analytic geometry. 
For the consideration of Donaldson-Thomas theory, we only restrict to three dimensional smooth non-archimedean $\kk$-analytic spaces, although the main result below is true for any smooth non-archimedean $\kk$-analytic space. 
Let  $\iI^{P}_{\kk}(X)$  denote the moduli space  of analytic ideal sheaves of curves on $X$ with Hilbert polynomial $P$. 
The main result is:

\begin{thm}\label{main_thm1}(Corollary \ref{cor_moduli_space_X_proper})
Let $X$ be a smooth non-archimedean $\kk$-analytic space of dimension three.  Suppose that $\hat{L}$ is a K\"ahler structure on $X$ with respect to an SNC formal model $\XX$ of $X$. Let  $\iI^{P}_{\kk}(X)$  denote the moduli stack  of analytic ideal sheaves on $X$ with Hilbert polynomial $P$ such the degree of $P$ with respect to $\hat{L}$ is bounded. Then 
$\iI^{P}_{\kk}(X)$ is a compact $\kk$-analytic space.  Let $\kappa$ has character zero.  If the $\kk$-analytic space $X$ is proper,  then $\iI^{P}_{\kk}(X)$ is a proper $\kk$-analytic stack. 
\end{thm}
\begin{rmk}
Here for  a smooth  non-archimedean $\kk$-analytic space $X$, ``proper" means ``compact" and ``no boundary".
\end{rmk}

Our strategy to prove  Theorem \ref{main_thm1} is through the formal model $\XX$ of $X$.  So it is routine to construct a formal version of the above result.  Let $\mM(P)$ denote the general moduli functor of semistable coherent sheaves with Hilbert polynomial $P$. 

Let $\sS$ be a locally noetherian base scheme, and $\sX/\sS$ a scheme locally of finite presentation over $\sS$. 

\begin{thm}\label{main_thm2}(Theorem \ref{moduli_locally_noetherian_scheme})(Construction of the moduli stack of semistable coherent sheaves over locally noetherian scheme)
The moduli space $\mM_{\sX/\sS}(P)$ of semistable  sheaves $\F$ over $\sX$ with fixed Hilbert polynomial $P$ exists and is an algebraic stack. The coarse moduli space $M_{\sX/\sS}(P)$  is a scheme. 
\end{thm}

Our result for the formal moduli space is:
\begin{thm}\label{main_thm3}(Theorem \ref{thm_formal_model_up_to_etale_covering})(Construction of the moduli stack of formal semistable sheaves over $R$)
Let $\XX$ be a {\em stft} formal $R$-scheme. Let $T$ be a strictly $\kk$-affinoid space and let 
$$(\F\to T, f)$$
be a family of $\kk$-analytic semistable coherent sheaves on $\XX_\eta$ over $T$. Then up to passing to a quasi \'etale covering of $T$,  there exists a formal model $\TT$ of $T$ and a family of formal semistable sheaves
$$(\F\to \TT, \hat{f})$$
of $\XX$ over $\TT$ such that when applying the generic fiber functor, we get the family $(\F\to T, f)$ back. 
\end{thm}

Here is a result about the moduli space of stable coherent sheaves on a degeneration family.  Assume that the character of $\kappa$  is zero. 
Let $X$ be a quasicompact  non-archimedean $\kk$-analytic space. 
 A simple normal crossing (SNC) formal model of $X$ is a {\em stft} formal scheme $\XX$, such that its special fiber $\XX_s$ has only simple normal crossing divisors, and its generic fiber $\XX_\eta\cong X$.  Let us fix such a  pair $(X, \XX)$. 
 
 Let $\{D_i: i\in I_{\XX}\}$ be the irreducible components of $\XX_s$, and let $D_{I}:=\cap_{j\in I}D_{j}$ for any $I\subset I_\XX$.  The non-archimedean $\kk$-analytic space $X$ is independent to the SNC formal model $\XX$ we choose, i.e., if $\XX^\prime\to\XX$ is an admissible formal blow-up along some $D_I$, then $\XX^\prime\cong \XX_\eta\cong X$.  Using admissible formal blow-ups, we construct a universal stack $\MM$ of SNC formal models of $(X,\XX)$.  The stack  $\MM$ is a stack over $\spf(R)$ with generic fiber $X$. 
We also define the moduli stack $\iI_{R}^{P}(\MM)$ of stable ideal sheaves with Hilbert polynomial $P$ to the universal stack $\MM$, see \S \ref{sec_stack_SNC_formal_model}. 

Let us fix to a simple case.  Let $X$ be a $\kappa$-scheme, and $\pi: W\to\aaa^1_{\kappa}$ be  degeneration family of 
$X$ in \cite{LW} such that $W_0=D_1\cup_{D_{12}}D_2$.  Let $\XX\to\spf(R)$ be the formal completion of $W$ along the origin.  Then $\XX$ is a SNC formal model of $X^{\an}$. 
Let $\iI_{R}^{P}(\MM)$ be the moduli stack of admissible ideal sheaves of curves with Hilbert polynomial $P$. 
From \cite[Theorem 4.14]{LW}, let $\iI_{\kappa}^{P}(\MM^{\alg})$ be the moduli stack of stable ideal sheaves over the stack $\MM^{\alg}$ of expanded degenerations with Hilbert polynomial $P$.  Then we have 
\begin{thm}(Theorem \ref{thm_formal_completion_expanded_degeneration})
We have 
$\iI_{R}^{P}(\MM)=\widehat{\iI^{P}_{\kappa}(\MM^{\alg})}$, the formal completion along the origin $0\in\aaa^1_{\kappa}$. 
\end{thm}

Let $\iI_{\kappa}^{P_1}(D_i, D_{12})$ be the moduli stack of relative stable ideal sheaves in \cite{LW} to the stack of relative expanded pairs. 
For a decomposition $\gamma=(P_{1}, P_2, P_{12})$ of the Hilbert polynomial $P$, we have a gluing result.   See \S \ref{sec_stack_expanded_degeneration} for more details.

\begin{thm}(\cite[Theorem 5.28]{LW})\label{thm_moduli_space_absolute_relative}
Let $X$ be a proper $\kappa$-scheme and $\XX$ its $t$-adic formal completion such that $\XX_s=D_1\cup_{D_{12}}D_2$.  Then the moduli stack 
$\iI^{P}_{\kappa}(\MM_s)$, after applying the special fiber functor, has a canonical gluing isomorphism
$$\iI_{\kappa}^{P}(\MM_s)\stackrel{\sim}{\rightarrow}\iI_{\kappa}^{P_1}(D_1, D_{12})\times _{\iI_{\kappa}^{P_{12}}(D_{12})}\iI_{\kappa}^{P_1}(D_1, D_{12})$$
of Deligne-Mumford stacks.
\end{thm}

In character zero, every smooth non-archimedean $\kk$-space $X$ has a  SNC formal model $\XX$.  If $\XX$ is proper, then $\XX_s$ is proper, and we prove that the moduli stack $\iI_{\kappa}^{P}(\XX_s)$ is a proper stack, see Proposition \ref{prop_moduli_central_fiber_proper}.  Then the moduli space  $\iI^{P}_{\kk}(X)$   is proper since the moduli stack $\iI^{P}_{\kappa}(\XX_s)$ is proper. 

\subsubsection{Outline of the proof of the main results}
In \S \ref{formal_Berkovich_space} we review the basic materials for the formal schemes and Berkovich analytic spaces.  The moduli stack of semistable sheaves over a locally noetherian scheme is constructed in \S \ref{sec_construction_noetherian_scheme}. This proves Theorem \ref{main_thm2}.  We construct the non-archimedean analytic moduli stack of semistable sheaves over $\kk$ in \S \ref{sec_non-archimedean_moduli_stack} and prove the main results in Theorem \ref{main_thm3}.

\subsection{Main results on motivic Donaldson-Thomas invariants}

In this section we apply the construction of formal and non-archimedean moduli space of stable sheaves to the motivic localization formula of motivic Donaldson-Thomas invariants. 

Let $(X,s)$ be an oriented $d$-critical non-archimedean $\kk$-analytic space.  Let $(\XX,s)$ be a SNC formal model of $X$, then $(\XX,s)$ is an oriented $d$-critical formal $R$-scheme.   We define the absolute motive
$$\mMF_{X,s}^{\phi}=\int_{\XX_s}\mMF_{\XX,s}, $$ where 
$\int_{\XX_s}$ means pushforward to a point.  
The motive $\mMF_{X,s}^{\phi}$ is independent to the choice of the SNC formal model. 
If $X$ admits a  good, circle-compact action of $\Gm$, on each fixed strata $X_i^{\Gm}$, Maulik defined the virtual index of $X_i^{\Gm}$  as 
\begin{equation}\label{intro_index_formula}
\ind^{\virt}(X_i^{\Gm},X)=\dim_{\kk}(T_{x}(X)_+)-\dim_{\kk}(T_{x}(X)_-)
\end{equation}
where $T_{x}(X)_+$ and $T_{x}(X)_-$ are the weight positive and negative parts of the $\Gm$-action on the tangent space 
$T_x(X)$ for a generic point $x$ in the strata $X_i^{\Gm}$. 
Then we have the following result of the motivic localization formula. 

\begin{thm}\label{intro_thm_Maulik}(Theorem \ref{thm_Maulik})
Let $(X,s)$ be a $d$-critical non-archimedean $\kk$-analytic space and $\mu$ is a good, circle-compact action of $\Gm$ on $X$, which preserves the orientation $K_{X,s}^{\frac{1}{2}}$. Then on each fixed strata $X_i^{\Gm}$, there is an oriented  
$d$-critical $\kk$-analytic space structure $(X_i^{\Gm},s_i^{\Gm})$, hence a global motive $\mMF_{X_i^{\Gm},s_i^{\Gm}}^{\phi}$. Moreover we have the following motivic localization formula
$$\mMF_{X,s}^{\phi}=\sum_{i\in J}\ll^{-\ind^{\virt}(X_i^{\Gm},X)/2}\odot \mMF_{X_i^{\Gm},s_i^{\Gm}}^{\phi}\in \overline{\mM}_{\kappa}^{\hat{\mu}}.$$
\end{thm}

We use the techniques of motivic integration for formal schemes developed in \cite{Nicaise} to prove Theorem \ref{intro_thm_Maulik}. If $(\XX, s)$ is the  oriented 
$d$-critical formal scheme corresponding to the moduli space $X$ of stable coherent sheaves over Calabi-Yau threefold $Y$, and denote by $\mMF_{\XX}^{\phi}$ the global motive on $\XX$.  Suppose that there is a $\mathbb{G}_m$ action on the scheme $\XX$, which is {\em good} and circle compact. In this case,  we get the  motivic localization formula  of  Maulik in \cite{Maulik}.  

\begin{cor}\label{intro_prop_localization_motivic_Donaldson-Thomas_Invariants}(Proposition \ref{prop_localization_motivic_Donaldson-Thomas_Invariants})
Let $Y$ be a smooth Calabi-Yau threefold over $\kappa$ of character zero, and $X=M_n(Y,\beta)$ the moduli scheme of stable coherent sheaves in $\Coh(Y)$ with topological data $(1,0,\beta,n)$.    Then from \cite{BJM}, if there exists an orientation $K_{X,s}^{\frac{1}{2}}$,  there exists  a unique global motive $\mMF_{X}^{\phi}\in \overline{\mM}_{X}^{\hat{\mu}}$.  Moreover if $X$ admits a good circle-compact $\Gm$-action which preserves the orientation $K_{X,s}^{\frac{1}{2}}$, then
$$\int_{X}\mMF_{X,s}^{\phi}=\sum_{i\in J}\ll^{-\ind^{\virt}(X_i^{\Gm}, X)/2}\odot \int_{X_i}\mMF_{X_i^{\Gm}, s_i^{\Gm}}^{\phi}$$
where $X^{\Gm}=\bigsqcup_{i\in J}X_i^{\Gm}$ is the fixed locus of $X$ under the $\Gm$-action.  
The notation $\int_{X}\mMF_{X,s}^{\phi}$ means pushforward to a point, i.e., the absolute motive, and $\ind^{\virt}(X_i^{\Gm}, X)$ is the virtual index on tangent space similar as in (\ref{intro_index_formula}). 
\end{cor}

Our method to prove Corollary \ref{intro_prop_localization_motivic_Donaldson-Thomas_Invariants} is to use formal schemes.  We take the $t$-adic formal completion 
$\XX=\widehat{X}$ of $X$ and its generic fiber $\XX_\eta$ have a formal $d$-critical scheme structure $(\XX,s)$ and 
a $d$-critical non-archimedean analytic space structure $(\XX_\eta,s)$.  
The canonical line bundle 
$K_{\XX,s}$ is isomorphic to the formal completion of the canonical line bundle 
$K_{X,s}$ where $(X,s)$ is the $d$-critical scheme in \cite{Joyce}.  Moreover, if there exists an orientation $K_{X,s}^{\frac{1}{2}}$, then   $K_{\XX,s}^{\frac{1}{2}}$ exists and there exists a unique 
$\mMF_{\XX,s}^{\phi}\in\overline{\mM}_{X}^{\hat{\mu}}$ such that if $\XX$ admits a good circle-compact $\Gm$-action which preserves the orientation $K_{\XX,s}^{\frac{1}{2}}$, then
$$\int_{\XX_s}\mMF_{\XX,s}^{\phi}=\sum_{i\in J}\ll^{-\ind^{\virt}(\XX_i^{\Gm}, \XX)/2}\odot \int_{(\XX_i^{\Gm})_s}\mMF_{\XX_i^{\Gm}, s_i^{\Gm}}^{\phi}$$
where $\XX^{\Gm}=\bigsqcup_{i\in J}\XX_i^{\Gm}$ is the fixed locus of $\XX$ under the $\Gm$-action. 
Since $\XX_s=X$, $(\XX_i^{\Gm})_s=X^{\Gm}_{i}$, and 
$\ind^{\virt}(\XX_i^{\Gm}, \XX)=\ind^{\virt}(X_i^{\Gm}, X)$.  Thus we get the localization formula in Corollary \ref{intro_prop_localization_motivic_Donaldson-Thomas_Invariants}.

This motivic localization formula in Corollary \ref{intro_prop_localization_motivic_Donaldson-Thomas_Invariants} is very useful in the calculations of refined Donaldson-Thomas invariants, for instance, in \cite{CKK}, the authors use this formula to calculate the refined Donaldson-Thomas invariants for local $\pp^2$. 
We expect more interesting applications about this formula.

\subsection{Related and  future works}

As we mentioned earlier, this work is motivated by Tony Yue Yu's study of non-archimedean enumerative geometry in \cite{Tony_Yu1}, \cite{Tony_Yu2}.  
Motivated by the MNOP conjecture equating the Gromov-Witten and Donaldson-Thomas invariants,  it is interesting to construct a sheaf counting theory  of Tony Yu's non-archimedean curve counting theory  by using Behrend's weighted Euler characteristic  of the moduli space of stable sheaves over smooth Calabi-Yau threefolds.  

The moduli stack $\mM_{g,n}(Y)$ of stable maps to a Calabi-Yau threefold $Y$  admits a perfect obstruction theory of Li-Tian \cite{LT}, and Behrend-Fantechi \cite{BF}.  Hence there is a dimensional zero virtual fundamental cycle 
$[\mM_{g,n}(Y)]^{\virt}$.  In the non-archimedean counting of stable maps as in \cite{Tony_Yu1}, \cite{Tony_Yu2} (the paper \cite{Tony_Yu2}  works on log Calabi-Yau surfaces), the author didn't use perfect obstruction theory to directly define his invariants $N_{L,\beta}$ on the corresponding moduli space $\mM_{L,\beta}$ for a tropical spine $L$ in the Berkovich retraction $Y^{\an}\to B$ and curve degree $\beta$,  instead he uses a restriction of the usual virtual fundamental class of a higher dimensional stable map spaces to the moduli space $\mM_{L,\beta}$.  In the forthcoming work of  \cite{Porta_Yu2}, Porta and Yu will construct a perfect obstruction theory on the  non-archimedean moduli  analytic space and define the virtual fundamental cycle directly.  In \cite{Jiang5}, we will address such a problem of symmetric obstruction theory on the non-archimedean moduli space of stable coherent sheaves.   The notion of $d$-critical formal schemes and $d$-critical non-archimedean analytic spaces will take an important role. We will study the Kashiwara-Schapira index theorem for 
non-archimedean analytic spaces, and a non-archimedean version of Behrend's theorem equating the Donaldson-Thomas invariants with the weighted Euler characteristic of a canonical constructible sheaf of vanishing cycles. 

\subsection{Convention}
 Let us fix some notations.  Throughout this paper,  $R$ will be a  complete discrete valuation ring $R=\kappa[\![t]\!]$, with quotient field $\kk:=\cc(\!(t)\!)$, and perfect residue field $\kappa$.  We fix a uniformizing element $t$ in $R$, i.e., a generator of the maximal ideal.  The field $\kk$ is a non-archimedean field with valuation $v$ such that 
$v(t)=1$. The absolute value $|\cdot|=e^{-v(\cdot)}$.    All formal schemes over $R$ is {\em stft} (separate and topologically of finite type)  in sense of \cite{NS}, and the non-archimedean analytic spaces over $\kk$ are quasi-compact Berkovich analytic spaces \cite{Ber1}. 
For the applications in \S \ref{sec_moduli_absolute_relative_sheaves} and \S \ref{sec_motivic_DT_invariants}, we consider the schemes and stacks over $\kappa=\cc$, the field of complex numbers. 

For any $\kappa$-variety $X$, we denote by $\Div(X)$ the abelian group of divisors of $X$, and 
$N^1(X)=\Div(X)/\Div^0(X)$ the divisor class group, where $\Div^0(X)$ is the group of principal divisors.  
An element $D\in N^1(X)$ is said to be {\em nef} if  the intersection $D\cdot C\geq 0$ for any curve $C$ in $X$; {\em ample}
if $D$ is a ample divisor. 

For the complete discrete valuation ring $R$, $R\{x_1,\cdots,x_n\}$ is denoted by the Tate algebra, which is the ring of convergent formal power series over $R$. 
For a strictly affinoid algebra $\aA$ over the non-archimedean $\kk$,  we use $\SP(\aA)$ to represent the affine rigid variety; and  $\SP_{B}(\aA)$ the Berkovich spectrum, i.e., the space of semi-norms of $\aA$ endowed with Berkovich real topology. 
Denote by $\Fsch_R$ the category of {\em stft} formal schemes over $R$.   Also denote by $\An_\kk$ the category of strictly $\kk$-analytic spaces, equipped  with the quasi-\'etale topology, and $\Rig_\kk$ the category of $\kk$-rigid varieties, equipped  with the Grothendieck topology.

We use Fraktur symbols $\XX$ to represent $R$-formal schemes.  We use general symbols  $X$ to represent both $\kappa$-varieties or schemes and Berkovich non-archimedean $\kk$-analytic spaces. 
In \S \ref{sec_construction_noetherian_scheme} the Calligraphic symbols $\sX$ represent the schemes over a locally noetherian scheme $\sS$. 
For a Berkovich analytic space $X$,  we use  $\chi(X)$ to represent the Euler characteristics the \'etale cohomology of $X$. 
We use $\ll$ to represent the Lefschetz motive $[\aaa^1_{\kappa}]$. 

In Part I, we use $\mM(P)$ to represent the moduli functor of (semi)-stable coherent sheaves on algebraic schemes, formal schemes and non-archimedean analytic spaces.  While in Part II,  $\mM$ and  $\mM^{\hat{\mu}}$   represent the 
localized Grothendieck ring and the equivariant localized Grothendieck ring.
\subsection*{Acknowledgments}

Y. J. would like to thank Yifeng Liu, Johannes Nicaise,  Song Sun, and Tony Yue Yu  for valuable discussions on Berkovich analytic spaces,  especially Johannes Nicaise for answering questions about the motivic integration of formal schemes in \cite{Nicaise}, Yifeng Liu for the hospitality when visiting Northwestern University, and Tony Yue Yu for the correspondence on his non-archimedean enumerative geometry in Gromov-Witten theory and valuable suggestions on the generalized version of the motivic localization formula to non-archimedean analytic spaces.   
Y. J. thanks Junwu Tu for the discussion of deformation quantization, and its relation to perverse sheaf of vanishing cycles and twisted de Rham complex; and Qile Chen for the discussion of log stable maps to generalized  Deligne-Faltings pairs. 

Many thanks to Professors Tom Coates, Alessio Corti and R. Thomas for the support in Imperial College London, where the author started to think about the research along this direction, and is inspired by joint work \cite{JT} with R. Thomas on the Behrend function and Lagrangian intersections.  Y. J. also thanks Professor Jun Li and Professor Dominic Joyce for the discussion of d-critical schemes and virtual critical manifolds, and Professor Sheldon Katz for the discussion about the motivic localization formula on motivic Donaldson-Thomas invariants when visiting UIUC in January 2017. 
Y. J. thanks Professors Jim Bryan, Andrei Okounkov and Balazs Szendroi for email correspondence on the index formula of plane partitions, and especially thanks to Balazs Szendroi for pointing an error in the last example in an earlier version of the paper. 
This work is partially supported by  NSF DMS-1600997.

\section*{Part I:}


\section{Preliminaries on formal schemes and Berkovich analytic spaces}\label{formal_Berkovich_space}

\subsection{Formal schemes, rigid varieties and Berkovich analytic spaces}

An adic $R$-algebra $A$ is  said to be {\em topologically finitely generated over} $R$ if $A$ is  topologically $R$-isomorphic to a quotient algebra of the algebra of restricted power series $R\{x_1,\cdots, x_n\}$.
The algebra $R\{x_1,\cdots, x_n\}$ is the Tate algebra which is the subalgebra of $R[\![x_1,\cdots, x_n]\!]$ consisting of the elements
$$\sum_{(i_1,\cdots, i_m)\in\nn^m}\left(c_i \prod_{j=1}^{m}x_{j}^{i_j}\right)$$
such that $c_i\to 0$ (with respect to the $t$-adic topology on $\kk$) as $|i|=i_1+\cdots +i_m$ tens to $\infty$. 
Let $J=tR$ be the ideal of definition of $R$, then from \cite[Ch. 0. 7.5.3]{EGA1}, the quotient algebra has an ideal of definition $JA$.
An {\em stft} affine formal scheme is the formal spectrum $\spf(A)$ for a topologically finitely generated $R$-algebra $A$.
 
In general, an {\em stft} formal $R$-scheme $\XX$ is a separated formal scheme, topologically of finite type over $R$, which is covered by a cover $\{\XX_i\}$ of {\em stft} affine formal subschemes of the formal $\XX_i=\spf(A_i)$ for $A_i$  topologically finitely generated over $R$.  We denote its special fiber by $\XX_s$, and its generic fiber by $\XX_\eta$.  
$\XX_s$ is a $\kappa$-scheme, which in the affine case $\XX=\spf(A)$, $\XX_s=\spec(A/(J))$. 
In general $\XX_s$ is covered by affine $\kappa$-schemes $\spec(A_i/(J))$. 
The generic fiber 
$\XX_\eta$ is a quasi-compact Berkovich  non-archimedean $\kk$-analytic space in the sense of \cite{Ber1}.  In the category of rigid varieties as in \cite{BGR}, $\XX_\eta$ is a separate quasi-compact rigid $\kk$-variety.  In the case that $\XX=\spf(A)$, in notations we have 
$\XX_\eta=\SP_{B}(\aA)$ for $\aA=A\otimes_{R}\kk$ in the sense of Berkovich, and $\XX_\eta=\SP(\aA)$ in the category of rigid varieties.  Like the difference between varieties and schemes, the Berkovich spectrum will add generic points and has a real topology, and rigid varieties have Grothendieck topologies.  If there is no confusion, we will mixed-use these two notions. 
For instance, if $\XX=\spf(R\{x_1,\cdots,x_n\})$, then 
the generic fiber is the closed unit disc
$D^n(0,1)$ in the  affine space 
$$\aaa^n=\SP_{B}(\kk[x_1,\cdots, x_n])=\bigcup_{r\geq 0} D^n(0,r)$$
where $D^n(0,r)=\SP_B(\kk\{r_1^{-1}x_1, \cdots, r_n^{-1}x_n\})$ for $r=(r_1,\cdots, r_n)$. 
From Berkovich's classification theorem, there are four type of Berkovich points in $\aaa^n$ and each point is the limit of a sequence of points 
$\Vert\cdot\Vert_{D_n}$ corresponding to a nested sequence $D_1\supset D_2\supset\cdots$
of balls of positive radius.

We fix a locally finite covering $\{\XX_i\}_{i\in I}$ of $\XX$, where $\XX_i$ are affine formal subschemes of the form $\spf(A_i)$ and $A_i$ is a topologically finitely generated $R$-algebra.  Then for any $i,j\in I$ the intersection $\XX_{ij}=\XX_i\cap\XX_j$ is also an affine subscheme of the same form.  The generic fiber $\XX_{ij,\eta}$ is a closed analytic domain in $\XX_{i,\eta}$, and the canonical morphism $\XX_{ij,\eta}\to\XX_{i,\eta}\times\XX_{j,\eta}$ is a closed immersion.  From \cite{Ber1}, we can glue $\XX_{i,\eta}$ to get an analytic space $\XX_{\eta}$.

\subsection{The specialization map} 
The special fiber $\mathfrak{X}_s$ is a $R/(J)=\kappa$-scheme. In the affine case $\XX=\spf(A)$, 
$\XX_s=\spec(\tilde{A})$, where $\tilde{A}=A/JA$. The specialization  map  as in \cite[\S 2.2]{NS}
$$sp: \XX_{\eta}\to\XX_{s}$$
sends the points in the generic fibre $\XX_\eta$ to the special fibre $\XX_s$.
Let $x\in \XX_{\eta}$ be a point, which corresponds to a semi-norm 
$$x: \aA\to \rr_{\geq 0}.$$
Let $\hH(x)=\aA/\wp_{x}$, where $\wp_x$ is the kernel of $x$. 
Then the point $x$ gives a character map 
$$\tilde{\chi}_{x}: \tilde{A}\to \tilde{\hH(x)}=\hH(x)/\aA^{\circ\circ},$$ 
where
$\aA^{\circ\circ}=\{f\in\aA| |f(x)|< 1 ~\mbox{for all}~x\in\SP_B(\aA)\}$  is the maximal ideal of 
$\aA$.  Then the kernel of the map $\tilde{\chi}_x$ is defined to be the image of $x$ under 
$sp$. 

Let $\yY\subset \XX_s$ be a closed subset, which is given by an ideal $(\tilde{f}_1,\cdots,\tilde{f}_n)$ for 
$f_i\in A$. Then 
$sp^{-1}(\yY)=\{x\in \XX_\eta| |f_i(x)|<1, 1\leq i\leq n\}$ is open in $\SP_{B}(\aA)=\XX_\eta$.
This correspondence means that under the reduction map $\pi$, the preimage of a closed subset is open, and similarly the preimage of an open subset is closed.  This is one of the special properties  for Berkovich analytic spaces.

\subsection{Formal affine critical schemes}\label{sec_affine_formal_critical_scheme}

In this section we talk about the notion of formal affine critical schemes.  
Let $f$ be a nonzero polynomial in $\kappa[T]:=\kappa[T_1,\cdots,T_m]$.  This defines a flat morphism from 
$\aaa_{\kappa}^m=\spec(\kappa[T])\to \spec(\kappa[t])$. Let 
$$\hat{f}: \XX\to \spf(R)$$
be the $t$-adic completion of the morphism $f$, where 
$\XX=\spf (A)$ and $A=R\{T\}/(f-t)$. The algebra  $R\{T\}:=R\{T_1,\cdots,T_m\}$ is the algebra of convergent  power series.
The formal scheme $\XX\to \spf(R)$ is a $\stft$
formal scheme, see \cite{NS}.
The special fiber $\XX_s$ is a $\kappa$-scheme $\spec(A/(t))$, which is canonically  isomorphic to the fiber of 
$f$ over $0$.
The generic fiber $\XX_{\eta}=\SP_{B}(A\otimes_{R}\kk)$ is a Berkovich analytic 
space over the field $\kk$. 

Let $\Crit(f)$ be the critical subscheme of $f$ inside $\aaa_{\kappa}^m$.  We make a assumption that 
$\Crit(f)\subset\XX_s$. Set 
\begin{equation}\label{affine_formal_critical_scheme}
\hat{f}: \hat{\XX}\to \spf(R)
\end{equation}
to be the formal completion  of $\XX$ along $\Crit(f)$.  Then the special fiber $\hat{\XX}_s$  is the subscheme 
$\Crit(f)$.  The generic fiber $\hat{\XX}_\eta=sp^{-1}(\Crit(f))$  is a subanalytic space of $\XX_\eta$. 

\begin{defn}\label{affine_formal_critical_scheme_defn}
Let $f$ be a polynomial function on $\kappa[T_1,\cdots, T_n]$.  We call the formal scheme in (\ref{affine_formal_critical_scheme}) the formal critical scheme associated with $f$. 
\end{defn}

\begin{defn}\label{analyticmilnorfiber}(\cite{NS})
Let $\mathcal{Y}\subset \XX_s$ be a closed subscheme, the analytic Milnor fiber $\mathfrak{F}_{\mathcal{Y}}(f)$ of 
$f$ is defined as 
$$\mathfrak{F}_{\mathcal{Y}}(f)=sp^{-1}(\mathcal{Y})$$

For any $x\in\XX_s$,  $\FF_x:=\mathfrak{F}_{x}(f)$ is called the  analytic Milnor fiber  of $f$ at $x$. 
\end{defn}

\subsection{Sheaf of vanishing cycles}

We recall the vanishing functor for schemes in \cite{SGA7}, which is reviewed in \cite[\S 5]{Ber4}.  Let $S=\spec(R)$ be the spectrum of $R$.  The scheme $S$ consists of the closed point $s=\spec(\kappa)$ and the generic point 
$\eta=\spec(\kk)$.  The field $\kk$ is quasi-complete (\cite[\S 2.4]{Ber1}) and the valuation on $\kk$extends uniquely to the separable closure $\kk^s$, and so the integral closure of $R$ in $\kk^s$ coincides with $R^s$, the ring of integers of $\kk^s$.  Set 
$\overline{S}=\spec(R^s)=\{\overline{s},\overline{\eta}\}$. 
Let $X$ be a scheme over $S$, and let $X_s$and $X_\eta$ (resp. $X_{\overline{s}}$ and $X_{\overline{\eta}}$) be the closed and the generic fibers of $X$ (resp. $\overline{X}=X\times_{S}\overline{S}$).  We have a canonical diagram:
\[
\xymatrix{
X_{\eta}\ar@{^{(}->}[r]^{j}& X & X_s\ar[l]_{i}\\
X_{\overline{\eta}}\ar[u]\ar@{^{(}->}[r]^{\overline{j}}& \overline{X}\ar[u]& X_{\overline{s}}\ar[u]\ar[l]_{\overline{i}}
}
\]
The nearby cycles functor is given by
$\Psi_{\eta}(\fF)=\overline{i}^*(\overline{j}_{*}\overline{\fF})$, where $\overline{\fF}$ is the pullback of a sheaf $\fF$ on $X_\eta$ to $X_{\overline{\eta}}$. The functor $\Psi_{\eta}$ takes values in the category of {\'e}tale sheaves on $X_{\overline{s}}$ that are endowed with a continuous action of $G_{\eta}=G(\kk^s/\kk)$ compatible with the action of 
$G(\kappa^s/\kappa)$ on $X_{\overline{s}}$.

Let $\Fsch$ be the category of $\stft$ formal $R$-schemes.
Let $\XX\in \Fsch$ be a formal $R$-scheme.  For $n\geq 1$, denote the scheme 
$(\XX, \oO_{\XX}/t^n\oO_{\XX})$ by $\XX_n$.  
A morphism of formal schemes over $R$, $\phi: \YY\to\XX$ is said to be \'etale if for all 
$n\geq 1$, the induced morphisms of schemes $\phi_n: \YY_n\to\XX_n$ are \'etale.

Let $\phi: \YY\to \XX$ be a morphism of formal schemes. Then it induces the morphism between the generic and central fibres, i.e.
$\phi_\eta: \YY_\eta\to\XX_\eta$ and $\phi_s: \YY_s\to\XX_s$, where
$\phi_\eta$ is a morphism of Berkovich analytic spaces and $\phi_s$ is a morphism of schemes. 
Here are two known results from  \cite{Ber4} which are needed to construct vanishing cycles.

\begin{lem}(\cite[Lemma 2.1]{Ber4})\label{lem_functor_special_fiber}
The correspondence $\YY\mapsto\YY_s$ gives an equivalence between the category of formal schemes \'etale over $\XX$ and the category of schemes \'etale over $\XX_s$.
\end{lem}

\begin{lem} (\cite[Lemma 2.2]{Ber4})\label{key-lemma1}
Let $\phi: \YY\to \XX$ be an \'etale morphism of formal schemes.
Then  $$\phi_{\eta}(\YY_\eta)=sp^{-1}(\phi_s(\YY_s)).$$
\end{lem}

The \'etale morphim between $\kk$-analytic spaces can be similarly defined, see \cite[\S 2]{Ber4}. 
Denote by $\XX_{\eta_{\et}}$ the \'etale site of $\XX_\eta$, which is the site induced from the Grothendieck topology of all \'etale morphisms of $\kk$-analytic spaces. 
Let  $\XX^{\sim}_{\eta_{\et}}$ be the category of sheaves of sets on the \'etale site $\XX_{\eta_{\et}}$.

For two Berkovich $\kk$-analytic spaces $\XX_\eta$ and $\YY_\eta$. A morphism $\psi: \YY_\eta\to\XX_\eta$ is called ``quasi-\'etale" if for every point $y\in \YY_\eta$ there exist affinoid domains 
$V_{\eta,1}, \cdots, V_{\eta,n}\subset \YY_\eta$ such that the union $V_{\eta,1}\cup\cdots\cup V_{\eta,n}$ is a neighbourhood of $y$ and each $V_{\eta,i}$ is identified with an affinoid domain in a $\kk$-analytic space \'etale over $\XX_\eta$.  

A basic fact from \cite[Proposition 2.3]{Ber4} is that an \'etale morphism $\phi: \YY\to \XX$ of formal schemes induces a quasi-\'etale morphism $\phi_\eta: \YY_\eta\to\XX_\eta$ over the generic fibres.
Denote by $\XX_{\eta_{q\et}}$ the quasi-\'etale site of $\XX_\eta$, which is the site induced from the Grothendieck topology of all quasi-\'etale morphisms of $\kk$-analytic spaces.
Let $\XX^{\sim}_{\eta_{q\et}}$ be the category of sheaves of sets on the quasi-\'etale site $\XX_{\eta_{q\et}}$. 
There exists a natural morphism of sites
\begin{equation}
\mu:  \XX^{\sim}_{\eta_{q\et}}\to \XX^{\sim}_{\eta_{\et}}, 
\end{equation}
which is understood as the pullback. 

Let $\YY_s\mapsto\YY$ be the functor obtained from inverting the functor in Lemma \ref{lem_functor_special_fiber}.
Then from Lemma \ref{key-lemma1} and the fact that \'etale morphisms on formal schemes induce quasi-\'etale morphisms on generic fibres, the composition of the functors $\YY_s\mapsto\YY$ and 
$\YY\mapsto\YY_\eta$ gives a morphism of sites
\begin{equation}
\nu: \XX_{\eta,q\et}\to\XX_{s,\et}.
\end{equation}

Let 
\begin{equation}
\Theta=\nu_*\mu^*: \XX^{\sim}_{\eta_{\et}}\to\XX^{\sim}_{\eta_{q\et}}\to\XX^{\sim}_{s_{\et}}
\end{equation}
be the functor obtained from composition. 
Let $F$ be an \'etale abelian torsion sheaf over $\XX_\eta$.  Let $\XX_{\overline{\eta}}=\XX_\eta\otimes \kk^s$, and $\overline{F}$ the pullback of $F$ to $\XX_{\overline{\eta}}$.
Then define the nearby cycle functor $\Psi_{\eta}$ by 
\begin{defn} The nearby cycle functor is defined as 
$$\Psi_{\eta}(F)=\Theta_{\widehat{\kk^s}}(\overline{F}).$$
The vanishing cycle functor $\Phi_{\eta}$ is defined to be the cone 
$$\cone[F\to\Psi_{\eta}(F)].$$
\end{defn}

Let $x\in \XX_s$ be a point and $\qq_l$ be an \'etale abelian sheaf. Then the stalk 
$$R^{i}\Psi_{\eta}(\qq_l)_x\cong H^i_{\et}(\FF_x,\qq_l)$$ 
is isomorphic to the \'etale cohomology $H^i_{\et}(\FF_x,\qq_l)$ of the analytic Milnor fibre $\FF_x$.  The stalk of the vanishing cycle
$$R^{i}\Phi_{\eta}(\qq_l)_x\cong \tilde{H}^i_{\et}(\FF_x,\qq_l)$$  
is isomorphic to the  reduced \'etale cohomology $\tilde{H}^i_{\et}(\FF_x,\qq_l)$ of the analytic Milnor fibre $\FF_x$.

\subsection{ Berkovich comparison Theorem}\label{sec_Berkovich_comparison_theorem}

 Let $\hat{F}$ be the corresponding \'etale abelian sheaf on $\XX_\eta$.
 
 \begin{prop}(\cite[Theorem 5.1]{Ber4})\label{comparison}
 There exists an isomorphism for an \'etale abelian torsion sheaf $F$ over $X_\eta$:
 $$ i^*(R^qj_*(F))\cong R^q(\Psi_\eta(\hat{F})).$$
 \end{prop}
 
 Let $\hat{F}$ be an \'etale abelian constructible sheaf over the analytic Milnor fibre $\FF_x$, and 
$H^q_{\et}(\FF_x, \hat{F})$ the \'etale cohomology of $\FF_x$.
 
 \begin{prop}(\cite{Jiang2})
 Let $f$ be a regular function in $A=\kappa[x_1,\cdots,x_n]$ and let $F_x$ be the topological Milnor fibre of $f$ at $x$. 
 Suppose that the formal scheme $\XX$ is the $t$-adic completion of the morphism 
 $f: \spec(A)\to\spec(\kappa[t])$. Then 
$$H^i(F_x,\kappa)\cong H^i(\FF_x, \zz_l)\otimes \kappa.$$
This isomorphism is compatible with the monodromy action.
 \end{prop}

 \subsection{Formal model of non-archimedean analytic spaces}
 
 \begin{defn}
 Let $X$ be  a {\em stft} quasi-compact  non-archimedean $\kk$-analytic space.  A {\em formal model} $\XX$ of $X$ consists of a {\em stft} formal scheme $\XX$ over $R$ such that the generic fiber $\XX_\eta$ is isomorphic to $X$. 
 \end{defn}
 
 We recall the simple normal crossing (SNC for short) formal schemes. 
 \begin{defn}
 A {\em stft} formal scheme $\XX$ is SNC if:
 \begin{enumerate}
 \item  Every point of $\XX$ has an open affine neighborhood  $\UU$ such that 
 $\UU\to\spf(R)$ factors through an \'etale morphism
 $$\phi: \UU\to \spf(R\{x_0,\cdots, x_n, x^{-1}_{d+1},\cdots, x_n^{-1}/(x_0^{m_0}\cdots x_{d}^{m_d}-\varpi))$$
 where $(m_0,\cdots, m_d)\in\zz_{>0}^{d+1}$.  We assume that any $m_i$ does not equal to the characteristic of $\kappa$.
 \item All the intersections of the irreducible components of the special fiber $\XX_s$ are either empty or geometrically irreducible. 
 \end{enumerate}
 \end{defn}
 
 Let $X$ be a {\stft} quasi-compact  non-archimedean $\kk$-analytic space and let $\XX$ be a SNC formal model of $X$.  If the characteristic $\mbox{Ch}(\kappa)=0$, then Temkin \cite{Temkin} shows that the SNC formal model $\XX$ always exists by resolution of singularities.  Let 
 $$\{D_i| i\in I_{\XX}\}$$
 denote the set of irreducible components of $\XX_s^{\red}$ with the reduced scheme structure.  For any $I\subset I_{\XX}$, let 
 $D_{I}=\cap_{i\in I}D_i$.  We denote by $m_i$ the multiplicity of $D_i$ in $\XX_s$. 
 
 We recall the Clemens polytope $S_{\XX}$ for the formal scheme $\XX$. 
 \begin{defn}
 The {\em Clemens polytope} $S_{\XX}$ for a SNC formal scheme $\XX$ is the simplicial subcomplex of the simplex 
 $\Delta^{I_{\XX}}$ such that for every non-empty subset $I\subset I_\XX$, the simplex $\Delta^{I}$ is a face of $S_{\XX}$ if and only if $D_{I}$ is non-empty.
 \end{defn}
 
 From \cite{Berkovich5}, there is a deformation retraction map
 $$\tau: \XX_\eta\to S_{\XX}.$$
 According to \cite{Tony_Yu2}, this retraction map corresponds to the Gross fibration and is the non-archimedean version of the SYZ-fibration. 
 
 \begin{defn}
 A {\em simple function} $\varphi$ on the Clemens polytope $S_{\XX}$ is a real valued function that tis affine on every simplicial face of $S_\XX$.  For $i\in I_{\XX}$, let $\varphi(i)$ be the value of $\varphi$ at the vertex $i$.
 
Let $\Div_0(\XX)_{\rr}$ be the vector space of $\rr$-divisors on $\XX$, which is, by definition, the Cartier divisors on $\XX$ supported on $\XX_s$.   $\Div_0(\XX)_{\rr}$ has the dimension of $|I_\XX|$.  The effective divisors $D$ on $\XX$ can be similarly defined and is locally given by a function $u$.  The valuation $\val(u(x))$ defines a continuous function on $\XX_\eta$ which we denote it by $\varphi_{D}^{0}$. Then 
$$D\mapsto \varphi_{D}^{0}$$
extends by linearity to a map
from $\Div_0(\XX)_{\rr}$ to the space of continuous functions $C^{0}(\XX_\eta)$ on $\XX_\eta$. Hence we get a map
\begin{equation}\label{map_tau}
\tau: \XX_\eta\to \Div_0(\XX)_{\rr}^{*}
\end{equation} 
by 
$\langle\tau(x), D\rangle=\varphi_{D}^0(x)$. 
 \end{defn}
 
 \begin{prop}(\cite[Proposition 2.7]{Tony_Yu1})
 The map $\tau$ defined in (\ref{map_tau}) has the following properties:
\begin{enumerate}
\item The image of $\tau$ concides with $S_\XX$;
\item For any $D=\sum_i a_i D_i\in \Div_0(\XX)_{\rr}$, there exists a unique $\varphi_D$ on $S_\XX$ such that 
$\varphi_{D}^{0}=\varphi_D\circ \tau$.
\end{enumerate}
\end{prop}
\begin{rmk}
 The Clemens polytope $S_\XX$ is the tropicalization of the $\kk$-analytic space $X$, and it satisfies some functoriality property as in \cite[Proposition 2.9]{Tony_Yu1}.
\end{rmk}

\section{Construction of moduli of (semi)-stable sheaves over locally noetherian schemes}\label{sec_construction_noetherian_scheme}

Let $\sS$ be a locally noetherian base scheme.  We define the moduli space of semistable sheaves $\mM_{\sX}(P)$
on the scheme $\sX$ locally of finite presentation over $\sS$. 
We construct the moduli stack of semistable sheaves over a general scheme locally of finite presentation over a locally noetherian scheme $\sS$.  This will be used to construct $\kk$-analytic semistable sheaves over a $\kk$-analytic space $X$. 

We fix $\sS$ to be a locally noetherian scheme.  Let $\sSch_{\sS}$ be the category of schemes over $\sS$. 
\begin{defn}
Let $\sX/\sS$ be a projective scheme over $\sS$, and $\oO_{\sX}(1)$ be the Serre line bundle.  Fix a Hilbert polynomial $P\in \qq[z]$, define the functor
$$\mM_{\sS}(P)(\sX): (\sSch_{\sS})\to Sets$$
such that 
$$\{\sT\to \sS\}\mapsto \Big\{ \substack{\text{flat~} \sT\text{-family of semistable sheaves~}\fF\to \sT, \\
\text{on~} \sX~ \text{with Hilbert polynomial~}P} \Big\}/\sim$$
where the equivalent relation $\sim$ is given by:
$$\fF\sim \fF^\prime, \fF, \fF^\prime\in\mM(\sT) \Leftrightarrow \fF\cong \fF^\prime\otimes p^\star(L), L\in \pic(\sT)$$
where $p: \fF\to \sT$ is the family the semistable sheaves. 
\end{defn}

\begin{rmk}
We recall the semistability of coherent sheaves here.  The sheaf $\fF$ over $\sX/\sS$ is (semi)-stable if $\fF$ is pure and any proper subsheaf $\fF^\prime\subset \fF$ implies that $P(\fF^\prime)< (\leq)P(\fF)$.

A morphism of families of semistable sheaves 
$$(\fF\stackrel{p}{\rightarrow} \sT)\to (\fF^\prime\stackrel{q}{\rightarrow} \sT)$$
is given by a commutative diagram:
$$
\xymatrix{
\fF\ar[r]\ar[d]&\fF^\prime\ar[d]\\
\sT\ar[r]^{f}&\sT^\prime
}
$$
such that $\fF=f^\star(\fF^\prime)$. 
\end{rmk}

Then $\mM_{\sS}(P)(\sX)$ is a category of semistable sheaves over $\sX/\sS$ fibered by groupoids over $\sSch_{\sS}$. 

\begin{thm}\label{moduli_locally_noetherian_scheme}
Let $\sX$ be a scheme locally of finite presentation over a locally noetherian scheme $\sS$. Then the functor 
$\mM_{\sS}(P)(\sX)$ is an algebraic stack locally of finite presentation over $\sS$. 
\end{thm}

The proof of Theorem \ref{moduli_locally_noetherian_scheme} is based on checking the conditions $(1), (2), (3), (4)$ in Theorem 5.3 of \cite{Artin}. We list them as lemmas.
First from Tag07WP in the stack project  \cite{Stack_Project}, 
\begin{lem}
Let $\sT_1, \sT_2, \sT, \sT^\prime$ be spectra of local Artinian rings of finite type over $\sS$. Assume that $\sT\to \sT_1$ is a closed immersion and that 
$$
\xymatrix{
\sT\ar[r]\ar[d]& \sT_1\ar[d]\\
\sT_2\ar[r]& \sT^\prime=\sT_1\cup_{\sT}\sT_2
}
$$
is a pushout diagram in $\sSch_{\sS}$, then the functor of fiber categories 
$$\mM_{\sS}(P)(\sT^\prime)\stackrel{\sim}{\rightarrow} \mM_{\sS}(P)(\sT_1)\times_{\mM_{\sS}(P)(\sT)} \mM_{\sS}(P)(\sT_2)$$
is an equivalence of groupoids. 
\end{lem}
\begin{proof}
The pushout property here follows from this property for quasi-coherent sheaves, see Tag 08LQ  and  Tag08IW in   \cite{Stack_Project}. 
\end{proof}

The condition $(2)$ in \cite[Theorem 5.3]{Artin} is the limit preserving property:
\begin{lem}
Let $\hat{A}$ be a complete local algebra over $\sS$, with maximal ideal $\mathfrak{m}$, and the residue field is finite type over $\sS$. Then the canonical map
$$\mM_{\sS}(P)(\hat{A})\to \varprojlim_{l}\mM_{\sS}(P)(\hat{A}/\mathfrak{m}^l)$$
is an equivalence of groupoids. 
\end{lem}
\begin{proof}
Let $\{\fF^{(l)}\stackrel{f^{(l)}}{\rightarrow}\spec \hat{A}/\mathfrak{m}^l\}$ be a formal object on the right hand side, then the Grothendieck Existence Theorem for formal schemes tells us that 
there exists a formal object 
$$\{\hat{\fF}\stackrel{\hat{f}}{\rightarrow}\spec \hat{A}\}$$
on the left hand side. Hence we only need to show that $\{\hat{\fF}\stackrel{\hat{f}}{\rightarrow}\spec \hat{A}\}$ is semistable over $\hat{A}$. This is from the fact that the semistability condition is a closed condition on the base scheme $\sT$ and the full faithfulness of the functor by Grothendieck's existence theorem. 
\end{proof}

All the other conditions in \cite[Theorem 5.3]{Artin} are deformation and obstructions. 

\begin{lem}
Let $A$ be an $\sS$-algebra, and $A\otimes\mathfrak{n}$ the trivial thickening of $A$.  In the scheme level this corresponds to the sheaf of rings $\spec(A\otimes\mathfrak{n})$. Let 
$$\sX=(\fF\to \spec A)\in \mM_{\sS}(P)(\spec A).$$
Then 
\begin{enumerate}
\item  The module of infinitesimal automorphisms is $\Aut_{\sX}(\mathfrak{n})$;
\item  The module of  infinitesimal deformations is  $\Def_{\sX}(\mathfrak{n})=\Ext^1_{\oO_{\sX\times A}}(\eE, \eE\otimes\mathfrak{n})$, where $\eE$ is a sheaf on $\sX\times\spec A$;
\item The module $\mathfrak{o}_{\sX}(\mathfrak{n})$of obstructions is given by $\Ext^2_{\oO_{\sX\times A}}(\eE, \eE\otimes\mathfrak{n})$.
\end{enumerate}
\end{lem}
\begin{proof}
This is the standard results in deformation-obstruction theory of coherent sheaves as in \cite{Thomas}. 
\end{proof}

This lemma verifies Condition $(3)$  and the last part of $(1)$ of Theorem 5.3 of \cite{Artin}. 
We are left to check Condition $(4)$, which is the ``local quasi-separation" property.  Let 
$x:=\{\fF\stackrel{f}{\rightarrow}\sT\}$ be an element in $\mM_{\sS}(P)(\sT)$ and $\phi$ an automorphism of $x$. 
Suppose that $\phi$ induces the identity on $\mM_{\sS}(P)(\sT)$ for a dense set of points $t\in\sT$ of finite type, then $\phi$ is the identity ona dense set of points of finite type on $\fF$. Hence $\phi$ must be the identity on the whole space since $\fF\to\sT$ is flat and separate over $\sT$.  So from \cite[Theorem 5.3]{Artin}, the category $\mM_{\sS}(P)$ is an algebraic stack locally of finite presentation over $\sS$.  This finishes the proof of Theorem \ref{moduli_locally_noetherian_scheme}.


\section{The moduli stack of non-archimedean (semi)-stable sheaves}\label{sec_non-archimedean_moduli_stack}

\subsection{The construction}

We construct the moduli stack of formal semistable coherent sheaves and the moduli stack of non-archimedean analytic semistable sheaves.   We first recall the definition of formal stacks locally of finite type over $R$, and the definition of strictly 
$\kk$-analytic stacks. 

\begin{defn}\label{def_formal_stack}
A {\em formal stack} $\XX$ locally of finite type over $R$ is a stack fibered by groupoids over the site $\Fsch_{R}$, such that the diagonal morphism $\XX\to \XX\times_{R}\XX$ is representable and there exists a formal scheme $\UU$ locally of finite type over $R$ and a smooth effective epimorphism $\UU\to \XX$.   
\end{defn}

\begin{defn}\label{def_analytic_stack}
A {\em strictly $\kk$-analytic  stack} $X$  is a stack fibered by groupoids over the site $\An_{\kk}$, such that the diagonal morphism $X\to X\times_{\kk}X$ is representable and there exists a strictly $\kk$-analytic space  $U$ and a quasi-smooth   effective epimorphism $U\to X$.  We say that a strictly $\kk$-analytic  stack $X$ si {\em compact} we there is a covering $\{U_i\}$ by compact strictly $\kk$-analytic  spaces. 
\end{defn}

\begin{defn}
Let $\XX, \TT  (X, T)$ be {\em stft} formal schemes over $R$  (resp. non-archimedean $\kk$-analytic spaces).  A family of formal (resp. $\kk$-analytic) semistable sheaves  
$$(\fF\stackrel{\hat{f}}{\rightarrow} \TT)\quad (resp. (\F\stackrel{f}{\rightarrow} T))$$
 of $\XX$ $ (resp.  X)$ is a morphism $\hat{f}: \fF\to \TT  \quad (f: \F\to T)$ such that
it is semistable with respect to the Hilbert polynomial $P$. 
\end{defn}
\begin{rmk}
For a fixed Hilbert polynomial $P$ and a non-archimedean analytic space $X$,  we use the \'etale sheaf cohomology $H^i(X,\F)$ in \cite{Ber1} for the Berkovich space $X$ to define $P(\F)$. 
\end{rmk}

We take into account the analytification functor $(\cdot)^{\an}$, the special fiber functor $(\cdot)_s$, and the generic fiber functor $(\cdot)_{\eta}$ for a formal scheme in the definitions above.  We have:

\begin{lem}\label{lem_analytification_functor}
Let $X$ be an algebraic variety over the non-archimedean field $\kk$, and $A$ be a strictly $\kk$-affinoid algebra. Let 
$(F\stackrel{f}{\rightarrow} \spec A)$  be a family of semistable coherent sheaves of $X$ over $\spec A$. The the analytic triple 
$(F^{\an}\stackrel{f^{\an}}{\longrightarrow}\SP_{B}(A) )$ obtained by the relative analytification functor $(\cdot)^{\an}$ is a family of non-archimedean $\kk$-analytic semistable sheaves of $X^{\an}$ over $\SP_{B}(A)$.
\end{lem}
\begin{proof}
This is true since a geometric point of the non-archimedean $\kk$-analytic spectrum $\SP_{B}(A)$ is a geometric point of $\spec A$.
\end{proof}

\begin{lem}\label{lem_special_fiber_functor}
Let $\XX, \TT$ be {\em stft} formal schemes  over $R$.  Let $(\fF\stackrel{\hat{f}}{\longrightarrow} \TT)$ be a family of formal semistable sheaves of $\XX$ over $\TT$. Then applying the special fiber functor $(\cdot)_s$ we get 
a family of semistable sheaves  $(\fF_s\stackrel{f_s}{\longrightarrow} \TT_s)$ of $\XX_s$ over $\TT_s$. 
\end{lem}
\begin{proof}
A geometric point of the scheme $\TT_s$ is in particular a geometric point of the formal scheme $\TT$. 
\end{proof}

\begin{lem}\label{lem_generic_fiber_functor}
Let $\XX, \TT$ be  {\em stft} formal schemes  over $R$.  Let $(\fF\stackrel{\hat{f}}{\longrightarrow} \TT)$ be a family of formal semistable sheaves of $\XX$ over $\TT$. Then applying the generic fiber functor $(\cdot)_\eta$ we get 
a family of semistable $\kk$-analytic coherent sheaves  $(\fF_\eta\stackrel{f_\eta}{\longrightarrow} \TT_\eta)$ of $\XX_\eta$ over $\TT_\eta$. 
\end{lem}
\begin{proof}
A geometric point of the scheme $\TT_\eta$ is  given by a morphism 
$$\SP_{B}(\kk^\prime)\to \TT_\eta$$
for some algebraically closed non-archimedean field $\kk^\prime$.  Let $R^\prime$ be the ring of integers of $\kk^\prime$ and let 
$\TT^\prime=\spf(R^\prime)$.  Let 
$$i^\prime: \TT^\prime\to \TT$$
be the morphism given by $R^\prime$, and let 
$$(\fF^\prime\stackrel{f^\prime}{\longrightarrow}\TT^\prime)=(i^\prime)^\star ((\fF\stackrel{\hat{f}}{\longrightarrow} \TT))$$
be the pullback of the family over $\TT$.  By flatness it suffices to show that 
the family of $\kk$-analytic coherent sheaves 
$(\fF_\eta^\prime\stackrel{f_\eta^\prime}{\longrightarrow}\TT_\eta^\prime)$ by applying the generic fiber functor $(\cdot)_\eta$
to $(i^\prime)^\star ((\fF\stackrel{\hat{f}}{\longrightarrow} \TT))$ is semistable.
Let $(\fF_s^\prime\stackrel{f_s^\prime}{\longrightarrow}\TT_s^\prime)$ be the object by applying the special fiber functor 
$(\cdot)_s$. Then it is a semistable sheaf over $\TT_s^\prime$.  Consider the reduction maps:
$$
\xymatrix{
\fF^\prime_\eta\ar[r]^{sp}\ar[d]& \fF_s^\prime\ar[d]\\
\TT^\prime_{\eta}\ar[r]^{sp}&\TT_s^\prime
}
$$
Then by the relative GAGA, 
$$H^i(\XX, \fF)\cong H^i(\XX_s, \fF_s)^{\hat{}}\otimes_{\kappa}R$$
and 
$$H^i(\XX_\eta, \fF_\eta)\cong H^i(\XX, \fF)\otimes_{R}\kk.$$
The family on $\XX_s$ is semistable, so does $\XX_\eta$.
\end{proof}

We then prove some global results on the moduli spaces, parallel to the definitions and lemmas above. 

Let $\XX$ be an {\stft}  formal $R$-scheme and let $\mM_{R}(P)(\XX)$ be the moduli stack of semistable formal sheaves over $\XX$ with Hilbert polynomial $P$.  We denote by $\mM_{\kk}(P)(X)$ the moduli stack of $\kk$-analytic semistable coherent sheaves with Hilbert polynomial $P$.  Let $\varpi$ be a uniformizer of the field $\kk$, and let 
$$\sS_m=\spec (R/\varpi^{m+1})$$

\begin{prop}\label{prop_moduli_special_fiber_functor}
Suppose that $\XX$ is a  {\em stft} formal $R$-scheme. Let 
$$X_m=\XX\times_{R}\sS_m$$
Then we have 
$$\varinjlim_{m}\mM_{\sS_m}(P)(X_m/\sS_m)\stackrel{\sim}{\rightarrow} \mM_{R}(P)(\XX),$$
where both stacks are over the site $\Fsch_{R}$.  Hence there is a natural isomorphism
$$(\mM_{R}(P)(\XX))_{s}\stackrel{\sim}{\rightarrow} \mM_{\kappa}(P)(\XX_s)$$
where $(\cdot)_s$ denote the special fiber functor. 
\end{prop}
\begin{proof}
Since $\sS_m$ is a locally noetherian scheme, from Theorem \ref{moduli_locally_noetherian_scheme} $\mM_{\sS_m}(P)(X_m/\sS_m)$ is an algebraic stack locally finite presented over $\sS_m$.  If $\TT\in\Fsch_{R}$ is a formal scheme, and let 
$T_m=\TT\times_{R}\sS_m$. From the definition of the limit $\varinjlim_{m}\mM_{\sS_m}(P)(X_m/\sS_m)$, a morphism 
$$\TT\to \varinjlim_{m}\mM_{\sS_m}(P)(X_m/\sS_m)$$
is given by a compatible sequence of morphisms
$$\it{t}_m:  T_m\to \mM_{\sS_m}(P)(X_m/\sS_m)$$
such that 
$$\it{t}_m=\it{t}_{m+1}\times_{\sS_{m+1}}\sS_m.$$
So we have a family of formal semistable coherent sheaves 
$\fF\to \TT$ and hence a morphism $\TT\to \mM_{R}(P)(\XX)$. The result follows. 
\end{proof}

The following result is an analogue of Theorem 8.7 in \cite{Tony_Yu1}.  If a $\kk$-analytic space $X$ is the analytification of a proper algebraic variety over $\kk$, then the representability of the moduli stack is given by the non-archimedean analytic GAGA. 

\begin{thm}
Let $X$ be a proper algebraic variety over $\kk$.  There exists a natural isomorphism of stacks 
$$(\mM_{\kk}(P)(X))^{\an}\cong \mM_{k}(P)(X^\an)$$
where $(\cdot)^{\an}$ is the analytification functor.  So $ \mM_{k}(P)(X^\an)$ is an analytic stack. 
\end{thm}
\begin{proof}
Let $T=\SP_{B}(A)$ be the Berkovich spectrum for a strictly $\kk$-affinoid algebra $A$.  A morphism 
$$T\to (\mM_{\kk}(P)(X))^{\an}$$
gives rise to a family of semistable sheaves
\begin{equation}\label{family_coherent_sheaves}
F\to \spec(A)
\end{equation}
over $\spec(A)$.  From Lemma \ref{lem_analytification_functor} the analytification of (\ref{family_coherent_sheaves})
gives a family of $\kk$-analytic semistable coherent sheaves $\F\to \SP_{B}(A)$. So we get a morphism:
$$T\to \mM_{\kk}(P)(X^\an).$$
The construction is functorial, hence we have a natural morphism
$$(\mM_{\kk}(P)(X))^{\an}\to \mM_{\kk}(P)(X^\an).$$
We show that the functor 
$$(\mM_{\kk}(P)(X))^{\an}(T)\stackrel{\sim}{\longrightarrow}\mM_{\kk}(P)(X^{\an}(T))$$
is equivalent as categories fiber by groupoids over $\An_{\kk}$. 
It is faithful (easily from the construction).  To prove the surjectivity, let 
$(\F\stackrel{f}{\longrightarrow} X^\an)$ be a family of semistable $\kk$-analytic coherent sheaves of $X^{\an}$ over $T$. 
Let 
$$X^{\an}_{T}=X^{\an}\times_{\kk}T.$$
Then by the $\kk$-analytic GAGA \cite{Conrad}, \cite{Conrad2}, we have $\F\to T$ is the analytification of an algebraic family of semistable coherent sheaves over $\spec(A)$.  So the functor is surjective. The fullness of the functor is also by the GAGA theorem.
\end{proof}

The following result implies that one can globally take the generic fiber for moduli stack of formal coherent sheaves. 
\begin{thm}\label{thm_moduli_generic_fiber_functor}
Let $\XX$ be a  {\em stft} formal $R$-scheme. There is a natural morphism of stacks  over the category 
$\An_{\kk}$ of non-archimedean $\kk$-analytic spaces:
$$(\mM_{R}(P)(\XX))_{\eta}\cong \mM_{\kk}(P)(\XX_\eta),$$
where $(\cdot)_{\eta}$ denotes the generic fiber functor. The stack $\mM_{\kk}(P)(\XX_\eta)$ is a $\kk$-analytic stack. 
\end{thm}

In order to prove Theorem \ref{thm_moduli_generic_fiber_functor}, the existence of the formal model of $\kk$-analytic semistable sheaves is essential. 

\begin{thm}\label{thm_formal_model_up_to_etale_covering}
Let $\XX$ be a  {\em stft} formal scheme  over $R$. Let $T$ be a strictly $\kk$-affinoid space and let 
$$(\F\stackrel{f}{\longrightarrow} T)$$
be a family of $\kk$-analytic semistable sheaves of $\XX_\eta$ over $T$. Then up to passing to a quasi-\'etale covering of $T$, there exists a formal model $\TT$ of $T$ and a family of formal semistable sheaves 
$$(\fF\stackrel{\hat{f}}{\longrightarrow}\TT)$$
of $\XX$ over $\TT$ such that when applying the generic fiber functor we get the family $(\F\stackrel{f}{\longrightarrow} T)$
back. 
\end{thm}
\begin{proof}
From \cite{BW}, for $(\F\stackrel{f}{\longrightarrow} T)$, there is a formal model
$$(\fF\stackrel{\hat{f}}{\longrightarrow}\TT)$$
of $(\F\stackrel{f}{\longrightarrow} T)$ such that it is flat family of formal coherent sheaves. 
Parallel to the stable curve case in \cite{Tony_Yu1} using De Jong's method of alterations, we need to modify the formal model $(\fF\stackrel{\hat{f}}{\longrightarrow}\TT)$ so that it bocomes semistable. 

First let $A$ be a topological algebra finitely presented over $R$ such that 
$\TT=\spf(A)$.  Let 
$$\TT^{\alg}=\spec(A).$$
Since $(\fF\stackrel{\hat{f}}{\longrightarrow}\TT)$ is a flat family of coherent sheaves over $\TT$, by formal GAGA by Grothendieck and Conrad \cite{Conrad}, $(\fF\stackrel{\hat{f}}{\longrightarrow}\TT)$ is isomorphic to the completion of a family of algebraic coherent sheaves 
$$(\fF^{\alg}\stackrel{f^{\alg}}{\longrightarrow}\TT^{\alg})$$
along the special fiber of $\TT$.  Then since the semistability of sheaves is an open condition \cite{Langer}, there must exists an open locus $\TT$, which is equivalent to an \'etale covering $\TT^{\prime\alg}$ of $\TT^{\alg}$ such that 
$$\fF|_{\TT^\prime}\to \TT^{\prime\alg}$$
is a semistable family.  Then we take the completion 
$$\widehat{\fF|_{\TT^{\prime\alg}}}\times_{R}\kk\to \widehat{\TT^{\prime\alg}}\times_{R}\kk$$
along the special fiber over $\kappa$.   Then we get the desired semistable family  of $\kk$-analytic coherent sheaves. 
\end{proof}

\subsection{Proof of Theorem \ref{thm_moduli_generic_fiber_functor}}

First if we have an affine  {\em stft} formal scheme $\TT$  over $R$, then a morphism 
$$\TT\to \mM_{R}(P)(\XX)$$
gives rise to a family of formal semistable coherent sheaves of $\XX$ over $\TT$.  By Lemma \ref{lem_generic_fiber_functor}, when applying the generic fiber functor we get a family of $\kk$-analytic semistable coherent sheaves. Hence we have a morphism 
$$(\mM_{R}(P)(\XX))_{\eta}\to \mM_{\kk}(P)(\XX_\eta).$$
We prove that the functor 
$$\Phi: (\mM_{R}(P)(\XX))_{\eta}(T)\stackrel{\sim}{\rightarrow} \mM_{\kk}(P)(\XX_\eta)(T)$$ 
is an equivalence of groupoids for any strictly $\kk$-affinoid space $T$. 

By construction,  the functor is faithful.  We prove that the functor $\Phi$ is full.  Suppose that we have two families 
$$(\fF_1\stackrel{\hat{f}_1}{\longrightarrow}\TT_1); \quad (\fF_2\stackrel{\hat{f}_2}{\longrightarrow}\TT_2)$$
of formal semistable sheaves of $\XX$ over $\TT_1$ and $\TT_2$, respectively.  Assume that we have an isomorphism of 
$\kk$-analytic semistable sheaves
$$((\fF_1)_\eta\stackrel{(\hat{f}_1)_\eta}{\longrightarrow}(\TT_1)_\eta); \quad ((\fF_2)_\eta\stackrel{(\hat{f}_2)_\eta}{\longrightarrow}(\TT_2)_\eta)$$
when passing to the generic fiber we have the following commutative diagram: 
$$
\xymatrix{
(\fF_1)_\eta\ar[r]^{\sim}\ar[d]_{(\hat{f}_1)_\eta}& (\fF_2)_\eta\ar[d]^{(\hat{f}_2)_\eta}\\
(\TT_1)_\eta\ar[r]^{\sim} &(\TT_2)_\eta.
}
$$
From \cite[Proposition 2.19]{Nicaise}, up to replacing $\TT_1$ and $\TT_2$ by admissible blow-ups, we can assume that 
$\TT_1\cong \TT_2$, which we denote it by $\TT_{12}$.  As in the proof of 
Theorem \ref{thm_formal_model_up_to_etale_covering}, up to passing to the Zariski open covering of $\TT_{12}$, the formal GAGA implies that 
$$(\fF_1\stackrel{\hat{f}_1}{\longrightarrow}\TT_1)$$
and 
$$(\fF_2\stackrel{\hat{f}_2}{\longrightarrow}\TT_2)$$
come from completions of algebraic coherent families 
$$(\fF_1^{\alg}\stackrel{f^{\alg}_1}{\longrightarrow}\TT_1^{\alg})$$
and 
$$(\fF_2^{\alg}\stackrel{f_2^{\alg}}{\longrightarrow}\TT_2^{\alg}).$$
Up to the quasi-\'etale covering of $\TT_{12}^{\alg}$, we can make them semistable. 
Let 
$$\Isom:=\Isom_{\TT_{12}^{\alg}}((\fF_1^{\alg}\stackrel{(f^{\alg}_1}{\longrightarrow}\TT_{12}^{\alg}), (\fF_2^{\alg}\stackrel{(f^{\alg}_2}{\longrightarrow}\TT_{12}^{\alg}))$$
be the isomorphism scheme parameterizing isomorphisms of these two algebraic families. We have a morphism
$$t_{I}: \TT_{12}^{\alg}\times_{R}\kk\to \Isom $$
given by the isomorphism over the generic fibers. Let 
$\widehat{\TT}_{12}^{\alg}$ be the image of $t_{I}$, and let 
$\widehat{\TT}:=\widehat{\widehat{\TT}}_{12}^{\alg}$ be the completion of $\widehat{\TT}_{12}^{\alg}$ along the special fiber over $\kappa$. Then the isomorphism 
$$((\fF_1)_\eta\stackrel{(\hat{f}_1)_\eta}{\longrightarrow}(\TT_1)_\eta)\stackrel{\sim}{\rightarrow} ((\fF_2)_\eta\stackrel{(\hat{f}_2)_\eta}{\longrightarrow}(\TT_2)_\eta) $$
is extended over $\widehat{\TT}_{12}$ and we are done the fullness of the functor $\Phi$. 

The surjectivity of $\Phi$ is given by the following argument. Let 
$$t: T\to \mM_{\kk}(P)(\XX_{\eta})$$
be a morphism, which gives a family of $\kk$-analytic semistable coherent sheaves of $\XX_\eta$ over $T$.  By Theorem 
\ref{thm_formal_model_up_to_etale_covering}, up to quasi-\'etale covering $T^\prime\to T$, there exists a formal scheme 
$\TT^\prime$ and a family of formal semistable coherent sheaves 
$$(\fF^\prime\stackrel{\hat{f}^\prime}{\longrightarrow} \TT^\prime)$$
of $\XX$ over $\TT^\prime$.  Applying the generic fiber functor to above we get the family $(\F\to T)$ back. 
So we obtain morphisms
$$\Tt^\prime\to \mM_{R}(P)(\XX)$$
and 
$$T^\prime\to (\mM_{R}(P)(\XX))_{\eta}.$$
Using the proof of fullness the morphism $T^\prime\to (\mM_{R}(P)(\XX))_{\eta}$ descends to a morphism 
$T\to (\mM_{R}(P)(\XX))_{\eta}$, which gives the morphism $T\to \mM_{R}(P)(\XX_\eta)$. So $\Phi$ is an equivalence of groupoids. The theorem follows.

\subsection{The compactness property}

\subsubsection{K\"ahler structure on non-archimedean analytic spaces}

In the case of the curve counting via stable maps in Gromov-Witten theory, or the sheaf counting in a Calabi-Yau threefold $Y$ in Donaldson-Thomas theory, the degree $\beta$ of the curve is a second homology class $H_2(Y, \zz)$. 
In order to make this work in non-archimedean geometry, we make use of the 
K\"ahler structure of Tschinkel-Kontsevich for the $\kk$-analytic space $X$, as reviewed in \cite[\S 3]{Tony_Yu1}.

We review the most useful part of \cite{Tony_Yu1}.  For a $\kk$-analytic space $X$ with $\XX$ a SNC formal model of $X$.  Let 
$\Sim_{\XX}$ be the sheaf on $S_\XX$ such that for any open subset $U$ of $S_\XX$, $\Sim_{\XX}(U)$ is the set of simple functions of $S_\XX$ restricted to $U$. 

\begin{defn}
For $I\subset I_\XX$, let $N^1(D_{I})_{\rr}=N^1(D_{I})\otimes \rr$, where $D_I=\cap_{i\in I}D_i$.  Let $\Delta^{I}$ be a face of $S_\XX$.  Let $\varphi$ be a simple function on $S_\XX$, define 
$$\partial_{I}\varphi=\sum_{i\in I_\XX}m_i\cdot \varphi(i)\cdot [D_i]|_{D_{I}}\in N^1(D_I)_\rr$$
where $[D_i]$ is the class of the divisor $D_i$. 
\end{defn}

\begin{defn}
A simple function $\varphi$ is said to be {\em linear} (resp. {\em convex, strictly convex}) along the open simplicial face 
$(\Delta^{I})^0$ corresponding to $I$ if $\partial_{I}\varphi$ is {\em trivial} (resp. {\em nef, ample}) in $N^1(D_I)_{\rr}$. 

Let $\LinLoc(\varphi)$ (resp. $\ConvLoc(\varphi), \SConvLoc(\varphi)$) b the union of the open simplicial faces in $S_\XX$ along which $\varphi$ is {\em linear} (resp. {\em convex, strictly convex}).

As in \cite{Tony_Yu1}, let $\Lin_{\XX}$ (resp. {\em $\Conv_{\XX}, \SConv_{\XX}$}) the subsheaf of $\Sim_{\XX}$ whose germs are germs of {\em linear} (resp. {\em convex, strictly convex}) functions. The sheaf $\Lin_{\XX}$ acts on the sheaf $\Sim_{\XX}$ (resp. $\Conv_{\XX}, \SConv_{\XX}$) via:
$$\psi\mapsto (\varphi\mapsto \varphi+\psi)$$
where $\psi$ is a local section of $\Lin_{\XX}$ and $\varphi$ is a local section of $\Sim_\XX$ (resp. $\Conv_{\XX}, \SConv_{\XX}$).
\end{defn}

\begin{defn}
A virtual line bundle $L$ on a non-archimedean $\kk$-analytic space $X$ with respect to the formal model $\XX$ is a torsor over the sheaf $\Lin_\XX$.  A {\em simple} (resp. {\em convex, strictly convex}) metrixation $\widehat{L}$ of a virtual line bundle $L$ is a global section of the sheaf $\Sim_{\XX}\otimes L$ (resp. {\em $\Conv_{\XX}\otimes L, \SConv_{\XX}\otimes L$}), where the tensor product is taken over the sheaf $\Lin_\XX$. 
\end{defn}

\begin{defn}
A K\"hler structure $\widehat{L}$ on $X$ with respect to the formal model $\XX$ is a virtual line bundle $L$ over $X$ with a strictly convex metrization $\widehat{L}$. 
\end{defn}

For each $i\in I_\XX$, a simple metrization $\widehat{L}$ gives a germ of a simple function $\varphi_i$ at the vertex $i$ up to addition by linear functions.  We get a collection of numerical classes 
$$\partial_i\varphi_i\in N^1(D_i)_{\rr}$$
for $i\in I_\XX$.  The collection of classes $\partial_i\varphi_i\in N^1(D_i)_{\rr}$ for every $i\in I_\XX$  is called the {\em curvature} $c(\widehat{L})$ of the metrized virtual line bundle $\widehat{L}$. 
The curvature $c(\widehat{L})$ satisfies the property (\cite[Lemma 3.6]{Tony_Yu1}): for any $\Delta^{I}, I\subset I_\XX$, and $i, j\in I$, 
$$(\partial_i\varphi_i)|_{D_{I}}=(\partial_j\varphi_j)|_{D_{I}}.$$
Tony Yu also proves some functoriality preperty of the curvature:
\begin{prop}(\cite[Proposition  4.3]{Tony_Yu1})
Let $\hat{f}: \XX\to \YY$ be a morphism of SNC formal schemes over $R$.  Let $L$ be a virtual line bundle over $\YY_\eta$ and $\widehat{L}$ a metrization of $L$. Then we have 
$$\hat{f}_s^{\star}c(\widehat{L})=c(\hat{f}^{\star}(\widehat{L})).$$
\end{prop}

\subsubsection{Degree of the virtual line bundle on curves}

In this section we define the degree of virtual line bundle on curves.  Let $X$ be  a smooth connected proper $\kk$-analytic curve.  Let $\XX$ be its formal model, then the Clemens polytope $S_\XX$ is a finite connected simple graph, see 
\cite[\S 4]{Ber}. Although the $\kk$-analytic curve $X$ is smooth, the singularities in the formal model (only double point singularities) correspond to the Type II Berkovich points in $X$.  In \cite{BPR}, Baker,  Payne, and Rabinoff prove a theorem which states that the semistable vertex sets of $X$ are in natural bijective correspondence with semistable models of $X$, where the semistable vertex sets of $X$ are a finite  Type II Berkovich points in $X$.   

Let us fix an order of the set $I_\XX$. For $i, j\in I_\XX$, we say $i\prec j$ if $i, j$ are connected by an edge and $i$ is inferior to $j$ with respect to the fixed order. For $i\in I_\XX$, let $U_i$ be the open neighborhood, which is the union of the vertex $i$ and all the open edges whose closure contain $i$.  If $e_{ij}$ is an edge, then $U_i\cap U_j$ is the interior of $e_{ij}$. 
As in \cite[\S 5]{Tony_Yu1}, using \v{C}ech cohomology of the open cover $\{U_i\}$, Tony Yu proves a degree map
\begin{equation}\label{deg_L}
\deg: H^1(\Lin_\XX)\stackrel{\sim}{\rightarrow} \rr
\end{equation}
and the degree of a virtual line bundle $L$ is defined by (\ref{deg_L}) since $L\in H^1(\Lin_\XX)$. 
Also the degree $\deg(c(\widehat{L}))$ is the curvature 
$c(\widehat{L})$ is given by:
$$\deg(c(\widehat{L}))=\sum_{i\in I_\XX}m_i\cdot \deg(\partial_i \varphi_i).$$

\subsubsection{Moduli of analytic curve counting via ideal sheaves}\label{sec_moduli_absolute_relative_sheaves}

For the counting sheaf theory, we fix our $\kk$-analytic space $X$ to be three dimensional over $\kk$.   
For instance, let $X$ be a smooth Calabi-Yau threefold  over $\kappa$, then the analytification $X^\an$ of $X$ is a smooth three-dimensional $\kk$-analytic space. Let $\XX$ be a SNC formal model of $X$.
Let $P$ be the Hilber polynomial determined by the Chern character 
$$c=(1, 0, \beta, n)\in H^*(X,\zz).$$ 

We fix a K\"ahler structure $\widehat{L}$ on $X$ with respect to a SNC formal model $\XX$.   In this section we mainly consider the ideal sheaves $I_C$ of a connected proper smooth $\kk$-analytic curve $C$, which is surely stable. 

\begin{defn}
The degree of the curve $C$ in $X$ is defined as follows.  Let $\CC$ be a SNC formal model of $C$ and $\CC \subset \XX$ such that $I_\CC$ is the ideal sheaf of $\CC$ in $\oO_{\XX}$.  Then $L|_{\CC}$ is a virtual line bundle on $C$ with respect to the formal model $\CC$.  Define 
$$\deg(I_C)=\deg (L|_{\CC}).$$
\end{defn}
\begin{rmk}
The degree of the ideal sheaf of the curve $C$ does not depend on the choice of the formal model $\CC$, since any two of them can be dominated by another one.  If the $\kk$-analytic curve $C$ is not smooth but with ordinary double points,  we define the degree of $C$ to be the sum of the degrees of the connected components after taking normalization. 
\end{rmk}

We fix a Hilbert polynomial $P$. 
Let $\iI^{P}_{\kk}(X)$ be the moduli stack of  $\kk$-analytic ideal sheaves over $X$ with Hilbert polynomial $P$. 
Similarly let $\iI^{P}_{R}(\XX)$ and $\iI_{\kappa}^{P}(\XX_s)$  be the moduli stack of stable formal  ideal sheaves over 
$\XX$ with Hilbert polynomial $P$  and algebraic ideal sheaves over $\XX_s$ with Hilbert polynomial $P$. 
In this section we prove a result that ideal sheaves (or more general stable coherent systems as in \cite{LW})  on a proper $\kk$-analytic space $X$, after a choice of SNC formal model $\XX$ of $X$, the sheaf when restricted to the special fiber $\XX_s$, will induces a stable sheaf on the SNC divisors of $\XX_s$.  Hence we deduce a decomposition result for the moduli space. 
From Theorem \ref{thm_moduli_generic_fiber_functor} and Proposition \ref{prop_moduli_special_fiber_functor}
$$\iI^{P}_{R}(\XX)_{\eta}\cong \iI^{P}_{\kk}(X)$$ 
$$\iI^{P}_{\kk}(\XX)_s\cong \iI^{P}_{\kappa}(\XX_s).$$

\begin{prop}\label{prop_moduli_central_fiber_proper}
If the formal scheme $\XX$ is proper and $\XX_s$ has only simple normal crossing divisors, then 
the moduli stack $\iI^{P}_{\kappa}(\XX_s)$ is proper. 
\end{prop}
\begin{proof}
Still let $\{D_i|i\in I_\XX\}$ be the smooth irreducible components of $\XX_s$. Then every ideal sheaf $I_{Z}$ of a curve $Z\subset \XX_s$ with Hilbert polynomial $P$ admits a splitting of $P$, which is the set of 
$$\gamma=\{P_{I}| I\subset I_\XX\}$$
such that 
$$P=\sum_{I\subset I_\XX}(-1)^{|I|-1}P_{I}.$$
Thus there is a closed embedding of the moduli stacks
$$\iI^{P}_{\kappa}(\XX_s)\to \bigcup_{\gamma}\bigtimes_{\{P_I\}} \iI_{\kk}^{P_I}(D_I)$$
which sends an ideal $I_Z$ to its corresponding restriction on $D_I$. Since every $D_I$ is smooth and proper, the moduli stack $\iI_{\kk}^{P_I}(D_I)$ is proper.  Hence we get the properness of $\iI^{P}_{\kappa}(\XX_s)$.
\end{proof}

\begin{cor}\label{cor_moduli_space_X_proper}
Assume that $\mbox{Ch}(\kappa)=0$. Let $X$ be a proper $\kk$-analytic space, then the moduli stack $\iI^{P}_{\kk}(X)$is a proper $\kk$-anaytic stack. 
\end{cor}
\begin{proof}
Let $\XX$ be the SNC formal model of $X$.  Then $\XX$ is proper if and only if the $\kk$-analytic space $X$ is proper \cite[Corollary 4.4]{Temkin2}.   From Proposition \ref{prop_moduli_central_fiber_proper} the moduli stack  $\iI^{P}_{\kappa}(\XX_s)$ is proper as an algebraic moduli stack. 
The result hence follows from Proposition \ref{prop_moduli_special_fiber_functor} and Theorem \ref{thm_moduli_generic_fiber_functor}. 
\end{proof}

\subsubsection{Universal stack of SNC formal models}\label{sec_stack_SNC_formal_model}

Let $X$ be a smooth non-archimedean $\kk$-analytic space, and fix a SNC formal model $\XX$ of $X$.  First let us recall a result in \cite[Proposition 2.19]{Nicaise}.
\begin{prop}
Let $\XX^\prime\to\XX$ be an admissible formal blow-up with center an admissible ideal $I$. Then 
$$\XX^\prime_{\eta}\cong \XX_\eta$$
as rigid varieties and Berkovich analytic spaces. 
\end{prop}

We define a universal stack $\MM:=\MM^{(X,\XX)}$ of SNC formal models of $(X,\XX)$.  Recall  $I_{\XX}$ is the number 
of irreducible components of $\XX_s$.   For $I\subset I_{\XX}$ a subset, we let 
$$I_p:=\{\text{the first~}(|I|-1)\text{~elements in } I\}$$
and 
$$D_p:=\bigcap_{i\in I_p}D_i.$$
Using $I\subset I_\XX$, we construct the admissible blow-ups 
$\XX_{I}[n]\to\XX$ for any positive integer $n$ by induction. Let $\XX_{I}[0]=\XX$, and
$\pi_1: \XX_{I}[1]\to \XX$ be the admissible formal blow-up along $D_{I}$.  Let 
$\Delta_{I}$ be the exceptional locus of $\pi_1$, which is a projective bundle over $D_I$, and let 
$D_I[1]:=D_{I_p}\cap \Delta_{I}$.  Then $D_I[1]$ has the same singular type as $D_I$.  Let 
$\pi_2: \XX_{I}[2]\to \XX_{I}[1]$ be the admissible formal blow-up along $D_{I}[1]$.  
Suppose that we have $\pi_{n-1}: \XX_{I}[n-1]\to \XX_{I}[n-2]$.  Let 
$\Delta_{I}[n-2]$ be the exceptional locus of $\pi_{n-1}$, and let 
$D_{I}[n-1]=\Delta_{I}[n-2]\cap D_{I_p}$.  Then let 
$\pi_n: \XX_{I}[n]\to \XX_{I}[n-1]$ be the admissible formal blow-up along $D_{I}[n-1]$. 

Now we use $\XX_{I}[n]$ to construct the universal stack of SNC formal models. 
Let 
$\xi:\SS\to\spf(R)$ be a formal $R$-scheme. An effective formal scheme over $\SS$ is a family 
$$\WW:=\XX_I[n]\times_{R}\SS$$
for some $I\subset I_\XX$ 
over $\SS$ and a tautological projection 
$\WW\to \XX\times_{R}\SS$.  We denote by $\xi$ this effective family.  Let $\xi_1:  \SS_1\to\spf(R)$ and 
$\xi_2:  \SS_2\to\spf(R)$ be two effective families.  Then an arrow $r: \xi_1\to\xi_2$ is given by 
\[
\xymatrix{
\SS_1\ar[rr]^{r}\ar[dr]_{\xi_1}&& \SS_2\ar[dl]^{\xi_2} ~\\
&spf(R)&
}
\]
such that $\xi_1^\star\XX_I[n_1]\cong r^\star\xi_2^\star\XX_I[n_2]$.  We denote by $r$ this isomorphism. We say two effective formal families $\WW_1$ and $\WW_2$ over $\SS$ are isomorphic if $\WW_1$ is  $\SS$-isomorphic to  $\WW_2$
by $r$ and they are compatible with respect to the projections to $\XX\times_{R}\SS$. 
We define the universal stack $\MM$ of SNC formal models of $(X,\XX)$. 

\begin{defn}
Let $\SS$ be a formal $R$-scheme.  A universal SNC formal scheme over $\SS$ is a pair $(\WW,\rho)$, such that 
$\WW\to\SS$ is a family of formal schemes over $\SS$, and 
$$\rho: \WW\to\XX\times_{R}\SS$$
is a $\SS$-projection; so that 
there exists an open covering $\SS_\alpha$ of affine formal schemes of $\SS$ and 
$$\WW\times_{\SS}\SS_\alpha=\XX_I[n]\times_{R}\SS$$
for some $I\subset I_\XX$ 
and 
$\rho|_{\WW\times_{\SS}\SS_\alpha}: \XX_I[n]\times_{R}\SS\to\XX\times_{R}\SS$ is the tautological projection. Let 
$(\WW_1, \rho_1)$ and $(\WW_2, \rho_2)$ be two families of universal formal $\SS_1$- and $\SS_2$-schemes.  An arrow $\WW_1\to \WW_2$ is given by a 
$\spf(R)$-morphism $\SS_1\to \SS_2$ such that 
$\WW_1\cong \WW_2\times_{\SS_2}\SS_1$
is a $\SS_1$-isomorphism compatible with the projections to $\XX\times_{R}\SS$. 
\end{defn}

Then we define the functor 
\begin{equation}\label{defn_stack_universal}
\MM:=\MM^{(X,\XX)}: \Fsch_{R}\to (\text{groupoids})
\end{equation}
by
$$\SS\mapsto \{\text{the groupoid of family of universal SNC formal schemes over~}\SS\}.$$

Then a routine check from the definition of stacks shows that 
\begin{prop}\label{prop_stack_SNC_formal_model}
The functor $\MM$ is a stack over $\spf(R)$. There exists a morphism $\pi: \MM\to \spf(R)$ such that $\MM_\eta\cong X$. 
\end{prop}

\begin{rmk}
Proposition \ref{prop_stack_SNC_formal_model} is related to an interesting result for formal blow-ups.  Fixing a formal scheme $\XX$, let $\{\XX_i\}_{i\in K}$ be all the admissible formal blow-ups of $\XX$, then the direct limit 
$\lim_{i}(\overline{\XX}_i)$ of all the underlying space of the formal blow-ups modulo the equivalent relations is the non-archimedean analytic space $\XX_\eta$. 
\end{rmk}

\subsubsection{Relation to the stack of expanded degenerations}\label{sec_stack_expanded_degeneration}

Recall from \cite{Li1}, let $X$ be a $\kappa$-scheme and $W\to \aaa^1_{\kappa}$ be a degeneration family such that
$W_t\cong X$, and $W_0\cong D_1\cup_{D_{12}}D_2$.  Let $\XX\to\spf(R)$ be the formal completion of $W$ along $W_0$. 
Then $\XX\to \spf(R)$ is a {\em stft} formal scheme such that $\XX_s=W_0$. 
We have the universal stack of SNC formal models $\MM^{(X,\XX)}$ for $X$.  Since $\XX_s$ only has two irreducible components $D_1, D_2$,
there exists only one nontrivial admissible formal blow-up $\XX_{I}[n]$ for $|I|=|I_{\XX}|=2$. We just denote by
$\XX[n]:=\XX_{I_\XX}[n]$. 
For any formal model $\XX[n]$ in the  universal stack of SNC formal models $\MM^{(X,\XX)}$, the central fiber $(\XX[n])_s$ consists of $D_1, D_2$ and $n$ irreducible components 
$\Delta_i=\pp_{D_{12}}(N_{D_{12}/D_1}\oplus\oO)$ for $i=1, \cdots, n$ in between.    The intersection of $\Delta_i\cap \Delta_j\cong D_{12}$ and there are totally $n+1$ such singular $D_{12}$'s and we denote them by $E_i$ for $1\leq i\leq n+1$. 

\begin{defn}(\cite[Definition 3.1]{LW})
A coherent sheaf $F$ of $\XX_s$ is normal to a subscheme $D$ if $\Tor_{1}^{\oO_{\XX_s}}(F, \oO_{D})=0$.
\end{defn}

\begin{defn}(\cite[Definition 3.9]{LW})
A coherent sheaf $F$ of $\XX[n]$ is admissible if it is normal  to every $E_i$.
\end{defn}

Then we can construct a moduli stack $\iI_{R}^{P}(\MM)$ of admissible ideal sheaves of curves with Hilbert polynomial $P$.   This stack is defined by the following category: let $(\WW,\rho)\in \MM(\SS)$ be a universal family of formal schemes over $\SS$, then we define a family of admissible stable ideal sheaves over $\WW$. Thus we get a category of stable ideal sheaves over the universal stack of SNC formal models $\MM$.   

For the degeneration $W\to \aaa^1_{\kappa}$,  let $\MM^{\alg}$ be the stack of  algebraic expanded degenerations as in \cite[Proposition 1.10]{Li1} and \cite[Definition 2.5]{LW}.   We also have the the moduli stack $\iI^{P}_{\kappa}(\MM^{\alg})$  of stable ideal sheaves over the stack  $\MM^{\alg}$ of  expanded degenerations with Hilbert polynomial $P$ as in \cite[Proposition 4.4]{LW}, \cite[Theorem 4.14]{LW}.  The stack $\iI_{\kappa}^{P}(\MM^{\alg})$ is a stack over $\aaa^1_{\kappa}$, so  let $\iI_{\kappa}^{P}(\MM^{\alg})_0$ be the central fiber as in \cite[\S 5.7]{LW}. 

\begin{thm}\label{thm_formal_completion_expanded_degeneration}
We have 
$\iI_{R}^{P}(\MM)=\widehat{\iI^{P}_{\kappa}(\MM^{\alg})}$, the formal completion along the origin $0\in\aaa^1_{\kappa}$. 
So 
$\iI_{R}^{P}(\MM)$ is a formal stack  of finite type. 
\end{thm}
\begin{proof}
Let $\SS$ be the formal completion $\widehat{S}$ of a $\kappa$-scheme $S$ over $\aaa^1_{\kappa}$ along the origin. 
We check this over $\MM^{\alg}(S)$ and $\MM(\SS)$ such that there are families of 
effective degenerations $W[n]\times_{\aaa^1_{\kappa}}S$ and universal formal scheme $\XX[n]\times_{R}\SS$.  The moduli stack  $\iI_{R}^{P}(\MM)$  over $\SS$ is exactly the formal completion $\widehat{\iI^{P}_{\kappa}(\MM^{\alg})}$ over 
$S$. Thus the result follows. 
\end{proof}

\begin{rmk}
\begin{enumerate}
\item  Suppose that the support of $F$ is a curve $C$ on $\XX_s$. The normality of $F$ means that $C$ intersects with $D_{ij}$ transversally;
\item  Transversality of the curve $C$ with divisors $D_{ij}$ of $\XX_s$ is essential to the study of relative Gromov-Witten and  Donaldson-Thomas theory as in \cite{Li1}, \cite{Li2}, \cite{LW}.  Gross-Siebert \cite{GS} use Log stable maps to Log schemes to define the log version of Gromov-Witten invariants, where the normality is solved by logarithmic techniques. 
See \cite{AC} for a program on the case that $Y$ is normal crossing. 
\item  Our result implies that it seems natural to use non-archimedean geometry to study degenerations of the moduli stack of stable coherent sheaves. 
It is  also interesting to produce the degeneration formula of Jun Li \cite{Li2} using non-archimedean geometry. 
\end{enumerate}
\end{rmk}

We set for the splittings of the Hilbert polynomial $P$, which is the set of 
$$\gamma=\{P_{1}, P_2, P_{12}\}$$
such that 
$$P=P_1+P_2-P_{12}.$$

Let 
$\iI_{\kappa}^{P_i}(D_i, D_{12})$  for $i=1,2$ be the moduli stack of  relative stable ideal sheaves of $D_i$ relative to  $D_{12}$ in the sense of \cite{LW} using the stack of relative expanded pairs.  
We have the following gluing theorem as in \cite[Theorem 5.28]{LW}.

\begin{thm}(\cite[Theorem 5.28]{LW})\label{thm_moduli_space_absolute_relative}
Let $X$ be a proper $\spec(\kappa[t])$-scheme and $\XX$ its $t$-adic formal completion such that $\XX_s=D_1\cup_{D_{12}}D_2$.  Then the moduli stack 
$\iI^{P}_{\kappa}(\MM_s)$, after applying the special fiber functor, has a canonical gluing isomorphism
$$\iI_{\kappa}^{P}(\MM_s)\stackrel{\sim}{\rightarrow}\iI_{\kappa}^{P_1}(D_1, D_{12})\times _{\Hilb_{\kappa}^{P_{12}}(D_{12})}\iI_{\kappa}^{P_1}(D_1, D_{12})$$
of Deligne-Mumford stacks, where $\Hilb_{\kappa}^{P_{12}}(D_{12})$ is the Hilbert scheme of points on $D_{12}$.
\end{thm}

\begin{rmk}
Note that in \cite{AC}, the authors constructed the moduli stack of log stable maps to a generalized Deligne-Faltings pair. 
A scheme $X$ with a SNC divisor $D$   is a generalized Deligne-Faltings pair.  It is interesting to see whether the techniques in \cite{AC} works for relative or log ideal sheaves of curves.
\end{rmk} 
\begin{rmk}
In the SNC formal model $\XX$, if $\XX_s$ consists of two irreducible components $D_1, D_2$, then 
Corollary \ref{cor_moduli_space_X_proper} can also be obtained from Theorem \ref{thm_moduli_space_absolute_relative}, since the moduli stack $\iI_{\kappa}^{P}(\MM_s)$ is proper if $\XX$ is proper.
\end{rmk}

\begin{cor}
Now let $A$ be a positive number.  Suppose that in $\iI^{P}_{\kk}(X)$, the degree of the stable sheaf with respect to the K\"ahler structure $\widehat{L}$ is $A$.  Then we have a decomposition 
$$A=A_1+A_2$$
where $A_i$ is the degree of the stable sheaf  in $\iI_{\kk}^{P_i}(D_i, D_{12})$  with respect to the K\"ahler structure $\widehat{L}|_{D_i}$.
\end{cor}
\begin{proof}
This is from the decomposition of the stable sheaf and the normality property. 
\end{proof}


\section*{Part II:}

 \section{Introduction to motivic Donaldson-Thomas theory}\label{sec_motivic_DT_introduction}

\subsection{Donaldson-Thomas invariants}

Let $Y$ be a smooth Calabi-Yau threefold or a smooth threefold Calabi-Yau Deligne-Mumford stack. The Donaldson-Thomas invariants of $Y$ count stable coherent sheaves on $Y$.   In \cite{Thomas}, R. Thomas  constructed a perfect obstruction theory $E^\bullet$ in the sense of Li-Tian \cite{LT}, and Behrend-Fantechi \cite{BF} on the moduli space $X$ of stable sheaves over $Y$, hence a virtual fundamental class $[X]^{\virt}$ on $X$.   If $X$ is proper, then the virtual dimension of $X$ is zero, and the integral
$$\DT_Y=\int_{[X]^{\virt}}1$$
is the Donaldson-Thomas invariant of $Y$.  Donaldson-Thomas invariants have been proved to have deep connections to Gromov-Witten theory and provided more deep understanding of the curve counting invariants, see \cite{MNOP1}, \cite{MNOP2}, \cite{PT}, etc. 

In the Calabi-Yau threefold case, in \cite{Behrend} Behrend proves that the moduli scheme $X$ of stable sheaves on $Y$ admits a symmetric obstruction theory which is defined by him in the same paper \cite{Behrend}.  Also Behrend constructs a canonical integer-valued constructible function 
$$\nu_{X}: X\to\zz$$
by using the local Euler obstruction of an intrinsic integral cycle 
$\mathbf{c}_{X}\in Z_*(X)$
on $X$.  We call $\nu_{X}$ the {\em Behrend function} of $X$.  If $X$ is proper, then in \cite[Theorem 4.18]{Behrend} Behrend proves that 
$$\DT_Y=\int_{[X]^{\virt}}1=\int_{X}c^{\tiny\mbox{CSM}}(\nu_X)=\chi(X,\nu_{X}),$$
where $\chi(X,\nu_{X})$ is the weighted Euler characteristic weighted by the Behrend function, and  $c^{\tiny\mbox{CSM}}(\nu_X)$ is the Chern-Schwartz-MacPherson class of the Behrend function $\nu_X$.  The above result is the index theorem of MacPherson, which is a generalization of Gauss-Bonnet theorem to singular scheme $X$. 
Same result for a proper Deligne-Mumford stack $X$ with a symmetric perfect obstruction theory is conjectured by Behrend in \cite{Behrend}, and  is proved in \cite{Jiang3}.  This makes the Donaldson-Thomas invariants {\em  motivic}. 

\subsection{Categorification of Donaldson-Thomas invariants}

Around 2006, Kai Behrend proposed  a natural question called ``the categorification" of Donaldson-Thomas invariants, i.e., to find a cohomology theory of the Donaldson-Thomas moduli space $X$ of stable sheaves on Calabi-Yau threefolds such that its Euler characteristic is the weighted Euler characteristic by the Behrend function.   If locally  the moduli space $X$ is the critical locus of a holomorphic function $f: M\to\kappa$ on a higher dimensional smooth scheme $M$, then the categorification of Donaldson-Thomas invariants is given by $\mMF_{M,f}^{\phi}$, the sheaf of vanishing cycles of the function $f$.  The sheaf $\mMF_{M,f}^{\phi}$ is a constructible sheaf and its Euler characteristic at a point $P\in X$ gives rise to the value of the Behrend function $\nu_{X}$ at $P$. 

So the question is whether locally the Donaldson-Thomas moduli space $X$ is the critical locus of a holomorphic function.  The local or germ deformation theory of $X$ is controlled by a differential graded Lie algebra $L$, then one can study this local question by studying the local controlled differential graded Lie algebra $L$.   Behrend and Getzler \cite{BG} actually studied the local behavior  of the moduli space using  cyclic differential graded Lie algebra and cyclic $L_\infty$-algebras, and announced that the local function $f$ is coming from the cyclic  $L_\infty$-algebra structure of the Donaldson-Thomas moduli space $X$  and is holomorphic.  But the paper is still not available. 

The local structure of the moduli space $X$ is solved by Ben-Bassat, Brav, Bussi and  Joyce \cite{BBJ}, \cite{BBBJ} by using the techniques of derived schemes of \cite{PTVV}, \cite{Lurie}.  In \cite{PTVV}, Pantev, Toen, Vaquie and  Vezzosi introduced the notion of $(-n)$-shifted symplectic structure on the derived schemes.  The moduli space $X$ of stable sheaves over the Calabi-Yau threefold $Y$ can be lifted to  a $(-1)$-shifted symplectic derived scheme $\mathbf{X}$ such that its underlying scheme is $X$.   There is a natural inclusion $i: X\to \mathbf{X}$ to the derived scheme, and there is a  symmetric obstruction theory of Behrend \cite{Behrend} on $X$, which is given by the pullback of the cotangent complex $L_{\mathbf{X}}$ of $\mathbf{X}$ to $X$.  Brav, Bussi and  Joyce in  \cite{BBJ} prove that if $X$ is the underlying scheme of a $(-1)$-shifted symplectic derived scheme $\mathbf{X}$, then locally $X$ is given by the critical locus of regular function $f$.  That is: for any point $x\in X$, there is an open neighborhood $x\in R$, and a regular function $f: U\to\kappa$ on a smooth scheme $U$ such that $R=\Crit(f)$.  Joyce calls $(R, U, f,i)$ a critical chart of $X$, where $i: R\to X$ is the inclusion. 
We kindly reminder the reader here that $R$ represents the critical scheme $\Crit(f)$ of the function $f$, not the discrete valuation ring. 
All the critical charts of $X$ glue together to give a structure on $X$, which Joyce \cite{Joyce} calls $X$ the $d$-critical scheme.  Hence locally on $R$, we have a sheaf $\mMF_{U,f}^{\phi}$ of vanishing cycles of $f$. In \cite{BJM}, Bussi, Joyce, and Meinhardt prove that these data of vanishing cycles glue to give a global sheaf $\mMF_{X}^{\phi}$ if there is an orientation on $X$, i.e., a square root $K_{X}^{\frac{1}{2}}$ of the canonical line bundle $K_{X}$.  Thus the categorification of $X$ is obtained.   The vanishing cycle sheaf is a perverse sheaf, Kiem and Li \cite{KL} also use the gluing of perverse sheaves to give a global sheaf $\mMF_{X}^{\phi}$ and categorify the moduli space $X$.  The perverse sheaf of vanishing cycles $\mMF_{X}^{\phi}$ is used recently by D. Maulik and Y. Toda  in \cite{MT} to define the Gopakumar-Vafa invariants for Calabi-Yau threefolds and relate them to Gromov-Witten invariants and Pandharipande-Thomas stable pair invariants. 
Note that Maulik and Toda require the orientation data $K_{X,s}^{\frac{1}{2}}$ is trivial, which is called the CY orientation data. 

\subsection{Motivic Donaldson-Thomas invariants and derived schemes}

The vanishing cycle sheaf can be made to be motivic by the notion of motivic vanishing cycles, which is an element in $\mM_{X}^{\hat{\mu}}$, the equivariant Grothendieck ring of varieties.  Kontsevich and Soibelman \cite{KS} introduced the motivic Donaldson-Thomas theory for any oriented Calabi-Yau category $\mathcal{C}$.  They defined the motivic weights for any object $E\in \mathcal{C}$ by using the motivic vanishing cycle of the object $E$ and the technique of  a cyclic $A_\infty$-algebra  $L_{E}=\Ext(E,E)$ associated with the object $E$. Then they prove that there is a homomorphism from the motivic Hall algebra $H(\mathcal{C})$ of $\mathcal{C}$ to the motivic quantum torus of $\mathcal{C}$, hence deduce a wall crossing formula of motivic Donaldson-Thomas invariants.  In the case that the moduli space $X$ is the global critical locus of a regular function, the motivic Donaldson-Thomas invariants are defined and studied in \cite{BBS}. 

In \cite{BJM}, Bussi, Joyce, and Meinhardt also study the motivic Donaldson-Thomas invariants for the oriented $d$-critical scheme
$X$.  
On each $d$-critical chart $(R,U,f,i)$, one needs to consider the motive $\Upsilon(P)$ of a principal $\zz_2$-bundle $P\to R$. For another principal $\zz_2$-bundle $Q\to R$, it is hard to prove that  $\Upsilon(P\otimes_{\zz_2}Q)=\Upsilon(P)\odot\Upsilon(Q)$.  So we define a motivic ring $\overline{\mM}_{X}^{\hat{\mu}}$, by quotient out the relations  
$\Upsilon(P\otimes_{\zz_2}Q)-\Upsilon(P)\odot\Upsilon(Q)$. Then 
Bussi, Joyce, and Meinhardt prove that the global vanishing cycle sheaf $\mMF_{X}^{\phi}$ is obtained by gluing the motivic vanishing cycle 
$\mMF_{U,f}^{\phi}$ for all the $d$-critical charts $(R,U,f,i)$ and get  a global motivic element in $\overline{\mM}_{X}^{\hat{\mu}}$.

\subsection{Contribution of this work}

In the second part of the paper we generalize Joyce's $d$-critical scheme to the case of {\em formal $d$-critical schemes} and {\em $d$-critical non-archimedean $\kk$-analytic spaces}.  We mainly use Joyce's definition of $d$-critical schemes. The version of Kiem-Li's virtual critical manifold structure in \cite{KL} can also be generalized to formal schemes and Berkovich non-archimedean analytic spaces. 

The $d$-critical scheme of Joyce is the classical model for the $(-1)$-shifted symplectic derived scheme in \cite{PTVV}.  The notion of  derived formal schemes was developed in Chapter 8 of  Lurie \cite{Lurie}, and the notion of derived non-archimedean $\kk$-analytic spaces was given in \cite{Porta_Yu1} using the terminology  of Lurie.  
The  $(-n)$-shifted symplectic structures on such derived spaces can be similarly defined, and to the author's knowledge, these have not been explored.  Once there is a $(-1)$-shifted symplectic structure on the derived formal scheme and derived non-archimedean $\kk$-analytic spaces, 
one hopes that taking generic functor of the derived formal scheme we get  the latter.  
We mainly focus on the classical part of such spaces, and generalize Joyce's (or Kiem-Li's) arguments to formal schemes and non-archimedean $\kk$-analytic spaces. 

As mentioned earlier, we hope that the $d$-critical non-archimedean $\kk$-analytic spaces will be the underlying 
non-archimedean spaces of derived  non-archimedean spaces with  a $(-1)$-shifted symplectic structure.  Hence there should exist a symmetric obstruction theory of Behrend in \cite{Behrend} on $d$-critical non-archimedean $\kk$-analytic spaces.  Such a space will be the foundation space to the non-archimedean counting invariants coming from a symmetric obstruction theory. 

Assume that $\XX$ is a $d$-critical formal $R$-scheme, such that its generic fiber $\XX_\eta$ is a $d$-critical non-archimedean $\kk$-analytic space.  We also construct a canonical line bundle $K_{\XX}$ on $\XX$ and define the notion of orientation of $\XX$.   With the orientation $K_{\XX}^{\frac{1}{2}}$, we prove that there exists a global motive of vanishing cycles $\mMF_{\XX}^{\phi}$, which is an element in $\overline{\mM}_{\XX_s}^{\hat{\mu}}$ where $\XX_s$ is the special fiber of $\XX$ and is a $\kappa$-scheme.  This global motive is also obtained by gluing the local motivic vanishing cycle $\mMF_{\UU, \hat{f}}^{\phi}$ for the formal $d$-critical chart $(\RR,\UU,\hat{f},i)$ of $\XX$ by considering the motive of principal $\zz_2$-bundle on $\RR$ and the orientation. 
Suppose that $\XX$ is the formal completion  of a $d$-critical scheme $X$ of Joyce in \cite{Joyce} or Kiem-Li in \cite{KL}, then by the relative GAGA in \cite{Conrad}, the global sheaf  $\mMF_{\XX}^{\phi}$ is the formal completion of  $\mMF_{X}^{\phi}$. 

Let $X$ be a $d$-critical non-archimedean $\kk$-analytic space.  $X$ is said to be oriented if the square root $K_{X}^{\frac{1}{2}}$ of the canonical line bundle $K_{X}$ exists.  Let $\XX$ be a formal model of $X$, such that $\XX$ is an oriented $d$-critical formal $R$-scheme and $\XX_\eta\cong X$.  We define the absolute motive 
$\mMF_{X}^{\phi}:=\int_{\XX_s}\mMF_{\XX}^{\phi}\in\overline{\mM}_{\kappa}^{\hat{\mu}}$, where $\int_{\XX_s}$ means pushforward to a point.  The absolute global motive 
$\mMF_{X}^{\phi}$ depends only on $X$, i.e., is independent to the choice of the formal model. 

We also introduce the $\Gm$-equivariant $d$-critical formal schemes and $\Gm$-equivariant $d$-critical non-archimedean $\kk$-analytic spaces. 
As an application, we generalize the motivic localization formula of D. Maulik \cite{Maulik} for the motivic Donaldson-Thomas invariants to $d$-critical non-archimedean $\kk$-analytic spaces and $d$-critical formal schemes under the $\Gm$-action by using motivic integration for formal schemes  in \cite{Nicaise}.

\subsection{Outline}

We outline the structure of Part II.  In \S \ref{sec_formal_d_critical_scheme} we introduce the notion of  $d$-critical formal $R$-schemes and orientations on $d$-critical formal $R$-schemes.  The $d$-critical $\kk$-analytic spaces will also be constructed.   We basically follow Joyce's methods in \cite{Joyce}, and provide proofs for necessary steps. 
\S \ref{sec_main_defn_results} contains the main construction and results. 
\S \ref{sec_Proof_main_Theorem} provides the proof of the main construction of the structure sheaf of the $d$-critical formal schemes.   We introduce $d$-critical non-archimedean $\kk$-analytic spaces in \S \ref{sec_d_critical_analytic_space}; and 
in \S \ref{sec_Gm_equivariant_formal_scheme_analytic_space} we talk about the $\Gm$-equivariant  $d$-critical formal $R$-schemes and   $d$-critical non-archimedean $\kk$-analytic spaces. 
In \S \ref{sec_motivic_localization} we study the motivic localization formula for motivic Donaldson-Thomas invariants; where in \S \ref{Grothendieck:ring} we recall the equivariant Grothendieck group of varieties;  in \S \ref{motivic:integration:rigid} we briefly review the motivic integration for formal schemes in \cite{Nicaise}; in \S \ref{subsec_global_motive_formal_scheme} we define the global motive for 
oriented $d$-critical formal schemes;  in \S \ref{subsec_global_motive_analytic_space} we define the global motive for 
oriented $d$-critical non-archimedean $\kk$-analytic spaces; and 
finally in \S \ref{subsec_motivic_localization_formula} we prove the motivic localization formula for oriented  $d$-critical 
non-archimedean $\kk$-analytic spaces.  We also relate it to Maulik's motivic localization formula for motivic Donaldson-Thomas invariants.


\section{$d$-critical formal schemes}\label{sec_formal_d_critical_scheme}

We introduce the formal version of $d$-critical locus in the sense of \cite{Joyce}, and critical virtual manifold in the sense of Kiem-Li \cite{KL}, which we call the $d$-critical formal scheme.

\subsection{Definitions and Results}\label{sec_main_defn_results}
\subsubsection{Main construction}

The basic knowledge for formal schemes can be found in \cite{Bosch}, \cite{Fujiwara-Kato}.
For instance, the sheaf of differentials $\Omega_{\XX}$ for a {\em stft} $R$-formal scheme $\XX$ is defined in
\cite[\S 5.2]{Fujiwara-Kato}.
We first have the following result for formal schemes:

\begin{thm}\label{thm_existence_sheaf_PX}
Let $\XX$ be a {\em stft} formal $R$-scheme. Then there exists a sheaf $\FP_{\XX}$ of $R$-commutative algebras on $\XX$, unique up to canonical isomorphism, is characterized by the following properties:
\begin{enumerate}
\item Suppose that $\RR\subset \XX$ is Zariski open, and $\UU$ a smooth formal $R$-scheme, and 
$i: \RR\to \UU$ an embedding. Then there exists an exact sequence of shaves of $R$-commutative algebras
$$0\rightarrow I_{\RR,\UU}\longrightarrow i^{-1}(\oO_{\UU})\longrightarrow \oO_{\XX}|_{\RR}\rightarrow 0$$
where $\oO_{\UU}, \oO_{\XX}$ are the structure sheaves, and 
an exact sequence of sheaves of $R$-commutative algebras:
$$0\rightarrow \FP_{\XX}|_{\RR}\stackrel{\iota_{\RR}}{\longrightarrow} \frac{i^{-1}(\oO_{\UU})}{I_{\RR,\UU}^2}
\stackrel{d}{\longrightarrow} \frac{i^{-1}(\Omega_{\UU})}{I_{\RR,\UU}\cdot i^{-1}(\Omega_{\UU})}$$
where $d(f+ I^2_{\RR,\UU})=df+I_{\RR,\UU}\cdot i^{-1}(\Omega_{\UU})$.
\item If $\RR\subset \SS\subset\XX$ are Zariski open formal subschemes, and 
$$i: \RR\hookrightarrow \UU; \quad j: \SS\hookrightarrow \VV$$
are closed embeddings into smooth formal $R$-schemes $\UU, \VV$. Let 
$$\Phi: \UU\to \VV$$
be a morphism with $\Phi\circ i=j|_{\RR}: \RR\to \VV$. Then the following diagram of $R$-commutative algebra sheaves commutes:
\begin{equation}\label{diagram_1}
\xymatrix{
0\ar[r]&\FP_{\XX}|_{\RR}\ar[r]^{\iota_{\SS,\VV}|_{\RR}}\ar[d]^{\id}& \frac{i^{-1}(\oO_{\VV})}{I_{\SS,\VV}^2}|_{\RR}\ar[r]^{d}\ar[d]^{i^{-1}(\Phi^{\#})}& \frac{i^{-1}(\Omega_{\VV})}{I_{\SS,\VV}\cdot i^{-1}(\Omega_{\VV})}|_{\RR}\ar[r]\ar[d]^{i^{-1}(d\Phi)}&0 
\\
0\ar[r]&\FP_{\XX}|_{\RR}\ar[r]^{\iota_{\RR,\UU}}& \frac{i^{-1}(\oO_{\UU})}{I_{\RR,\UU}^2} \ar[r]^{d}& 
\frac{i^{-1}(\Omega_{\UU})}{I_{\RR,\VV}\cdot i^{-1}(\Omega_{\UU})}\ar[r]&0
}
\end{equation}
where $\Phi: \UU\to\VV$ induces 
$$\Phi^{\#}: \Phi^{-1}(\oO_{\VV})\to \oO_{\UU}$$
on $\UU$ and 
\begin{equation}\label{i_inverse_Phi}
i^{-1}(\Phi^{\#}): j^{-1}(\oO_{\VV})|_{\RR}=i^{-1}\circ\Phi^{-1}(\oO_{\VV})\to i^{-1}(\oO_{\UU})
\end{equation}
is a morphism of sheaves of $R$-commutative algebras over $R$.  As $\Phi\circ i=j|_{\RR}$, then the above maps to 
$I_{\SS,\VV}|_{\RR}\to I_{\RR,\UU}$, and $I^2_{\SS,\VV}|_{\RR}\to I^2_{\RR,\UU}$.  Thus the map 
(\ref{i_inverse_Phi}) induces the morphism in the second column of (\ref{diagram_1}). Similarly, 
$d\Phi:  \Phi^{-1}(\Omega_{\VV})\to\Omega_{\UU}$ induces the third column of (\ref{diagram_1}).
\end{enumerate}
\end{thm}

In \cite{Joyce}, the sheaf $\FP_{X}$ for a $d$-critical scheme $X$ has a natural decomposition
$$\FP_{X}=\FP_{X}^0\oplus \kappa_{X}$$
where $\kappa_{X}$ is the constant sheaf on $X$, and $\FP_{X}^0\subset \FP_{X}$ is the kernel of the composition map:
$$\FP_{X}\stackrel{\beta_{X}}{\longrightarrow} \oO_{X}\stackrel{i_{X}^{\#}}{\longrightarrow} \oO_{X^{\red}}$$
where $\beta_{X}$ is the natural map.  

Let $\XX$ be a {\em stft} formal $R$-scheme with ideal of definition $I$. Then the underlying scheme is $(X, \oO_{\XX}/I)$.  From \cite[Proposition 1.1.26]{Fujiwara-Kato}, any locally noetherian formal scheme $\XX$ has a unique large ideal of definition $I$ of finite type such that the underlying scheme  $(X, \oO_{\XX}/I)$ is reduced.  So we do not pursue a decomposition of the sheaf $\FP_{\XX}$ here. 

\begin{defn}\label{defn_d_formal_scheme}
A {\em formal $d$-critical $R$-scheme} is a pair $(\XX,s)$, where $\XX$ is a {\em stft} formal $R$-scheme, and 
$s\in H^0(\FP_{\XX})$ is a section such that the following conditions are satisfied. 
For any $x\in\XX$, there exists an adic morphism of formal schemes
$$i: \RR\to \UU=\spf(R\{x_1,\cdots,x_n\})$$
such that $\RR\subset \XX$ is an open affine formal subscheme of the form
$$i(\RR)=\spf(R\{x_1,\cdots,x_n\}/(df))$$
for $f\in R\{x_1,\cdots,x_n\}$ and $\iota_{\RR,\UU}(s|_{\RR})=i^{-1}(f)+I_{\RR,\UU}^2$.  The quadruple 
$(\RR,\UU, f, i)$ is called a  {\em formal d-critical chart} of $(\XX,s)$. 
\end{defn}

\begin{example}\label{example_affine_d_formal_scheme}
Recall from \S \ref{sec_affine_formal_critical_scheme} we have the critical formal affine scheme 
$$\hat{f}: \hat{\XX}\to \spf(R).$$
We already have the embedding $i: \hat{\XX}\to\UU$.  Let us define the coherent sheaf 
$\FP_{\hat{\XX}}$.   Let 
$I_{\hat{\XX},\UU}$ be the formal ideal sheaf of $\hat{\XX}$ inside $\UU$, which is generated by $df$.
Then 
$$i^{-1}(f)\in H^0(i^{-1}(\oO_{\UU}))$$  and 
$d(i^{-1}(f))\in H^0(I_{\XX,\UU}\cdot i^{-1}(\Omega_{\UU}))\subset H^0(i^{-1}(\Omega_{\UU}))$.  So 
$\FP_{\hat{\XX}}$ is the sheaf 
$$\ker\left(\frac{i^{-1}(\oO_{\UU})}{I_{\RR,\UU}^2} \stackrel{d}{\longrightarrow}
\frac{i^{-1}(\Omega_{\UU})}{I_{\RR,\VV}\cdot i^{-1}(\Omega_{\UU})}\right)$$
The section $s\in H^0(\FP_{\hat{\XX}})$ is 
$s=f+I_{\hat{\XX},\UU}^2$. 
\end{example}

\subsubsection{Canonical bundles of formal $d$-critical schemes}\label{sec_formal_canonical_line_bundle}

\begin{prop}
Let $(\XX, s)$ be a formal $d$-critical $R$-scheme, and let 
$$\Phi: (\RR,\UU,f,i)\hookrightarrow (\SS, \VV, g, j)$$
be an embedding of formal critical charts on $(\XX,s)$.  Then for any $x\in\RR$, there exist open neighborhoods $\UU^\prime, \VV^\prime$ of $i(x), j(x)$ in $\UU, \VV$, such that $\Phi(\UU^\prime)\subseteq \VV^\prime$ and functions:
$$\alpha: \VV^\prime\to \UU; \quad \beta: \VV^\prime\to\spf(R\{x_1,\cdots,x_n\})$$
where $n=\dim(\VV)-\dim(\UU)$ such that
$$\alpha\times\beta: \VV^\prime\to \UU\times\spf(R\{x_1,\cdots,x_n\})$$
is an isomorphism and $\Phi|_{\UU^\prime}=\id|_{\UU^\prime}$, $\beta\circ\Phi|_{\UU^\prime}=0$,  and 
$$g|_{\VV^\prime}=f\circ\alpha+(x_1^2+\cdots+x_n^2)\circ\beta. $$
\end{prop}
\begin{proof}
Let us briefly prove this result.  Recall that an embedding of formal critical charts
$$\Phi: (\RR,\UU,f,i)\hookrightarrow (\SS, \VV, g, j)$$
is a diagram:
\[
\xymatrix{
\RR\ar@{^{(}->}[r]\ar[d]_{i}& \SS\ar[d]^{j}\\
\UU\ar@{^{(}->}[r]^{\Phi}& \VV
}
\]
such that
$$\UU=\spf(R\{x_1,\cdots,x_{m}\})\to\VV=\spf(R\{x_1,\cdots,x_{m}, y_1,\cdots, y_n\})$$
and  
$\Phi\circ i=j|_{\RR}$. Then the essential idea in the proof of \cite[Proposition 2.22]{Joyce} works well here. 
If we have functions $f\in A=R\{x_1,\cdots,x_{m}\}$ and $g\in B=R\{x_1,\cdots,x_{m},y_1,\cdots, y_n\}$ such that 
$I_{\RR,\UU}=I_{(df)}$, $I_{\SS,\VV}=I_{(dg)}$, then we have:
since $\Phi(\UU)\subseteq \VV$ with $y_1=\cdots=y_n=0$ and $i(\RR)\to j(\RR)\subset j(\SS)$, 
$I_{(df)}\cong I_{(dg)}|_{y_1=\cdots=y_n=0}$.  Hence
\begin{multline*}
\left(\frac{\partial f}{\partial x_a}(x_1,\cdots,x_m): a=1,\cdots,m\right)= \\
\left(\frac{\partial g}{\partial x_a}(x_1,\cdots,x_m,0,\cdots,0): a=1,\cdots,m; \frac{\partial g}{\partial y_b}(x_1,\cdots,x_m,0,\cdots,0): b=1,\cdots,n\right)
\end{multline*}
Then we can write $\frac{\partial g}{\partial y_b}(x_1,\cdots,x_m,0,\cdots,0)$ in terms of 
$\frac{\partial g}{\partial x_a}(x_1,\cdots,x_m,0,\cdots,0)$. 
Hence let $h\in B$ be a new function $\tilde{g}-f$, where $\tilde{g}$ is the function of $g$ with new coordinates
$$\tilde{x}_1, \cdots, \tilde{x}_m, \tilde{y}_1,\cdots, \tilde{y}_n$$
such that 
$$\tilde{x}_a=x_a+\sum_{b=1}^{m}A_{ab}(x_1,\cdots,x_m)y_b$$
$$\frac{\partial g}{\partial y_b}=\sum_{a=1}^{m}A_{ab}(x_1,\cdots,x_m)\cdot \frac{\partial g}{\partial x_a}(x_1,\cdots,x_m,0,\cdots,0).$$
Then $h(\tilde{x}_1, \cdots, \tilde{x}_m,0,\cdots,0)=0$ and $ \frac{\partial g}{\partial y_b}(\tilde{x}_1,\cdots,\tilde{x}_m,0,\cdots,0)=0$.  So 
$$h(\tilde{x}_1, \cdots, \tilde{x}_m, \tilde{y}_1,\cdots,\tilde{y}_n)=\sum \tilde{y}_b \tilde{y}_c Q_{bc}(\tilde{x}_1, \cdots, \tilde{x}_m, \tilde{y}_1,\cdots,\tilde{y}_n) $$
for some $Q_{bc}\in R\{\tilde{x}_1, \cdots, \tilde{x}_m, \tilde{y}_1,\cdots,\tilde{y}_n\}$. 
The thing is that $Q_{bc}$ is invertible so that we can have new algebras $A=R\{x_1,\cdots, x_m\}$ and 
$B=R\{x_1,\cdots, x_m, y_1,\cdots, y_n\}$ such that 
$$g=f+y_1^2+\cdots+y_n^2.$$
\end{proof}

Our main result in this section is that we also can define a {\em canonical line bundle} on $(\XX,s)$.

\begin{thm}\label{thm_canonical_line_bundle}
Let $(\XX,s)$ be a $d$-critical formal scheme. Then there exists a line bundle $K_{\XX,s}$,  the {\em canonical line bundle} on $(\XX,s)$,  which is natural up to isomorphism and has the following properties:
\begin{enumerate}
\item Let $(\RR,\UU,f,i)$ be a formal critical chart of $(\XX,s)$. Then there exists a natural isomorphism
$$\iota_{\RR,\UU,f,i}: K_{\XX,s}|_{\RR}\to i^\star(K_{\UU}^{\otimes 2})$$
Here $K_{\UU}=\Lambda^{\dim(\UU)}\Omega_{\UU}$ is the canonical sheaf on $\UU$ in the sense of 
\cite[\S 1.8]{Fujiwara-Kato}.
\item If $\Phi: (\RR,\UU,f,i)\hookrightarrow (\SS, \VV, g, j)$ is an embedding of formal critical charts on $(\XX,s)$, let 
$$J_{\Phi}: i^\star(K_{\UU}^{\otimes 2})\stackrel{\sim}{\rightarrow}j^\star(K_{\VV}^{\otimes 2})|_{\RR}$$
is the isomorphism defined in the proof. Then 
$$\iota_{\SS,\VV,g,j}=J_{\Phi}\circ \iota_{\RR,\UU,f,i}: K_{\XX,s}\to j^\star(K_{\VV}^{\otimes 2})|_{\RR}.$$
\item For each $x\in \XX$, we have 
$$k_x: K_{\XX,s}|_{x}\cong (\Lambda^{\top}\Omega_{x,\XX})^{\otimes 2}.$$
\item For a critical chart $(\RR,\UU,f,i)$, and $x\in\RR$, we have the following exact sequence
\begin{equation}\label{exact_point_cotangent}
0\rightarrow T_{x}\XX\stackrel{di|_{x}}{\longrightarrow} T_{i(x)}\UU
\stackrel{Hess_{i(x)}f}{\longrightarrow}\Omega_{i(x),\UU}\stackrel{di|_{x}^{*}}{\longrightarrow}\Omega_{x,\RR}\rightarrow 0
\end{equation}
and a commutative diagram:
\[
\xymatrix{
K_{\XX,s}|_{x}\ar[r]^{k_x}\ar[dr]_{\iota_{\RR,\UU,f,i}|_{x}}&(\Lambda^{\top}\Omega_{x,\XX})^{\otimes 2}\ar[d]^{\alpha_{x,\RR,\UU,f,i}}\\
&K_{\UU}^{\otimes 2}|_{x}
}
\]
Here $\alpha_{x,\RR,\UU,f,i}$ is the one given by (\ref{exact_point_cotangent}).
\end{enumerate}
\end{thm}
\begin{proof}
Let $\Phi: (\RR,\UU,f,i)\hookrightarrow (\SS, \VV, g, j)$ be an embedding of formal critical charts. 
Then let $N_{\UU/\VV}$ be the normal bundle of $\Phi(\UU)$ in $\VV$, we have 
$$0\rightarrow T_{\UU}\stackrel{d\Phi}{\longrightarrow} \Phi^\star(T_{\VV})\stackrel{\pi_{\UU/\VV}}{\longrightarrow} 
N_{\UU/\VV}\rightarrow 0$$
so that $i^\star(N_{\UU/\VV})$ is a vector bundle on $\RR\subset \XX$. Then there exists a unique quadratic form 
$q_{\UU\VV}\in H^0(S^2(N_{\UU/\VV}^*))$ on $i^\star(N_{\UU/\VV})$. These data will satisfy the properties in 
\cite[Proposition 2.25]{Joyce} since we choose our embeddings and critical charts as:
$$\UU=\spf(R\{x_1,\cdots, x_m\})$$ and 
$$\VV=\spf(R\{x_1,\cdots, x_m, y_1,\cdots, y_n\}).$$
For another embedding of formal critical charts
$$\Psi: (\SS,\VV,g,j)\hookrightarrow (\TT, \WW, h,k)$$
such that 
$$\Psi\circ\Phi: (\RR,\UU,f,i)\hookrightarrow (\TT, \WW, h,k)$$
is also an embedding of formal critical charts. We have the following diagram:
\[
\begin{CD}
&&&& 0 && 0\\
&& &&@VV{}V@VV{}V \\
0 @ >>>T_{\UU}@ >>> \Phi^\star(T_{\VV})@ >\pi_{\UU/\VV}>>
N_{\UU/\VV} @
>>> 0\\
&& @VV{\id}V@VV{}V@VV{\gamma_{\UU\VV\WW}}V \\
0@ >>> T_{\UU}@ >{}>>(\Psi\circ\Phi)^\star(T_{\WW}) @ >\pi_{\UU/\WW}>> N_{\UU/\WW}
@>>> 0\\
&& @VV{}V@VV{}V@VV{\delta_{\UU\VV\WW}}V \\
&& 0@ >{}>>\Phi^\star(N_{\VV/\WW}) @ >>> \Phi^\star(N_{\VV/\WW})
@>>> 0\\
&& &&@VV{}V@VV{}V \\
&&&& 0 && 0,
\end{CD}
\]
where $\gamma_{\UU\VV\WW}$ and $\delta_{\UU\VV\WW}$ are defined accordingly. 
Pulling back by $i^\star$:
\begin{equation}\label{exact_2}
0\rightarrow i^\star N_{\UU/\VV}|_{\RR}\stackrel{i^\star(\gamma)}{\longrightarrow} i^\star(N_{\UU/\WW})|_{\RR}
\stackrel{i^\star(\delta)}{\longrightarrow} j^\star(N_{\VV/\WW})|_{\RR}\rightarrow 0. 
\end{equation}
Hence we have 
\begin{equation}\label{formula_normal_bundle}
i^\star N_{\UU/\WW}\cong i^\star N_{\UU/\VV}\oplus i^\star N_{\VV/\WW}
\end{equation}
and 
$$q_{\UU\WW}\cong q_{\UU\VV}\oplus q_{\VV\WW}\oplus 0$$
in 
$$S^2(i^\star(N_{\UU/\WW}^*))\cong S^2(i^\star(N_{\UU/\VV}^*))|_{\RR}\oplus S^2(i^\star(N_{\UU/\WW}^*))|_{\RR}\oplus 
(i^\star(N_{\UU/\VV}^*)\otimes i^\star(N_{\VV/\WW}^*))|_{\RR}.$$

Now for any $\Phi: (\RR,\UU,f,i)\hookrightarrow (\SS, \VV, g, j)$, we have 
$$\rho_{\UU\VV}:  (i^\star K_{\UU})\otimes i^\star\Lambda^{n}(N_{\UU/\VV}^*)\stackrel{\sim}{\longrightarrow} j^\star(K_{\VV})|_{\RR}$$
by taking the top exterior powers in the dual of exact sequence.  Since $q_{\UU\VV}$ is a non-degenerate quadratic form, its determinant $\det(q_{\UU\VV})$ is a nonvanishing section of $i^\star(\Lambda^n N_{\UU/\VV})^{\otimes 2}$. Hence we can define the following isomorphism:
\begin{equation}\label{isomorphism_JPhi}
J_{\Phi}:  i^\star(K_{\UU}^{\otimes 2})\stackrel{\sim}{\rightarrow}j^\star(K_{\VV}^{\otimes 2})|_{\RR}
\end{equation}
by the diagram:
\[
\xymatrixcolsep{8pc}\xymatrix{
i^\star(K_{\UU}^{\otimes 2})\ar[r]^-{\id\otimes \det(q_{\UU\VV})|_{\RR}}\ar[dr]_{J_{\Phi}}&(\Lambda^{\top}i^\star(K_{\UU}^{\otimes 2})\otimes (\Lambda^{\top}N_{\UU/\VV}^*)^{\otimes 2}|_{\RR}\ar[d]^{\rho_{\UU\VV}^{\otimes 2}|_{\RR}}\\
&j^\star (K_{\VV}^{\otimes 2})|_{\RR}
}
\]
Then one can check that the isomorphism $J_{\Phi}$ is independent of the choice $\Phi$ and 
$$J_{\Psi|_{\RR}}\circ J_{\Phi}=J_{\Psi\circ \Phi}$$
if $\Psi: (\SS, \VV, g, j)\hookrightarrow (\TT, \WW, h, k)$ is another embedding of formal critical charts. 
The formula is Proposition 2.27 of Joyce in \cite{Joyce} and we check that it holds for formal critical charts. 

Then we construct the line bundle $K_{\XX,s}$ locally using $K_{\UU}^{\otimes 2}$. The existence of $K_{\XX,s}$ is from the following lemma:
\begin{lem}
Let $(\RR, \UU, f, i), (\SS, \VV, g, j)$ be formal critical charts on $(\XX, s)$. Then there is a unique isomorphism
$$J_{\RR,\UU,f,i}^{\SS,\VV,g,j}: i^\star(K_{\UU}^{\otimes 2})|_{\RR\cap\SS}\to j^\star(K_{\VV}^{\otimes 2})|_{\RR\cap\SS}$$
such that 
if $\Phi: (\RR, \UU, f, i)\hookrightarrow (\SS, \VV, g, j)$ and $\Psi: (\SS, \VV, g, j)\hookrightarrow (\TT, \WW, h, k)$ are embeddings of formal critical charts, then 
$$J_{\RR,\UU,f,i}^{\SS,\VV,g,j}|_{\RR\cap\SS}=J_{\Psi}^{-1}\circ J_{\Phi}|_{\RR\cap\SS}$$
and 
$$J_{\RR,\UU,f,i}^{\RR,\UU,f,i}=\id; \quad J_{\SS,\VV,g,j}^{\RR,\UU,f,i}=(J_{\RR,\UU,f,i}^{\SS,\VV,g,j})^{-1}$$
$$J_{\SS,\VV,g,j}^{\TT,\WW,h,k}\circ J_{\RR,\UU,f,i}^{\SS,\VV,g,j}|_{\RR\cap\SS\cap\TT}=J_{\RR,\UU,f,i}^{\TT,\WW,h,k}|_{\RR\cap\SS\cap\TT}.$$
\end{lem}
\begin{proof}
We use a similar result as in Theorem 2.20 of \cite{Joyce}: for any formal $d$-critical charts 
$(\RR,\UU,f,i)$ and $(\SS, \VV, g, j)$,  let $x\in\RR\cap\SS$, then $\RR\cap\SS$ is also a formal scheme of the form 
$\spf(A_{\RR\SS})$, where $A_{\RR\SS}$ is a {\em stft} $R$-algebra. Then we have a morphism
$$\spf(A_{\RR\SS})\to \SS=\spf(R\{x_1,\cdots,x_m,y_1,\cdots,y_n)/(I_{(dg)})$$
as formal schemes.  Since also 
$\spf(A_{\RR\SS})\hookrightarrow \RR=\spf(R\{x_1,\cdots,x_m)/(I_{(df)})$, then there exists 
\[
\xymatrix{
\RR^\prime=\RR\cap\SS\ar@{^{(}->}[r]\ar@{^{(}->}[d]& \SS\ar@{^{(}->}[d]\\
\UU^\prime\ar[r]^{\Theta}& \VV
}
\]
where $\UU^\prime$ is another smooth formal scheme $\spf(R\{x_1,\cdots, x_r\})$ with $r<m$. 
Then similar to the arguments as in Theorem 2.20 of Joyce \cite{Joyce} we have
$$f^\prime=f|_{\UU^\prime}-g\circ\Theta\in I_{\RR^\prime,\UU^\prime}^2$$
and we can construct a formal critical chart 
$(\TT,\WW,h,k)$ such that 
$$\WW=\VV\times\spf(R\{z_1,\cdots,z_n, t_1,\cdots,t_n\}); \TT=\SS; k=j\times\{0\}$$ and 
$\TT=\SS\hookrightarrow \VV\times \spf(R\{z_1,\cdots,z_n, t_1,\cdots,t_n\})=\WW$
by 
$$\Phi(u)=(\Theta(u), r_1(u),\cdots, r_n(u), s_1(u),\cdots, s_n(u));\quad \Psi(v)=(v,0)$$
and 
$$f^\prime=g\circ\Theta+r_1s_1+\cdots+r_ns_n$$
$$h(v, (z_1,\cdots,z_n, t_1,\cdots,t_n))=g(v)+z_1t_1+\cdots+z_nt_n.$$
Then using this result we can prove that 
$$J_{\Psi}^{-1}\circ J_{\Phi}|_{\RR\cap\SS}$$
is independent to the data chosen in the theorem as in Lemma 6.1 of \cite{Joyce}. So these isomorphisms give the conditions in the theorem. 
\end{proof}

Then it is routine check that $K_{\XX,s}$ satisfies the conditions in the theorem.  We omit the details. Thus the proof of  Theorem \ref{thm_canonical_line_bundle} is complete. 
\end{proof}

\begin{defn}(Orientation of $d$-critical formal scheme)
Let $(\XX,s)$ be a formal $d$-critical scheme and $K_{\XX,s}$ the canonical line bundle. 
An {\em orientation} on $(\XX,s)$ is a choice of square root $K_{\XX,s}^{\frac{1}{2}}$ for $K_{\XX,s}$. That is a line bundle 
$\lL$ over $\XX$ such that $\lL^{\otimes 2}\cong K_{\XX,s}$. We call $(\XX,s)$ the {\em oriented $d$-critical formal scheme} if 
$K_{\XX,s}^{\frac{1}{2}}$ exists. 
\end{defn}

As in \cite{Joyce}, an orientation on $(\XX,s)$ can be explained by principal $\zz_2$-bundles. 
\begin{defn}
Let $(\XX,s)$ be a formal $d$-critical scheme. For each embedding of formal critical charts:
$$\Phi: (\RR,\UU,f,i)\hookrightarrow (\SS, \VV, g, j),$$
let $\pi_{\Phi}: P_{\Phi}\to \RR$ be the bundle of square roots of the isomorphism $J_\Phi$ in (\ref{isomorphism_JPhi}). A local section $s_{\alpha}: \RR\to P_{\Phi}$ corresponds to a local isomorphism 
$$\alpha: i^\star(K_{\UU})|_{\RR}\to j^\star(K_{\VV})|_{\RR}$$
with $\alpha\otimes\alpha=J_{\Phi}$. Also $P_{\Phi}$ is independent of the choice of $\Phi$. 

If $\Psi: (\SS, \VV, g, j)\hookrightarrow (\TT,\WW,h,k)$ is another embedding of formal critical charts, then there exists an isomorphism
$$\Pi_{\Psi,\Phi}: P_{\Psi\circ\Phi}\stackrel{\sim}{\rightarrow}P_{\Psi}|_{\RR}\otimes_{\zz_2}P_{\Phi},$$
such that  if
$$\alpha: i^\star(K_{\UU})\to j^\star(K_{\VV})|_{\RR};  \quad \beta: j^\star(K_{\VV})|_{\RR}\to k^\star(K_{\WW})|_{\RR}; \quad \gamma: i^\star(K_{\UU})\to k^\star(K_{\WW})|_{\RR}$$
are the local isomorphisms such that
$$\alpha\otimes=J_{\Phi}; \quad  \beta\otimes\beta=J_{\Psi}|_{\RR};\quad \gamma\otimes\gamma=J_{\Psi\circ \Phi}$$
correspond to local sections
$s_\alpha: \RR\to P_{\Phi}$,  $s_\beta: \RR\to P_{\Psi}|_{\RR}$, and $s_\gamma: \RR\to P_{\Psi\circ\Phi}|_{\RR}$. 
Then $\Pi_{\Psi,\Phi}(s_\gamma)=s_{\beta}\otimes_{\zz_2}s_{\alpha}$ if and only if $\gamma=\beta\circ\alpha$. 
\end{defn}

\begin{rmk}
Actually using the theory of gerbes an orientation on $(\XX,s)$ is a $B\zz_2$-gerbe over $(\XX,s)$.
\end{rmk}

Let $K_{\XX,s}^{\frac{1}{2}}$ be an orientation on $(\XX,s)$. Let $(\RR,\UU,f,i)$ be a formal critical chart, define a principal $\zz_2$-bundle 
$$\pi_{\RR,\UU,f,i}: Q\to \RR$$
to be the bundle of square root of $\iota_{\RR,\UU,f,i}: K_{\XX,s}|_{\RR}\stackrel{\sim}{\rightarrow}i^\star(K_{\UU})^{\otimes 2}$. 
On an embedding 
$\Phi: (\RR, \UU, f, i)\hookrightarrow (\SS,\VV,g,j)$ we have principal $\zz_2$-bundles 
$$\pi_{\Phi}: P_\Phi\to\RR$$
$$\pi_{\RR,\UU,f,i}: Q_{\RR\UU}\to \RR$$
$$\pi_{\SS,\VV,g,j}: Q_{\SS\VV}\to \RR.$$
Then there exists a natural isomorphism 
\begin{equation}\label{isomorphism_principal_bundle}
\Lambda_{\Phi}:  Q_{\SS\VV}|_{\RR}\stackrel{\sim}{\rightarrow}P_{\Phi}\otimes_{\zz_2}Q_{\RR\UU}
\end{equation}
by local isomorphisms:
$$\alpha: i^\star(K_{\UU})|_{\RR}\to j^\star(K_{\VV})|_{\RR}; \quad \beta: K_{\XX,s}^{\frac{1}{2}}|_{\RR}\to i^\star(K_{\UU})|_{\RR}; \quad \gamma: K_{\XX,s}^{\frac{1}{2}}|_{\RR}\to j^\star(K_{\VV})|_{\RR}$$
with $\alpha\otimes\alpha=i^\star(J_\Phi)$, $\beta\otimes\beta=\iota_{\RR,\UU,f,i}$ corresponding to local sections 
$s_\alpha, s_\beta, s_\gamma$. 
If $\Psi: (\SS,\VV,g,j)\hookrightarrow (\TT,\WW,h,k)$ is another embedding of formal critical charts then we have the following commutative diagram:
\begin{equation}\label{diagram_2}
\xymatrixcolsep{5pc}\xymatrix{
Q_{\TT,\WW,h,k}|_{\RR}\ar[r]^{\Lambda_{\Psi\circ\Phi}}\ar[d]^{\Lambda_{\Psi}|_{\RR}}& P_{\Psi\circ\Phi}\otimes_{\zz_2}Q_{\RR,\UU,f,i}\ar[d]^{\Pi_{\Psi\Phi}\otimes\id}\\
(P_{\Psi}\otimes_{\zz_2}Q_{\SS,\VV,g,j})|_{\RR}\ar[r]^{\id\otimes \Lambda_{\Phi}}&P_{\Psi}|_{\RR}\otimes_{\zz_2}P_{\Phi}\otimes_{\zz_2}Q_{\RR,\UU,f,i}
}
\end{equation}
So the following proposition is straightforward.
\begin{prop}
Let $(\XX,s)$ be a formal $d$-critical scheme. Then there is an isomorphism between isomorphism classes of orientations $K_{\XX,s}^{\frac{1}{2}}$ on $(\XX,s)$, and isomorphism classes of the data:
\begin{enumerate}
\item each formal critical chart $(\RR,\UU,f,i)$, and  a choice of principal $\zz_2$-bundle 
$\pi: Q_{\RR,\UU,f,i}\to\RR$;
\item each embedding of formal critical charts 
$$\Phi: (\RR, \UU, f, i)\hookrightarrow (\SS, \VV, g, j)$$
a choice of $\Lambda_{\Phi}$ in  (\ref{isomorphism_principal_bundle}) such that the diagram (\ref{diagram_2}) holds.
\end{enumerate}
\end{prop}

\subsection{Proof of main theorem}\label{sec_Proof_main_Theorem}

We prove Theorem \ref{thm_existence_sheaf_PX}. 
We mainly check that the proof in \cite[\S 3.1, \S3.2]{Joyce} work for formal schemes, since a {\em stft} $R$-formal scheme 
$\XX$ is covered by affine stft $R$-formal schemes of the type $\spf(A)$.  First let 
$(\RR,\UU,f,i)$ be a formal critical chart as in the theorem. Then we have:
$$0\rightarrow \mathscr{K}_{\RR,\UU,i}|_{\RR}\stackrel{k_{\RR,\UU}}{\longrightarrow} \frac{i^{-1}(\oO_{\UU})}{I_{\RR,\UU}^2}
\stackrel{d}{\longrightarrow} \frac{i^{-1}(\Omega_{\UU})}{I_{\RR,\UU}\cdot i^{-1}(\Omega_{\UU})}$$
where $\mathscr{K}_{\RR,\UU,i}$ is the kernel.  
Then $\mathscr{K}_{\RR,\UU,i}$ is a sheaf of $R$-commutative algebras. Similar arguments as in \cite[\S 3.1]{Joyce} works for any morphism 
$$\Phi: (\RR,\UU,i)\to(\SS,\VV,j)$$
such that 
$\RR\subseteq \SS$, we have 
$$\Phi^\star: \mathscr{K}_{\SS,\VV,j}|_{\RR}\to \mathscr{K}_{\RR,\UU,i}$$
as a morphism of sheaves of $R$-commutative algebras. If $\Psi: (\SS,\VV,j)\to (\TT,\WW,k)$ is another morphism of formal critical charts, then 
$$(\Psi\circ\Phi)^\star=\Phi^\star\circ \Psi^\star|_{\RR}. $$
Also Lemma 3.1 of \cite{Joyce} is true for formal critical charts:  if $\Phi, \widetilde{\Phi}:  (\RR,\UU,i)\to(\SS,\VV,j)$ are morphisms of triples, then 
$$\Phi^\star=\widetilde{\Phi}^\star: \mathscr{K}_{\SS,\VV,j}|_{\RR}\to \mathscr{K}_{\RR,\UU,i}.$$
The main point in the proof is that for any 
local section $\alpha\in \mathscr{K}_{\SS,\VV,j}|_{\RR}$ around $x\in\RR$, 
$\alpha=f+(I_{\SS,\VV}^2)$.  Then 
$$\Phi^\star(\alpha)=f\circ\Phi+(I_{\RR,\UU}^2)$$
$$\widetilde{\Phi}^\star(\alpha)=f\circ\widetilde{\Phi}+(I_{\RR,\UU}^2).$$
We can write down ($\VV=\spf(R\{x_1,\cdots,x_m)$)
\begin{multline}\label{equation_difference}
f\circ\widetilde{\Phi}-f\circ\Phi=\sum_{a=1}^{m}(\frac{\partial f}{\partial x_i}\circ\Phi)\cdot (x_a\circ\widetilde{\Phi}-x_a\circ\Phi)\\
+\sum_{a,b=1}^{m}A_{ab}(x_a\circ\widetilde{\Phi}-x_a\circ\Phi)\cdot (x_b\circ\widetilde{\Phi}-x_b\circ\Phi)
\end{multline}
where $A_{ab}$ are functions. Then one can prove that each factor $(\cdot)$ in (\ref{equation_difference}) lies in $I_{\RR,\UU}$. So 
$$\Phi^\star(\alpha)=f\circ\Phi+(I_{\RR,\UU})^2=f\circ\widetilde{\Phi}+(I_{\RR,\UU})^2=\widetilde{\Phi}^\star(\alpha).$$
Similar as before, let $(\RR,\UU,i), (\SS,\VV,j)$ be formal critical charts. 
Then for any $x\in\RR\cap\SS$, there exists a $(\RR^\prime, \UU^\prime, i^\prime)$ such that 
$\RR^\prime\subseteq \rr\cap\SS$ and 
$\Phi: (\RR^\prime, \UU^\prime, i^\prime)\hookrightarrow (\SS,\VV,j)$ is an embedding of formal critical charts. 
We can show that 
$$I_{\SS,\VV,j}^{\RR,\UU,i}: \mathscr{K}_{\SS,\VV,j}|_{\RR\cap\SS}\stackrel{\sim}{\rightarrow}\mathscr{K}_{\RR,\UU,i}|_{\RR\cap\SS}$$
is an isomorphism, and 
$$I_{\SS,\VV,j}^{\RR,\UU,i}|_{\RR\cap\SS\cap\TT}\circ I^{\SS,\VV,j}_{\TT,\WW,k}|_{\RR\cap\SS\cap\TT}
=I_{\TT,\WW,k}^{\RR,\UU,i}|_{\RR\cap\SS\cap\TT}$$
for any other $(\TT,\WW,k)$. Hence we can define 
$$\FP_{\XX}|_{(\RR,\UU,i)}=\mathscr{K}_{\RR,\UU,i}.$$
Then these data glue together to give a sheaf $\FP_\XX$ of $R$-commutative algebras over $\XX$.

\subsection{$d$-critical non-archimedean $\kk$-analytic spaces}\label{sec_d_critical_analytic_space}

In this section we briefly recall that there exists a notion of $d$-critical rigid varieties and non-archimedean 
$\kk$-analytic spaces by taking the generic fiber $(\cdot)_{\rig}, (\cdot)_\eta$ to a $d$-critical formal scheme. 

First, let $\XX$ be a {\em stft} formal $R$-scheme.  Then from \cite{Nicaise}, \cite{Ber1}, recalled in Remark \ref{rmk_functor_generic_fiber}, there is a functor
$$(\cdot)_\eta: \Fsch_{R}\to \An_\kk; \quad (\cdot)_{\rig}: \Fsch_{R}\to \Rig_\kk$$
from the category of {\em stft} formal $R$-schemes to the category of paracompact $\kk$-analytic spaces by taking the generic fiber of formal schemes. If $\XX=\spf(A)$ for a {\em stft} $t$-adic $R$-algebra $A$, then 
$$(\XX)_\eta=\XX_\eta=\SP_{B}(A\otimes_{R}\kk)$$ 
is the Berkovich spectrum of semi-norms of $A\otimes_{R}\kk$. If we take take $A\otimes_{R}\kk$ as the Tate's $\kk$-affinoid algebra then we get the rigid variety $\XX_\eta$.

A general {\em stft} formal $R$-scheme is covered by a finite covering $\{\XX_i\}$ of open affine formal $R$-schemes of the form 
$\spf(A_i)$.  The intersections $\XX_{ij}=\XX_i\cap\XX_j$ are separate formal $R$-schemes, and $\XX_{ij,\eta}$ are closed analytic domains of $\XX_{i,\eta}$.  Gluing $\XX_{i,\eta}$ along $\XX_{ij,\eta}$ we get a paracompact $\kk$-analytic space 
$\XX_\eta$. 

\begin{prop}
Let $X$ be a quasi-compact $\kk$-analytic space. Then there exists a coherent sheaf $\FP_{X}$, unique up to isomorphism, and is characterized by the properties in Theorem \ref{thm_existence_sheaf_PX}.
\end{prop}
\begin{proof}
For a quasi-compact $\kk$-analytic space $X$, from \cite{BW}, \cite{Raynaud}, there is a formal model $\XX$ for $X$, i.e., there exists a {\em stft} formal $R$-scheme $\XX$ such that $\XX_\eta\cong X$. Hence from \cite[Proposition 2.6]{Nicaise}, there exists a unique exact funcotr 
$$(\cdot)_\eta: (\Coh_{\XX})\to \Coh(\XX_\eta)$$
from the category of coherent sheaves over $\XX$ to the category of coherent sheaves over $\XX_\eta$.  Then we get the coherent sheaf $\FP_{X}=(\FP_{\XX})_\eta$.
\end{proof}

\begin{defn}
A $d$-critical $\kk$-analytic space (or virtual critical $\kk$-analytic space of Kiem-Li) is a pair $(X,s)$, where $X$ is a quasi-compact $\kk$-analytic space, and $s\in H^0(\FP_{X})$ is a section, such that for any $x\in X$, there exists a local embedding 
$$R\stackrel{i}{\rightarrow} U=\SP_B(\kk\{x_1,\cdots,x_m\})$$
such that $R\hookrightarrow U$ is a closed embedding and 
$$i(R)=\SP_B(\kk\{x_1,\cdots,x_m\}/(df_t)),$$ 
for $f_t=f-t\in \kk\{x_1,\cdots,x_m\}$, and $\iota_{R,U}(s|_{R})=i^{-1}(f_t)+I_{R,U}^2$. 
\end{defn}

Similar to Theorem \ref{thm_canonical_line_bundle}, we have
\begin{prop}\label{prop_canonical_line_bundle_analytic_space}
Let $(X,s)$ be a $d$-critical $\kk$-analytic space. Then there exists a line bundle $K_{X,s}$,  the {\em canonical line bundle} on $(X,s)$,  which is natural up to isomorphism and satisfies similar properties in Theorem \ref{thm_canonical_line_bundle} by replacing the formal critical charts by the analytic critical charts.
\end{prop}
\begin{proof}
Since for a  $d$-critical $\kk$-analytic space $(X,s)$, there is a formal model $(\XX,s)$ which is a $d$-critical formal scheme.  Then the properties in Theorem \ref{thm_canonical_line_bundle} are from the above generic fiber functor $(\cdot)_\eta$. 
\end{proof}

\begin{defn}
An oriented $d$-critical $\kk$-analytic space $(X,s)$ is a $d$-critical $\kk$-analytic space such that there is a square root for the canonical line bundle $K_{X,s}$. 
\end{defn}

The following proposition is from the above construction of functors.
\begin{prop}\label{prop_generic_fiber_of_d_formal_scheme}
Let $(\XX,s)$ be an oriented $d$-critical formal scheme.  Then the generic fiber $(\XX_\eta, s_\eta)$ is an oriented  $d$-critical non-archimedean $\kk$-analytic space. 
\end{prop}

Similarly, 
\begin{prop}
Let $(\XX,s)$ be a $d$-critical formal scheme.  Then the special fiber $(\XX_s, s)$ is a $d$-critical  scheme in \cite{Joyce}.
\end{prop}
\begin{proof}
From Lemma \ref{lem_functor_special_fiber}, there is a special fiber functor
$$(\cdot)_s: \Fsch_{R}\to \Sch_\kappa$$ and an exact functor
$$(\cdot)_s: \Coh_{\XX}\to \Coh_{\XX_s}$$
such that $(\FP_\XX)_s=\FP_X$, the unique sheaf in \cite{Joyce}.  The canonical section $s\in H^0(\FP_\XX)$ will induce the canonical section $s\in H^0(\FP_{\XX_s})$.  Then the result follows. 
\end{proof}

\begin{example}
From \cite[\S 1.5]{Ber1} a quasi-compact $\kk$-analytic space $X$ is smooth if for any connected strictly affinoid domain 
$V$,  the sheaf of differentials $\Omega_{V}$ is locally free of rank $\dim(V)$ and there is no boundary.  
Then $\FP_{X}=0$ and 
the pair $(X, 0)$ is a trivial $d$-critical $\kk$-analytic space. 
\end{example}

\begin{example}
If $(X,s)$ is a $d$-critical $\kk$-analytic space such that it is the generic fiber of a $d$-critical formal $R$-scheme $(\XX,s)$ of the form $\XX=\spf(R\{x_1,\cdots,x_m\}/(I_{df}))$ in Example \ref{example_affine_d_formal_scheme}, where $s=f+I_{\XX,\UU}^2$. Then 
$$X=\SP_B(\kk\{x_1,\cdots,x_m\}/(I_{(df_t)}))$$
where $s=f_t+I_{X,U}^2$ and $U=\SP_B(\kk\{x_1,\cdots,x_m\})$ is the unit closed disc. 
\end{example}

\begin{rmk}
One also can define a canonical line bundle $K_{X,s}$ on any $d$-critical $\kk$-analytic space $(X,s)$.  Micmicing the construction  in \S \ref{sec_formal_canonical_line_bundle} or by taking the generic fiber $(\cdot)_\eta$ we get a canonical line bundle over $\XX_\eta=X$. 
\end{rmk}

\subsection{$\Gm$-equivariant $d$-critical formal schemes}\label{sec_Gm_equivariant_formal_scheme_analytic_space}

We talk about the torus action on formal schemes and $\kk$-analytic spaces. 
\begin{defn}
A torus $\Gm$-action of the ring $R$ is called {\em nice} if the induced action on the residue is trivial. 
\end{defn}

\begin{defn}
Let $(\XX,s)$ ($(X,s)$) be a $d$-critical $R$-formal scheme  ($d$-critical $\kk$-analytic space). 
A $\Gm$-action 
$$\mu: \Gm\times \XX\to\XX;\quad (\Gm\times X\to X)$$
is $\Gm$-invariant if $\mu(\gamma)^\star(s)=s$ for any $\gamma\in\Gm$, and 
$s\in H^2(\FP_{\XX})$ ($s\in H^2(\FP_{X})$), or equivalently 
$\mu^\star(s)=\pi_{\XX}^\star(s)$, where $\pi: \Gm\times\XX\to\XX$ ($\Gm\times X\to X$) is the projection. 
\end{defn}

\begin{defn}
The $\Gm$-action on a formal scheme $\XX$ is called {\em good} if any orbit is contained in an affine formal subscheme of 
$\XX$. Equivalently, there exists an open cover of formal subschemes which are $\Gm$-invariant. 
\end{defn}

\begin{prop}
Let $(\XX,s)$ be a formal $d$-critical $R$-scheme which is $\Gm$-invariant under the $\Gm$-action. Then 
\begin{enumerate}
\item If the action $\mu$ is good, then any $x\in \XX$ there exists a $\Gm$-invariant formal critical chart $(\RR,\UU,f,i)$ on 
$(\XX,s)$ such that $\dim(T_{x}\XX)=\dim(U)$;
\item If for all $x\in\XX$ we have a $\Gm$-invariant formal critical chart $(\RR,\UU,f,i)$, then the action  $\mu$ is good. 
\end{enumerate}
\end{prop}
\begin{proof}
Let $(\XX,s)$ be a formal $d$-critical $R$-scheme with a good  $\Gm$-action.  That the action is good  means  that there is a 
$\Gm$-invariant affine $R$-formal scheme $\RR^\prime\hookrightarrow \XX$, such that 
$\RR^\prime=\spf(A/I)$.  Performing the same argument as in the proof of Theorem 2.43 of \cite{Joyce}, we can take a $\Gm$-equivalent formal critical chart $(\RR,\UU,f,i)$ such that 
$\dim(\UU)=\dim(T_x\XX)$.  This is $(1)$. $(2)$ is clear. 
\end{proof}

\begin{prop}\label{prop_Gm_fixed_d_critical_formal_scheme}
Let $(\XX,s)$ be a formal $d$-critical $R$-scheme which is $\Gm$-invariant under the $\Gm$-action. Let $\XX^{\Gm}$ be the fixed formal subscheme.   Then the fixed subscheme
$\XX^{\Gm}$ inherits a formal $d$-critical scheme structure 
$(\XX^\Gm, s^{\Gm})$, where $s^\Gm=i^\star(s)$ and 
$i: \XX^\Gm\hookrightarrow \XX$ is the inclusion map. 
\end{prop}
\begin{proof}
This result is from a $\Gm$-equivariant version of the following result. 
Let $(\RR,\UU,f,i), (\SS,\VV,g,j)$ be $\Gm$-equivariant formal critical charts on $(\XX,s)$. Then for $x\in\RR\cap\SS$, there exist $\Gm$-equivariant charts  $(\RR^\prime,\UU^\prime,f^\prime,i^\prime), (\SS^\prime,\VV^\prime,g^\prime,j^\prime)$ and a 
$\Gm$-equivariant formal chart $(\TT,\WW,h,k)$, and $\Gm$-equivariant embeddings
$$\Phi: (\RR^\prime,\UU^\prime,f^\prime,i^\prime)\hookrightarrow (\TT,\WW,h,k)$$
$$\Psi: (\SS^\prime,\VV^\prime,g^\prime,j^\prime)\hookrightarrow (\TT,\WW,h,k).$$
Then let 
$f^{\Gm}=f|_{\UU^{\Gm}}$ and $i^{\Gm}=i|_{\RR^{\Gm}}$, we have a formal critical chart 
$(\RR^\Gm, \UU^\Gm, f^{\Gm}, i^{\Gm})$ on $(\XX^\Gm, s^{\Gm})$.  Hence by gluing $(\XX^\Gm, s^{\Gm})$ is a formal $d$-critical $R$-scheme. 
\end{proof} 

A similar notion of $\Gm$-invariant $d$-critical $\kk$-analytic space and $\Gm$-equivariant critical chart $(R, U, f,i)$ can be similarly defined and we omit the details.

\section{Motivic localization of Donaldson-Thomas invariants}\label{sec_motivic_localization}

\subsection{Grothendieck group of varieties.}\label{Grothendieck:ring}

In this section we briefly review the Grothendieck group of varieties. 
Let $S$ be an algebraic variety over $\kappa$. Let $\Var_{S}$ be the category of 
$S$-varieties.

Let $K_0(\Var_{S})$ be the Grothendieck group of $S$-varieties.  By definition $K_0(\Var_{S})$ 
is an abelian group with generators given by all the varieties $[X]$'s, where $X\rightarrow S$ are $S$-varieties,  and the relations are $[X]=[Y]$, if $X$ is isomorphic to $Y$, and 
$[X]=[Y]+[X\setminus Y]$ if $Y$ is a Zariski closed subvariety of $X$.
Let $[X],  [Y]\in K_0(\Var_{S})$,  and define $[X][Y]=[X\times_{S} Y]$.  Then 
we have a product on $K_0(\Var_{S})$. 
Let $\mathbb{L}$ represent the class of $[\mathbb{A}_{\kappa}^{1}\times S]$.
Let $\mathcal{M}_{S}=K_0(\Var_{S})[\mathbb{L}^{-1}]$
be the ring by inverting the class $\mathbb{L}$ in the ring $K_0(\Var_{S})$.

If $S$ is a point $\spec (\kappa)$, we write $K_0(\Var_{\kappa})$ for the Grothendieck ring of $\kappa$-varieties.
One can take the map $\Var_{\kappa}\longrightarrow K_0(\Var_{\kappa})$ to be the universal Euler characteristic.
After inverting the class $\mathbb{L}=[\mathbb{A}_{\kappa}^{1}]$, we get the ring $\mathcal{M}_{\kappa}$.

We introduce the equivariant Grothendieck group defined in \cite{DL}.
Let $\mu_n$ be the cyclic group of order $n$, which can be taken as the algebraic variety
$\spec (\kappa[x]/(x^n-1))$. Let $\mu_{md}\longrightarrow \mu_{n}$ be the map $x\mapsto x^{d}$. Then 
all the groups $\mu_{n}$ form a projective system. Let 
$$\varprojlim_{n}\mu_{n}$$
be the direct limit.

Suppose that $X$ is a $S$-variety. The action $\mu_{n}\times X\longrightarrow X$ is called a $good$ 
action if  each orbit is contained in an affine subvariety of $X$.  A good $\hat{\mu}$-action on $X$ is an action of $\hat{\mu}$ which factors through a good $\mu_n$-action for some $n$.

The $equivariant ~Grothendieck~ group$ $K^{\hat{\mu}}_0(\Var_{S})$ is defined as follows:
The generators are $S$-varieties $[X]$ with a good $\hat{\mu}$-action; and the relations are:
$[X,\hat{\mu}]=[Y,\hat{\mu}]$ if $X$ is isomorphic to $Y$ as $\hat{\mu}$-equivariant $S$-varieties,  
and $[X,\hat{\mu}]=[Y,\hat{\mu}]+[X\setminus Y, \hat{\mu}]$ if $Y$ is a Zariski closed subvariety
of $X$ with the $\hat{\mu}$-action induced from that on $X$,  if $V$ is an affine variety with a good 
$\hat{\mu}$-action, then $[X\times V,\hat{\mu}]=[X\times \mathbb{A}_{\kappa}^{n},\hat{\mu}]$.  The group 
$K^{\hat{\mu}}_0(\Var_{S})$ has a ring structure if we define the product as the fibre product with the good $\hat{\mu}$-action.  Still we let $\mathbb{L}$  represent the class $[S\times \mathbb{A}_{\kappa}^{1},\hat{\mu}]$ and let $\mathcal{M}_{S}^{\hat{\mu}}=K^{\hat{\mu}}_0(\Var_{S})[\mathbb{L}^{-1}]$ be the ring obtained from $K^{\hat{\mu}}_0(\Var_{S})$ by inverting the class $\mathbb{L}$.

If $S=\spec(\kappa)$, then we write $K^{\hat{\mu}}_0(\Var_{S})$ as $K^{\hat{\mu}}_0(\Var_{\kappa})$, and $\mathcal{M}_{S}^{\hat{\mu}}$ as $\mathcal{M}_{\kappa}^{\hat{\mu}}$.  Let  $s\in S$ be a geometric point. Then we have natural maps $K^{\hat{\mu}}_0(\Var_{S})\longrightarrow K^{\hat{\mu}}_0(\Var_{\kappa})$ and $\mathcal{M}_{S}^{\hat{\mu}}\longrightarrow \mathcal{M}_{\kappa}^{\hat{\mu}}$ given by the correspondence
$[X,\hat{\mu}]\mapsto [X_s,\hat{\mu}]$.

Let $S$ be a scheme.  Following \cite{BJM}, we need to define a new product $\odot$ on $\mM_{S}^{\hat{\mu}}$.
The following definition is due to \cite[Definition 2.3]{BJM}. 
\begin{defn}
Let $[X, \widehat{\sigma}], [Y,\widehat{\tau}]$ be two elements in $K_0^{\hat{\mu}}(\Var_{S})$ or $\mM_{S}^{\hat{\mu}}$. 
Then there exists $n\geq 1$ such that the $\hat{\mu}$-actions $\widehat{\sigma}, \widehat{\tau}$ on $X,Y$ factor through $\mu_n$-actions
$\sigma_n, \tau_n$.  Define $J_n$ to be the Fermat curve
$$J_n=\{(t,u)\in (\aaa^1\setminus \{0\})^2: t^n+u^n=1\}.$$
Let $\mu_n\times\mu_n$ act on $J_n\times(X\times_{S}Y)$ by
$$(\alpha,\alpha^\prime)\cdot ((t,u),(v,w))=((\alpha\cdot t, \alpha^\prime\cdot u), (\sigma_n(\alpha)(v), \tau_n(\alpha^\prime)(w))).$$
Write $J_n(X,Y)=(J_n\times (X\times_{S}Y))/(\mu_n\times\mu_n)$ for the quotient $\kappa$-scheme, and 
define a $\mu_n$-action $v_n$ on $J_n(X,Y)$ by
$$v_n(\alpha)((t,u), v,w)(\mu_n\times\mu_n)=((\alpha\cdot t, \alpha\cdot u),v,w)(\mu_n\times\mu_n).$$
Let $\hat{v}$ be the induced good $\hat{\mu}$-action on $J_n(X,Y)$, and set
$$[X, \widehat{\sigma}]\odot [Y,\widehat{\tau}]=(\ll-1)\cdot [(X\times_{S}Y/\mu_n, \hat{\iota})]-[J_n(X,Y),\hat{v}]$$
in  $K_0^{\hat{\mu}}(\Var_{S})$ or $\mM_{S}^{\hat{\mu}}$. This defines a commutative, associative product on  $K_0^{\hat{\mu}}(\Var_{S})$ or $\mM_{S}^{\hat{\mu}}$.
\end{defn}

Consider the Lefschetz motive $\ll=[\aaa^1_{\kappa}]$. As in \cite{BJM}, we define 
$\ll^{\frac{1}{2}}$ in $K_0^{\hat{\mu}}(\Var_{S})$ or $\mM_{S}^{\hat{\mu}}$ by:
$$\ll^{\frac{1}{2}}=[S,\hat{\iota}]-[S\times\mu_2,\hat{\rho}],$$
where $[S,\hat{\iota}]$ with trivial $\hat{\mu}$-action $\hat{\iota}$ is the identity in  $K_0^{\hat{\mu}}(\Var_{S})$ or $\mM_{S}^{\hat{\mu}}$,
and $S\times\mu_2$ is the two copies of $S$ with the nontrivial $\hat{\mu}$-action $\hat{\rho}$ induced by the left action of $\mu_2$ on itself, exchanging the two copies of $S$.  Then $\ll^{\frac{1}{2}}\odot\ll^{\frac{1}{2}}=\ll$.

\subsection{Motivic integration on rigid varieties}\label{motivic:integration:rigid}

 Let $\XX$ be a generically smooth  {\em stft} formal $R$-scheme. 
We follow the construction of Nicaise-Sebag, Nicaise in \cite{NS}, \cite{Nicaise} for the definition of the motivic integration of a gauge form $\omega$ on $\XX_\eta$, which takes values in $\mM_{\XX_s}$. 
We briefly recall the method to define the motivic integration $\int_{\XX}|\omega|$. 
First we have
$$\XX=\varinjlim_{m}\sX_m,$$
where $\sX_m=(\XX, \oO_{\XX}\otimes_{R}R_m)$ and $R_m=R/(\pi)^{m+1}$. In Greenberg \cite{Greenberg}, the functor
$$\yY\mapsto \Hom_{R_m}(\yY\times_{\kappa}R_m, \sX_m)$$
from the category of $\kappa$-schemes to the category of sets is presented by a $\kappa$-scheme 
$$\Gr_{m}(\sX_m)$$
of finite type such that
$$\Gr_m(\sX_m)(A)=\sX_m(A\otimes_{\kappa}R_m)$$ for any $\kappa$-algebra $A$.  The projective limit 
$\varinjlim_{m}\sX_m$ is denoted by $\Gr(\XX)$.  The functor $\Gr$ respects open and closed immersions and fiber products, 
and sends affine topologically of finite type formal $R$-schemes to affine $\kappa$-schemes. The motivic integration of a gauge form 
$\omega$ is defined by using the stable cylindrical subsets of $\Gr(\XX)$, introduced by Loeser-Sebag in \cite{LS}, and Nicaise-Sebag in \cite{NS-curve}. 

Let $\bC_{0,\XX}$ be the set of stable cylindrical subsets of $\Gr(\XX)$ of some level.  If $A\subset \bC_{0,\XX}$ is a cylinder, and we have a function
$$\alpha: A\to \zz\cup\{\infty\}$$
such that $\alpha^{-1}(m)\subset \bC_{0,\XX}$. Then 
$$\int_{A}[\aaa_{\XX_s}^{1}]^{-\alpha}d\widetilde{\mu}:=\sum_{m\in\zz}\widetilde{\mu}(\alpha^{-1}(m))\cdot [\aaa^{1}_{\XX_s}]^{-m},$$
where $$\widetilde{\mu}: \bC_{0,\XX}\to \mM_{\XX_s}$$
is the unique additive morphism defined in \cite[Proposition 5.1]{Le} by
$$\widetilde{\mu}(A)=[\pi_m(A)]\cdot [\aaa^{1}_{\XX_s}]^{-(m+1)d}$$
for $A$ a stable cylinder of level $m$,  $d$ is the relative dimension of $\XX$, and $\pi_m: \Gr(\XX)\to \Gr(\sX_m)$
is the canonical projection. 

Let $\omega$ be a gauge form on $\XX_\eta$, in \cite{LS}, the authors constructed an integer-valued function
$$\ord_{\pi, \XX}(\omega)$$
on $\Gr(\XX)$ that takes the role of $\alpha$ before.  The motivic integration $\int_{\XX}|\omega|$ is defined to be
\begin{equation}\label{defn:motivic:integration}
\int_{\XX}|\omega|:=\int_{\Gr(\XX)}[\aaa_{\XX_0}^{1}]^{-\ord_{\pi,\XX}(\omega)}d\widetilde{\mu}\in \mM_{\XX_s}.
\end{equation}
From \cite{LS}, \cite{NS}, the forgetful map
$$\int: \mM_{\XX_s}\to \mM_{\kappa}$$ defined by
$$\int_{\XX}|\omega|\mapsto \int_{\XX_\eta}|\omega|:=\int \int_{\XX}|\omega|$$
only depends on $\XX_\eta$, not on $\XX$.

\begin{rmk}
In \cite{Nicaise} Nicaise generalizes the motivic integration construction to generically smooth special formal $R$-schemes.  A special formal 
$R$-scheme $\XX$ is a separated Noetherian adic formal scheme endowed with a structural morphism $\XX\to \spf(R)$, such that $\XX$ is a finite union of open formal subschemes which are formal spectra of special $R$-algebras.   From Berkovich \cite{Ber2}, a topological $R$-algebra $A$ is special, iff $A$ is topologically $R$-isomorphic to a quotient of the special $R$-algebra
$$R\{T_1,\cdots, T_m\}[\![S_1,\cdots,S_n]\!]=R[\![S_1,\cdots,S_n]\!]\{T_1,\cdots, T_m\}.$$

The Noetherian adic formal scheme $\XX$ has the largest ideal of definition $J$. The closed subscheme of $\XX$ defined by $J$ is denoted by 
$\XX_s$, which is a reduced Noetherian $\kappa$-scheme. 
\end{rmk}

We briefly review the motivic integration of Nicaise in \cite{Nicaise} for special formal schemes.  Since every {\em stft} formal 
$R$-scheme $\XX$ is a  special formal scheme, the result below definitely works for {\em stft} formal 
$R$-schemes.

\begin{defn}
Let $\XX$ be a special formal $R$-scheme. By a \textbf{N\'eron smoothening} we mean a morphism of special formal 
$R$-schemes $\YY\to\XX$, such that $\YY$ is adic smooth over $R$ and $\YY_\eta\to\XX_\eta$ is an open embedding satisfying 
$\YY_\eta(\widetilde{K})=\XX_\eta(\widetilde{\kk})$ for any finite unramified extension $\widetilde{\kk}$ of $\kk$. 
\end{defn}

In \cite[\S 2]{Nicaise}, Nicaise proves that a N\'eron smoothening of $\XX$ exists and is given by the dilatation of $\XX$.  Then $\YY$ is a 
$\stft$ formal $R$-scheme. 

\begin{defn}
Let $\XX$ be a generically smooth special formal $R$-scheme.  We define
$$\int_{\XX}|\omega|:=\int_{\YY}|\omega|$$ and 
$$\int_{\XX_\eta}|\omega|:=\int_{\YY_\eta}|\omega|$$
for a gauge form $\omega$ on $\XX_\eta$. 
\end{defn}

We recall the motivic volume of $\XX_\eta$ in \cite{Nicaise}.  
For $m\geq 1$, let 
$\kk(m):=\kk[T]/(T^m-\pi)$ be a totally ramified extension of degree $m$ of $\kk$, and 
$R(m):=R[T]/(T^m-\pi)$ the normalization of $R$ in $\kk(m)$.  If $\XX$ is a formal $R$-scheme, we define 
$$\XX(m):=\XX\times_{R}R(m)$$ and
$$\XX_\eta(m):=\XX_\eta\times_{\kk}\kk(m).$$
If $\omega$ is a gauge form on $\XX_\eta$, we denote by $\omega(m)$ the pullback of $\omega$
via the natural morphism $\XX_\eta(m)\to \XX_\eta$. 

\begin{defn}
Let $\XX$ be a generically smooth special formal $R$-scheme. Let 
$\omega$ be a gauge form on $\XX_\eta$.  Then the volume Poincar\'e series of $(\XX,\omega)$ is defined to be
$$S(\XX,\omega;T):=\sum_{d>0}\left(\int_{\XX(d)}|\omega(d)|\right)T^d\in \mM_{\XX_s}[\![T]\!].$$
\end{defn}

\label{Resolution:singularities:formal:scheme}
\begin{defn}
Let $\XX$ be a generically smooth flat $R$-formal scheme.  A resolution of singularities of $\XX$ is a proper morphism 
$h: \YY\to \XX$ of flat special formal $R$-schemes such that  
$h$ induces an isomorphism on generic fibers, and such that $\YY$ is regular (meaning the local ring at points is regular), with a special fiber 
a strict normal crossing divisor $\YY_s$.  We say that the resolution $h$ is tame if $\YY_s$ is a tame normal crossing divisor. 
\end{defn}

By Temkin's resolution of singularities for quasi-excellent schemes of characteristic zero in  \cite{Temkin},   any affine generically smooth flat special formal $R$-scheme $\XX=\spf(A)$ admits a resolution of singularities by means of admissible blow-ups.   

In general for any generically smooth $R$-formal scheme $\XX$, suppose that there is a resolution of singularities 
\begin{equation}
h:  \YY\longrightarrow \XX
\end{equation}

Let $E_i$, $i\in \ii$, be the set of irreducible components of the exceptional divisors of the resolution. 
For $I\subset \ii$,  we set 
$$E_{I}:=\bigcap_{i\in I}E_{i}$$
and 
$$E_{I}^{\circ}:=E_{I}\setminus \bigcup_{j\notin I}E_j.$$
Let $m_{i}$ be the multiplicity of the component $E_i$, which means that 
the special fiber of the resolution is 
$$\sum_{i\in \ii}m_iE_i.$$
Let $m_{I}=\gcd(m_i)_{i\in I}$. Let $U$ be an affine Zariski open subset of $\YY$, such that, 
on $U$, $f\circ h=uv^{m_{I}}$, with $u$ a unit in $U$ and $v$ a morphism from 
$U$ to $\mathbb{A}_{\cc}^{1}$. The restriction of $E_{I}^{\circ}\cap U$, which we denote by
$\tilde{E}_{I}^{\circ}\cap U$, is defined by
$$\lbrace{(z,y)\in \mathbb{A}_{\cc}^{1}\times (E_{I}^{\circ}\cap U)| z^{m_{I}}=u^{-1}\rbrace}.$$
The $E_{I}^{\circ}$ can be covered by the open subsets $U$ of $Y$.  We can glue together all such 
constructions and get the Galois cover
$$\tilde{E}_{I}^{\circ}\longrightarrow E_{I}^{\circ}$$
with Galois group $\mu_{m_{I}}$.
Remember that $\hat{\mu}=\underleftarrow{lim} \mu_{n}$ is the direct limit of the groups
$\mu_{n}$. Then there is a natural $\hat{\mu}$ action on $\tilde{E}_{I}^{\circ}$.
Thus we get 
$[\tilde{E}_{I}^{\circ}]\in \mathcal{M}_{X_0}^{\hat{\mu}}$.

\label{motivic:result:formal:scheme} Using resolution of singularities, in \cite[Theorem 7.12]{Nicaise}, Nicaise proves the following result:

\begin{thm}\label{motivic:integration:formula:omegam}
Let $\XX$ be a generically smooth special formal $R$-scheme of pure relative dimension $d$. 
Then we have a structural morphism $f: \XX\to \spf(R)$. 
Suppose that $\XX$ has a resolution of singularities $\XX^\prime\to\XX$ with special fiber 
$\XX^\prime_s=\sum_{i\in I}N_i E_i$. 

Let $\omega$ be a $\XX$-bounded gauge form on $\XX_\eta$, where the definition of bounded gauge form is given by Nicaise in \cite[Definition 2.11]{Nicaise}. Then
for any integer $m>0$, 
$$\int_{\XX(m)}|\omega(m)|=\ll^{-d}\sum_{\emptyset\neq J\subset \ii}(\ll-1)^{|J|-1}[\tilde{E}_{J}^{\circ}]\left(\sum_{\substack{k_i\geq 1, i\in J\\
\sum_{i\in J}k_i N_i=d}}\ll^{-\sum_{i}k_i\mu_i}\right)\in \mM_{\XX_s}^{\mu_m}.$$
\end{thm}

Furthermore, from \cite[Corollary 7.13]{Nicaise} we have:
\begin{prop}\label{prop:motivic:volume:result}
With the same notations and conditions as in Theorem \ref{motivic:integration:formula:omegam}, 
the volume Poincar\'e series $S(\XX,\omega;T)$ is rational over $\mM_{\XX_s}$. In fact, 
let $\mu_i:=\ord_{E_i}\omega$, then
$$S(\XX,\omega;T)=\ll^{-d}
\sum_{\emptyset\neq J\subset \ii}(\ll-1)^{|J|-1}[\tilde{E}_{J}^{\circ}]\prod_{i\in J}\frac{\ll^{-\mu_i}T^{N_i}}{1-\ll^{-\mu_i}T^{N_i}}\in\mM_{\XX_s}^{\hat{\mu}}[\![T]\!].$$

The limit 
$$S(\XX,\widehat{\kk}^{s}):=-\lim_{T\to \infty}S(\XX,\omega;T):=\ll^{-d}\mMF_{f}$$
is called the \textbf{motivic volume} of $\XX$, where 
$$\mMF_{f}=\sum_{\emptyset\neq J\subset \ii}(\ll-1)^{|J|-1}[\tilde{E}_{I}^{\circ}].$$  
 And 
\begin{align*}S(\XX_\eta,\widehat{\kk}^{s}):&=-\lim_{T\to \infty}S(\XX_\eta,\omega;T)=-\lim_{T\to \infty}\sum_{m\geq 1}\left(\int_{\XX_{\eta}}|\omega(m)|\right)T^m\\
&=\ll^{-d}\int_{\XX_s}\mMF_{f}\in\mM^{\hat{\mu}}_{\kappa}
\end{align*}
is called the \textbf{motivic volume} of $\XX_\eta$.
\end{prop}

\begin{defn}(\cite{DL}, \cite{NS})\label{defn_motivic_nearby_cycle}
For the formal scheme 
$f: \XX\to\spf(R)$, the motivic cycle 
$$\mMF_{f}=\sum_{\emptyset\neq J\subset \ii}(\ll-1)^{|J|-1}[\tilde{E}_{I}^{\circ}]$$
is called the motivic nearby cycle of $f$. 
\end{defn}

Let $(\XX, f)$ be a generically smooth formal $R$-scheme.  From Proposition \ref{prop:motivic:volume:result}, the motivic nearby cycle 
$\mMF_{f}$ belongs to $\mM_{\XX_s}^{\hat{\mu}}$.  For any point $x\in \XX_s$, let 
$$\mMF_{f,x}=\sum_{\emptyset\neq J\subset \ii}(\ll-1)^{|J|-1}[\tilde{E}_{I}^{\circ}\cap h^{-1}(x)],$$  
where $h: \XX^\prime\to \XX$ is the resolution of singularities. We call $\mMF_{f,x}$ the motivic Milnor fiber of $x\in\XX_s$.

In summary, if we let $K(\GBSRig_{\kk})$ be the Grothendieck ring of the category of gauge bounded smooth rigid $\kk$-varieties. 
Here for an object $\XX_\eta$ in $\GBSRig_{\kk}$ we understand that the rigid variety $\XX_\eta$ 
comes from the generic fiber of a generically smooth special formal $R$-scheme $f: \XX\to \spf(R)$ with gauge bounded form 
$\omega$.  The Grothendieck ring 
$$K(\GBSRig_{\kk}):=\bigoplus_{d\geq 0}K(\GBSRig_{\kk}^d)$$
is defined in \cite[\S 5.2]{Le}. 

Let $K(\BSRig_{\kk})$ be the Grothendieck ring of the category $\BSRig_{\kk}$ of bounded smooth rigid $\kk$-varieties, which is obtained from $K(\GBSRig_{\kk})$ by forgetting the gauge form.  Then we can represent the above results in 
\S (\ref{motivic:result:formal:scheme}) as follows:

\begin{thm}\label{thm_MV}
There exists a homomorphism of additive groups:
$$\MV:  K(\BSRig_{\kk})\to \mM_{\kappa}^{\hat{\mu}}$$
given by:
$$[\XX_\eta]\mapsto S(\XX_\eta, \widehat{\kk}^{s})$$
for a generically smooth special formal $R$-scheme $\XX$.  Moreover, if $\XX$ has relative dimension $d$, then 
$$\MV([\XX_\eta])=\ll^{-d}\cdot \int_{\XX_s}\mMF_{f}\in \mM_{\kappa}^{\hat{\mu}}.$$
So $\MV$ is a morphism from the group $K(\BSRig_{\kk})$ to the group $\mM_{\kappa}^{\hat{\mu}}$. 

Moreover,  if $x\in\XX_s$ and let 
$$\hat{f}_{x}: \spf(\widehat{\oO}_{\XX,x})\to\spf(R)$$
be the formal completion of $\XX$ along $x$, then the generic fiber $\spf(\widehat{\oO}_{\XX,x})_{\eta}$ of the formal completion is the analytic Milnor fiber $\FF_x(\hat{f})$ of $\hat{f}$ at $x$ in Definition \ref{analyticmilnorfiber}  and (\ref{sec_Berkovich_comparison_theorem}), and we have
$$\MV([\FF_x(\hat{f})])=\ll^{d}\cdot \mMF_{f,x}.$$
\end{thm}

\subsection{Global motive of oriented formal $d$-critical schemes}\label{subsec_global_motive_formal_scheme}

We first define the motive of principal $\zz_2$ bundles.  

Let $\zz_2(X)$ be the abelian group of isomorphism classes $[P]$ of principal $\zz_2$-bundles $P\to X$, with multiplication
$[P]\cdot [Q]=[P\otimes_{\zz_2}Q]$ and the identity the trivial bundle $[X\times \zz_2]$.  We know that $P\otimes_{\zz_2}P\cong X\times\zz_2$, so every element in $\zz_2(X)$ has order $1$ or $2$. 

In \cite{BJM}, the authors define the motive of a principal $\zz_2$-bundle $P\to X$ by:
$$\Upsilon(P)=\ll^{-\frac{1}{2}}\odot([X,\hat{\iota}]-[P,\hat{\rho}])\in\mM_{X}^{\hat{\mu}},$$
where $\hat{\rho}$ is the $\hat{\mu}$-action on $P$ induced by the $\mu_2$-action on $P$. 

In \cite{BJM}, for any scheme $Y$, the authors define an ideal $I_{Y}^{\hat{\mu}}$ in $\mM_{Y}^{\hat{\mu}}$ which is generated by
$$\phi_{*}(\Upsilon(P\otimes_{\zz_2}Q)-\Upsilon(P)\odot\Upsilon(Q))$$
for all morphisms $\phi: X\to Y$ and principal $\zz_2$-bundles 
$P, Q$ over $X$.  Then define 
$$\overline{\mM}_{Y}^{\hat{\mu}}=\mM_{Y}^{\hat{\mu}}/I_{Y}^{\hat{\mu}}.$$
Then $(\overline{\mM}_{Y}^{\hat{\mu}},\odot)$ is a commutative ring with $\odot$ and there is a natural 
projection $\prod_{Y}^{\hat{\mu}}: \mM_{Y}^{\hat{\mu}}\to \overline{\mM}_{Y}^{\hat{\mu}}$.

Let $(\XX,s)$ be an oriented formal $d$-critical $R$-scheme and $K_{\XX,s}^{\frac{1}{2}}$ exists as a line bundle over $\XX$. 
Let $(\RR,\UU,f,i)$ be a formal critical chart of $(\XX,s)$. Then we have:
$$\hat{f}: \RR\to\spf(R)$$ 
is a formal scheme, such that the underlying scheme is given by the critical locus of the function $f$.
Then we have the sheaf of vanishing cycles 
$$\mMF_{\UU, f}^{\phi}\in\mM_{\XX_s}^{\hat{\mu}}$$
by
$$\mMF_{\UU,f}^{\phi}|_{\XX_c}=\ll^{-\dim(\UU)/2}\odot [[\UU_c,\hat{\iota}]-\mMF_{\UU,f-c}]|_{\XX_c}$$
where 
$\mMF_{\UU,f-c}$ is the motivic nearby cycle of $\hat{f}-c$ in Definition \ref{defn_motivic_nearby_cycle}.   Our main result in this section is:

\begin{thm}\label{thm_global_motive}
If $(\XX,s)$ is a {\em stft} formal  $d$-critical $R$-scheme with an orientation $K_{\XX,s}^{\frac{1}{2}}$. Then there exists a unique motive 
$$\mMF_{\XX,s}^{\phi}\in \overline{\mM}_{\XX_s}^{\hat{\mu}}$$
such that if  $(\RR,\UU,f,i)$ is a formal critical chart on $(\XX,s)$, then 
$$\mMF_{\XX,s}^{\phi}|_{\RR}=i^\star(\mMF_{\UU,f}^{\phi})\odot \Upsilon(Q_{\RR,\UU,i})\in \overline{\mM}_{\RR_s}^{\hat{\mu}}$$
where $\Upsilon(Q_{\RR,\UU,i})=\ll^{-\frac{1}{2}}\odot ([\RR,\hat{\iota}]-[Q,\hat{\rho}])\in \overline{\mM}_{\RR_s}^{\hat{\mu}}$ is the motive of the principal $\zz_2$-bundle as in \cite[\S 2.5]{BJM} and recalled above.
\end{thm}
\begin{proof}
We need to show that for formal critical charts $(\RR, \UU, i,j), (\SS, \VV, g,j)$, 
\begin{equation}\label{key_formula_proof}
\Big[i^\star(\mMF_{\UU,f}^{\phi})\odot \Upsilon(Q_{\RR,\UU,f,i})\Big]|_{\RR\cap\SS}=
\Big[j^\star(\mMF_{\VV,g}^{\phi})\odot \Upsilon(Q_{\SS,\VV,g,j})\Big]|_{\RR\cap\SS}.
\end{equation}
Recall the orientation of $(\XX,s)$.   For any $x\in\RR\cap\SS$, we choose 
$(\RR^\prime,\UU^\prime,f^\prime,i^\prime)$ and $(\SS^\prime,\VV^\prime,g^\prime,j^\prime)$
and $(\TT,\WW,h,k)$ such that we have morphisms
$$\Phi: (\RR^\prime,\UU^\prime,f^\prime,i^\prime)\to (\TT,\WW,h,k)$$
$$\Psi:  (\SS^\prime,\VV^\prime,g^\prime,j^\prime)\to (\TT,\WW,h,k).$$
The quadratic form $Q_{\TT,\WW,h,k}$ satisfies the property:
$$Q_{\TT,\WW,h,k}|_{\RR^\prime}\cong i|_{\RR^\prime}^\star(P_\Phi)\otimes_{\zz_2}Q_{\RR,\UU,f,i}|_{\RR^\prime}$$
for $P_\Phi\to \Crit(f^\prime)$ the principal $\zz_2$-bundle of orientations of 
$(N_{\UU^\prime\WW}|_{\Crit(f^\prime)}, q_{\UU^\prime\WW})$,  and the data are defined by the following local isomorphisms:
$$\alpha: K_{\XX,s}^{\frac{1}{2}}|_{\RR^\prime}\to i^\star(K_{\UU})|_{\RR^\prime}; \quad 
\beta: K_{\XX,s}^{\frac{1}{2}}|_{\RR^\prime}\to i^\star(K_{\UU})|_{\RR^\prime};$$
$$
\gamma: i^\star(K_{\UU})|_{\RR^\prime}\to k^\star(K_{\WW})|_{\RR^\prime}$$
where 
$$\alpha\otimes\alpha=\iota_{\RR,\UU f i|_{\RR^\prime}},$$
$$\beta\otimes\beta=\iota_{\TT,\WW h k|_{\RR^\prime}},$$
$$\gamma\otimes\gamma=i|_{\RR^\prime}^{\star}(J_{\Phi}).$$
Then we can calculate as in (5.10) of \cite{BJM},
$$\Big[k^\star(\mMF_{\WW,h}^{\phi})\odot \Upsilon(Q_{\TT\WW hk})\Big]|_{\RR^\prime}=
\Big[i^\star(\mMF_{\UU,f}^{\phi})\odot \Upsilon(Q_{\RR\UU fi})\Big]|_{\RR^\prime}.$$
Similarly for $\Psi$, 
$$\Big[k^\star(\mMF_{\WW,h}^{\phi})\odot \Upsilon(Q_{\TT\WW hk})\Big]|_{\SS^\prime}=
\Big[j^\star(\mMF_{\VV,g}^{\phi})\odot \Upsilon(Q_{\SS\VV gj})\Big]|_{\SS^\prime}.$$
So (\ref{key_formula_proof}) is proved by restricting to $\RR^\prime\cap\SS^\prime$. Hence 
$\mMF_{\UU, f}^{\phi}$ glue to give the global motive $\mMF_{\XX,s}^{\phi}$.
\end{proof}

\begin{cor}
Let $(X,s)$ be a $d$-critical  $\spec(\kappa[t])$-scheme in the sense of Joyce \cite{Joyce}, and let $(\XX, s)$ be the formal 
$t$-adic completion of $X$. Then by the relative GAGA there is a unique coherent sheaf $\FP_{\XX}$ which is the formal completion of $\FP_{X}$, and a section $s\in \FP_{\XX}$, such that $(\XX,s)$ is a $d$-critical formal scheme over $R$.  Moreover, the unique global motive $\mMF_{\XX,s}^{\phi}$ is the same as $\mMF_{X,s}$ as an element in $\overline{\mM}_{X}^{\hat{\mu}}$.
\end{cor}
\begin{proof}
By the relative GAGA in \cite{Conrad}, the first statement is obvious.  For the second one, note that locally the motivic vanishing cycles $\mMF_{\UU,f}^{\phi}$ is defined by the motivic nearby cycles, which are the same for formal $d$-critical charts $(\RR,\UU,f,i)$ and $d$-critical charts $(R,U,f,i)$.  So by gluing they must give the same global motive in $\overline{\mM}_{X}^{\hat{\mu}}$ since $\XX_s=X$. 
\end{proof}

\begin{rmk}
The motivic vanishing cycle $\mMF_{\UU,f}$ is close related to the perverse sheaf of vanishing cycles $\PV_{\UU,f}$ as in 
\cite{Ber4} and 
\cite{Sabbah}.   In \cite{Sabbah}, Sabbah proves that the  perverse sheaf of vanishing cycles $\PV_{\UU,f}$ is isomorphic to the cohomology of a formal twisted de Rham complex, which is the Kontsevich conjecture and  is inspired by the deformation quantization in physics.  It seems that working over the non-archimedean field $\kappa(\!(t)\!)$ is the right way for the quantization. 
\end{rmk}

\subsection{Global motive for oriented $d$-critical non-archimedean analytic spaces}\label{subsec_global_motive_analytic_space}

Let $(X,s)$ be a $d$-critical $\kk$-analytic space.  Choose a $d$-critical formal model $(\XX,s)$ for $(X,s)$ such that the generic fiber $\XX_\eta\cong X$. 

\begin{defn}\label{defn_global_motive_analytic_space}
We define the global motive $\mMF_{X,s}^{\phi}$ to be
$$\mMF_{X,s}^{\phi}:=\int_{\XX_s}\mMF_{\XX,s}^{\phi}\in \overline{\mM}_{\kappa}^{\hat{\mu}},$$
where $\int_{\XX_s}$ means pushforward to a point. 
\end{defn}

\begin{prop}\label{prop_global_motive_only_analytic_space}
For a $d$-critical $\kk$-analytic space $(X,s)$, the global motive $\mMF_{X,s}^{\phi}$ depends only on the 
$d$-critical $\kk$-analytic space $(X,s)$. 
\end{prop}
\begin{proof}
The global motive $\mMF_{X,s}^{\phi}$ is defined by $\int_{\XX_s}\mMF_{\XX,s}^{\phi}$ for a formal model of 
$(X,s)$.  Then the result just follows from Theorem \ref{thm_MV} and \cite[Proposition-Definition 7.43]{Nicaise}.
\end{proof}

\subsection{Maulik's motivic localization formula under the $\Gm$-action}\label{subsec_motivic_localization_formula}

We prove a $\Gm$-localization formula for the global motive $\mMF_{X,s}^{\phi}$ for an oriented 
$d$-critical $\kk$-analytic space $(X,s)$.  In the scheme level, this motivic localization formula is originally due to D. Maulik \cite{Maulik}, who, using the torus action on local vanishing cycle sheaves, proved the motivic localization formula as recalled in \cite[Theorem 5.16]{BJM}. 
We generalize Maulik's motivic localization formula  to formal schemes and non-archimedean $\kk$-analytic spaces and prove it by using motivic integration for formal schemes as in \cite{Nicaise}, \cite{Le} and  \cite{Jiang}.  

\subsubsection{$\Gm$-localization formula}

Let $(\XX,s)$ be a $d$-critical formal $R$-scheme with a good $\Gm$-action.  Let 
$$\XX^{\Gm}=\bigsqcup_{i\in J}\XX_i^{\Gm}$$
be the decomposition of the fixed locus $\XX^{\Gm}$ into connected components, such that 
$(\XX_i^{\Gm}, s_i^{\Gm})$ are oriented formal $d$-critical schemes.  On the tangent space 
$T_{x}\XX_i$ of $\XX_i$ at a point $x\in\XX_i$, where we take $T_{x}\XX_i$ as a $R$-module ( when reduced to the residue field $\kappa$, $T_{x}\XX_i$ becomes the tangent space $T_{x}(\XX_i)_s$ of the scheme $(\XX_i)_s$). The action $\Gm$ has a decomposition 
$$T_x(\XX_i)=(T_x\XX_i)_{0}\oplus T_x(\XX_i)_{+}\oplus (T_x(\XX_i))_-$$
where the direct sums are the parts of zero, positive and negative weights with respect to the $\Gm$-action.  
Maulik \cite{Maulik} defined the virtual index
\begin{equation}\label{virtual_index}
\ind^{\virt}(\XX_i^{\Gm},\XX)=\dim_{R}(T_{x}(\XX)_+)-\dim_{R}(T_{x}(\XX)_-)
\end{equation}
so that it is constant on the strata $\XX_i^{\Gm}$. 

All the above arguments work for $d$-critical $\kk$-analytic space $(X,s)$ with a good $\Gm$-action.  Let 
$$X^{\Gm}=\bigsqcup_{i\in J}X_i^{\Gm}$$
be the decomposition of the fixed locus $X^{\Gm}$ into connected components, such that 
$(X_i^{\Gm}, s_i^{\Gm})$ are oriented $d$-critical non-archimedean spaces.  The action $\Gm$ has a decomposition 
$$T_x(X_i)=(T_x X_i)_{0}\oplus T_x(X_i)_{+}\oplus (T_x(X_i))_-$$
where the direct sums are the parts of zero, positive and negative weights with respect to the $\Gm$-action.  
The virtual index
$$
\ind^{\virt}(X_i^{\Gm},X)=\dim_{\kk}(T_{x}(X)_+)-\dim_{\kk}(T_{x}(X)_-)
$$
is similarly defined and is constant on the strata $X_i^{\Gm}$.

\begin{defn}
We call the action 
$$\mu: \Gm\times\XX\to \XX; \quad  (\text{or~} \mu: \Gm\times X\to X)$$
{\em circle-compact} if the limit $\lim_{\lambda\to 0}\mu(\lambda)x$ exists for any $x\in\XX (\text{or~} x\in X)$. If $\XX (\text{or~} X)$ is proper, then any $\Gm$-action on $\XX ( \text{or~} X)$ is circle-compact. 
\end{defn}

We present the generalization of Maulik's motivic localization formula in \cite{Maulik} to $d$-critical non-archimedean $\kk$-analytic spaces. 
\begin{thm}\label{thm_Maulik}
Let $(X,s)$ be a $d$-critical non-archimedean $\kk$-analytic space and $\mu$ is a good, circle-compact action of $\Gm$ on $X$, which preserves the orientation $K_{X,s}^{\frac{1}{2}}$. Then on each fixed strata $X_i^{\Gm}$, there is an oriented  $d$-critical $\kk$-analytic space structure $(X_i^{\Gm},s_i^{\Gm})$, hence a global motive $\mMF_{X_i^{\Gm},s_i^{\Gm}}^{\phi}$. Moreover we have the following motivic localization formula
$$\mMF_{X,s}^{\phi}=\sum_{i\in J}\ll^{-\ind^{\virt}(X_i^{\Gm},X)/2}\odot \mMF_{X_i^{\Gm},s_i^{\Gm}}^{\phi}\in \overline{\mM}_{\kappa}^{\hat{\mu}}.$$
\end{thm}
\begin{proof}
For the oriented $d$-critical non-archimedean $\kk$-analytic space $(X,s)$, we choose a formal model 
$(\XX,s)$, which is an oriented formal $d$-critical scheme.  From Proposition \ref{prop_global_motive_only_analytic_space}, the global motive 
$$\mMF_{X,s}^{\phi}=\int_{\XX_s}\mMF_{\XX,s}^{\phi}$$ 
only depends on $(X,s)$, and is independent to the choice of the formal models. 
So it is sufficient to prove the result for a formal model $(\XX,s)$ of $(X,s)$. Then the $\Gm$-action on $X$ can be extended to a good and circle-compact action on the formal scheme $\XX$.   For each $\Gm$-fixed strata 
$(X_i, s_i^{\Gm})$, there is a corresponding $\Gm$-fixed strata $(\XX_i, ,s_i^{\Gm})$, which is an oriented $d$-critical formal $R$-scheme.  We need to prove the formula:
\begin{equation}\label{thm_motivic_localization_key_formula}
\int_{\XX_s}\mMF_{\XX,s}^{\phi}=\sum_{i\in J}\ll^{-\ind^{\virt}(\XX_i^{\Gm},\XX)/2}\odot \int_{(\XX_i^{\Gm})_s}\mMF_{\XX_i^{\Gm},s_i^{\Gm}}^{\phi}\in \overline{\mM}_{\kappa}^{\hat{\mu}}.
\end{equation}
We divide the proof into three steps. \\

Step 1:  We first prove the result on a formal $d$-critical chart $(\RR,\UU,f,i)$ on $(\XX,s)$.  We have the formal scheme
$$\UU=\spf(R\{x_1,\cdots,x_m\}); \quad f\in R\{x_1,\cdots, x_m\} $$ 
and 
$$\hat{f}: \RR=\spf(R\{x_1,\cdots,x_m\}/(I_{df}))\to\spf(R).$$
From our definition of motivic vanishing cycles, 
$$\mMF_{\UU,f}^{\phi}=\ll^{-\frac{1}{2}\dim(\UU)}\cdot ([\UU]-\mMF_{\UU,f}),$$
where $\mMF_{\UU,f}$ is the motivic nearby cycle of the formal scheme 
$\hat{f}$, see Definition \ref{defn_motivic_nearby_cycle}. Let 
$$\RR^{\Gm}=\bigsqcup_{i}\RR_i^{\Gm}$$
be the decomposition of $\RR^{\Gm}$ into $\Gm$-fixed connected components.  Since the function $f$ is $\Gm$-invariant, we can choose coordinates on $\UU$ such that 
$$T_x\XX$$
has dimension $\dim(\UU)$ for $x\in \RR_i^{\Gm}$, see \cite[Corollary 5.1.13]{Fujiwara-Kato}. 
The $\Gm$-action on $\UU$ has a decomposition
$$\UU=\UU^{\Gm}\oplus \UU_{+}\oplus\UU_{-}$$
under the $\Gm$-weights on the tangent space of $\UU$.  Then 
inside 
$$T_x\XX=(T_x\XX_i)_{0}\oplus T_x(\XX_i)_{+}\oplus (T_x(\XX_i))_-,$$
the dimensions of $(T_x\XX_i)_{0},  T_x(\XX_i)_{+}, (T_x(\XX_i))_-$ are exactly $\dim(\UU^{\Gm}), \dim(\UU_{+}), \dim(\UU_-)$, respectively, and 
$$x\in\UU_i^{\Gm}=\spf(R\{x_1,\cdots,x_r\});$$
for $r\leq m$.  
Here we assume that 
$$\UU_{+}=\spf(R\{x_{r+1},\cdots,x_s\}); \quad \UU_{-}=\spf(R\{x_{s+1},\cdots,x_m\}).$$
Hence 
$(\RR_i^{\Gm}, \UU_i^{\Gm}, f^{\Gm}, i^{\Gm})$ is a formal $d$-critical chart and 
$(\RR_i^{\Gm}, s^{\Gm}=i^{\Gm,\star}s)$ is a $d$-critical formal scheme, where $i^{\Gm}: \RR_i^{\Gm}\hookrightarrow \RR$ is the inclusion. Since $\XX_i^{\Gm}$ is covered by $\bigcup_{i}\RR_i^{\Gm}$ for the formal $d$-critical charts, then we have 
$(\XX_i^{\Gm}, s_i^{\Gm})$ is a formal $d$-critical scheme over $R$ for any $i$. 

We show that there is an induced orientation on $(\XX_i^{\Gm}, s_i^{\Gm})$.  For the formal $d$-critical chart 
$(\RR_i^{\Gm}, \UU_i^{\Gm}, f^{\Gm}, i^{\Gm})$ of $(\XX_i^{\Gm}, s_i^{\Gm})$, consider the following diagram
\[
\xymatrix{
\RR_i^{\Gm}\ar@{^{(}->}[r]^{i^{\Gm}}\ar@{^{(}->}[d]_{i_{\RR}}& \UU_i^{\Gm}\ar@{^{(}->}[d]^{j_{\UU}}\\
\RR\ar@{^{(}->}[r]^{i}& \UU,
}
\]
where all the morphisms are inclusions. Since the $\Gm$-action preserves the orientation $K_{\XX,s}^{\frac{1}{2}}$, we make the following diagram:
\[
\xymatrix{
i_{\RR}^{\star}K_{\XX,s}^{\frac{1}{2}}|_{\RR}\ar[r]^{i_{\RR}^{\star}\alpha}\ar[d]_{:=}& i_{\RR}^{\star}i^{\star}(K_{\UU})\ar[d]^{\cong}\\
K_{\XX_i^{\Gm}, s_i^{\Gm}}^{\frac{1}{2}}|_{\RR_i^{\Gm}}\ar[r]& (i^{\Gm})^{\star}j_{\UU}^{\star}(K_{\UU}),
}
\]
where $\alpha: K_{\XX,s}^{\frac{1}{2}}|_{\RR}\stackrel{\cong}{\longrightarrow}i^\star(K_{\UU})$ is the local isomorphism determined by the orientation $ K_{\XX,s}^{\frac{1}{2}}$, and we define $K_{\XX_i^{\Gm}, s_i^{\Gm}}^{\frac{1}{2}}$ by the gluing of the local data above.   Hence there exists an orientation $K_{\XX_i^{\Gm}, s_i^{\Gm}}^{\frac{1}{2}}$ on $(\XX_i^{\Gm}, s_i^{\Gm})$. 
In practice, if $\XX$ is the completion of a moduli scheme $X$ of stable sheaves over a smooth Calabi-Yau threefold, then 
there is a $d$-critical scheme structure $(X,s)$ on $X$, and the canonical line bundle 
$K_{\XX,s}=\widehat{\det(E_{X}^\bullet)}$ is the completion of the determinant line bundle of the symmetric obstruction theory complex $E_{X}^\bullet$.  If the $\Gm$-action preserves the orientation $K_{\XX,s}^{\frac{1}{2}}$, then there exists an equivariant symmetric obstruction on $(X, E_{X}^{\bullet})$, and on the fixed locus $X_i^{\Gm}$, there exists an induced symmetric obstruction theory and an oriented $d$-critical scheme structure $(X_i^{\Gm}, s_i^{\Gm})$. Taking completion we get the oriented formal $d$-critical scheme structure $(\XX_i^{\Gm}, s_i^{\Gm})$.
\\

Step2:  We prove the following:
\begin{equation}\label{local_picture}
\int_{\RR_s}\mMF_{\UU,f}=\sum_{i\in J}\ll^{-\ind^{\virt}(\RR_i^{\Gm}, \RR)/2}\cdot \int_{(\RR_i^{\Gm})_s}\mMF_{\UU_i^{\Gm},f^{\Gm}}
\end{equation}
From Theorem \ref{thm_MV}, also \cite{Le2}, we have
$$\int_{\RR_s}\mMF_{\UU,f}=\ll^{\dim(\UU)}\cdot \MV([\RR_\eta]).$$
Here $\RR_\eta$ is the generic fiber of the formal scheme $\RR\to\spf(R)$ and 
\[
\RR_{\eta}= 
\left\{
(x_0,x_+,x_-)\in\aaa_{\kk}^{m,\an} \Big|
 \begin{array}{l}
  \text{$ \val(x_0)>0;$} \\
  \text{$ \val(x_+)>0;$} \\
  \text{$ \val(x_-)\geq 0;$} \\
  \text{$f(x_0,x_+,x_-)=t$.} 
 \end{array}\right\}
\]
Here $\val(x_0):=\min_{1\leq i\leq r}\{\val(x_i)\}$, and $\val(x_+):=\min_{r+1\leq i\leq s}\{\val(x_i)\}$, $\val(x_-):=\min_{s+1\leq i\leq m}\{\val(x_i)\}$.  We explain here why in $\RR_\eta$, $\val(x_-)\geq 0$. This is because the $\Gm$-action on $\RR$ is circle-compact, which means that 
$\lim_{\lambda\to 0}\mu(\lambda)x$ exists on $\RR$. 
The formal scheme $\RR\to\spf(R)$ is the formal completion of the formal scheme 
$\UU\to\spf(R)$ along $\RR=\Crit(f)$.  
So the condition that  the $\Gm$-action on the cell $\UU_+$ has positive weights is a closed condition on $\RR_s$, and the corresponding preimage under the specialization map
$$sp: \RR_\eta\to\RR_s$$
must be open which is $|x_+|<1$ and equivalent to $\val(x_+)>0$; while 
on the affine formal scheme $\UU_-$, the $\Gm$-action  has negative weights and this is an open condition on $\RR_s$, so the
corresponding preimage under the specialization map $sp$ is closed, which is   $|x_-|\leq 1$ and equivalent to 
$\val(x_-)\geq 0$. Now let 
$$\RR_\eta=\mathit{R}_0\sqcup \mathit{R}_1,$$
where 
$$\mathit{R}_0=\{(x_0,x_+,x_-)\in\RR_{\eta} | x_+=0 \text{~or~} x_-=0 \}$$
$$\mathit{R}_1=\{(x_0,x_+,x_-)\in\RR_{\eta} | x_+\neq 0 , x_-\neq 0 \}=\RR_\eta\setminus \mathit{R}_0.$$
Since the function $f$ is $\Gm$-invariant, if one of $x_+, x_-$ is zero, then the function $f$ will not have $x_+, x_-$ terms, i.e.,
$$f(x_0, x_+, x_-)=f(x_0,0,0).$$ 

We use the same arguments as in \cite[Theorem 3.9]{Jiang4} to show
$$\MV([\mathit{R}_1])=0.$$
The key point is that using the Cluckers-Loeser motivic constructible functions the motivic volume of an annulus is zero.  For the completeness, we provide a proof here. 

The idea of Clucker-Loeser is to do integration on subobjects of 
$\kappa(\!(t)\!)^m\times\kappa\times\zz^r$.  In \cite[\S 16.2, \S 16.3]{CL} and  \cite[\S 3.2]{Jiang4},  let $\mathbb{T}$ be the theory of algebraic closed fields containing 
$\kappa$, then $(\kk(\!(t)\!), \kk, \zz)$ is a model for $\mathbb{T}$ for any field
$\kk$ containing $\kappa$.    The primary definable $\mathbb{T}$-subassignment has the forms
$$h[m,n,r](\kk):=\kk(\!(t)\!)^m\times\kk^n\times\zz^r.$$   A general definable $\mathbb{T}$-subassignment 
$h_{\varphi}$ is determined by a formula $\varphi$.  For instance, if $\varphi=W=\XX\times X\times\zz^r$, with $\XX$ a $\kappa(\!(t)\!)$-variety, $X$ a $\kappa$-variety, then 
$$h_{W}(\kk)=\XX(\kk(\!(t)\!))\times X(\kk)\times\zz^r.$$

Let $\Def_\kappa$ be the category of all definable $\mathbb{T}$-subassignments and $S\in \Def_{\kappa}$ be an element.  Then we have the equivariant Grothendieck group $K^{\hat{\mu}}_0(\RDef_S)$, where $\RDef_S$ is the subcategory of $\Def_S$ whose objects are subassignments of $S\times h_{\aaa_\kappa^n}$ for variable $n$, morphisms to $S$ are the ones induced by the projection onto the $S$-factor.  
Let $$A:=\zz[\ll, \ll^{-1}, (1-\ll^i)^{-1}, i>0].$$  
For $S\in\Def_{\kappa}$, let $\pP(S)$ be the subring of the ring of functions 
$$S\to A$$
generated by:
\begin{enumerate}
\item all constant functions into $A$;
\item all definable functions $S\to \zz$;
\item all functions of the form $\ll^{\alpha}$, where $\alpha: S\to\zz$ is a definable function. 
\end{enumerate}
The ring of the monodromic  constructible motivic functions $\C^{\hat{\mu}}(S)$ on $S$  are defined as:
$$\C^{\hat{\mu}}(S)=K^{\hat{\mu}}_0(\RDef_S)\otimes_{\pP^0(S)}\pP(S)$$
where $\pP^0(S)$ is the subring of $\pP(S)$ generated by $\ll-1$ and by character function $\mathbb{1}_{Y}$ for  all definable subassignments $Y$ of  $S$.  See \cite[Definition 3.5]{Jiang4}, and \cite[Proposition 3.6]{Jiang4} for more details. 

First $\nn_{>0}\in \Def_{\kappa}$.
Similar to \cite[Theorem 5.1]{Le2}, we show that $\MV([\mathit{R}_1])\in\C^{\hat{\mu}}(\nn_{>0})$, with structural map:
$$\theta: (x_0, x_+, x_-)\mapsto \val(x_+)+\val(x_-).$$
From Theorem \ref{prop:motivic:volume:result},
$$\MV([\mathit{R}_1])=-\lim_{T\to\infty}\sum_{m\geq 1}\left(\int_{\mathit{R}_1(m)}|\omega(m)|\right)T^m,$$
with $\omega$ a gauge form on $\mathit{R}_1$.  By choosing a formal model $\RR_1$ of $\mathit{R}_1$ and a N\'eron smoothening $\RR^\prime$,  according to 
\cite[\S 4]{Nicaise},
\begin{align*}
&\int_{\mathit{R}_1(m)}|\omega(m)|=\\
&\int_{\RR_0^\prime}\sum_{n\in\zz}[\{(x_0, x_+, x_-)\in\Gr_{l}\RR^\prime(m)| \ord_{t^{1/m},\RR^\prime(m)}(\omega(m))(x_0, x_+, x_-)=n\}\to\RR_0^\prime]
\end{align*}
So the correspondence 
$$(x_0, x_+, x_-)\mapsto \ord_{t^{1/m}}(x_+)+\ord_{t^{1/m}}(x_-)$$
defines a mapping 
$$\theta_m: \int_{\mathit{R}_1(m)}|\omega(m)|\to\nn_{>0}$$
for each $m\in\nn_{>0}$.
All of these maps $\theta_m$ give a map:
$$\theta: \MV([\mathit{R}_1])\to\nn_{>0}.$$
So $\MV([\mathit{R}_1])$ can be taken as an element in $\C^{\hat{\mu}}(\nn_{>0})$ with structure morphism $\theta$. 

Let $n\in\nn_{>0}$, and $\theta_{m}^{-1}(n)$ is a definable subset of 
$\int_{\mathit{R}_1(m)}|\omega(m)|$ defined by
$$\val(x_+)+\val(x_-)=n$$
i.e. $$\theta^{-1}(n)=\MV([\mathit{R}_1,n])$$
where
$$\mathit{R}_{1,n}:=\bigcup_{m\geq 1}\{(x_0, x_+,x_-)\in X_1| \val(x_+)+\val(x_-)=\frac{n}{m}\}.$$

A same proof as in \cite[Lemma 3.10]{Jiang4} shows 
$$\MV\left([\{(x_0, x_+, x_-)\in \mathit{R}_1| \val(x_+)+\val(x_-)=\frac{n}{m}\}]\right)=0.$$
The key part is that the motivic volume of an annulus is zero. 

Then let $$s_n: \mathit{R}_{1,n}\to\nn_{>0}$$ be the map
$$(x_0, x_+,x_-)\mapsto m,  \text{~if~} \val(x_+)+\val(x_-)=\frac{n}{m}.$$
Then $\theta^{-1}(n)=\MV([\mathit{R}_{1,n}])\in\C^{\hat{\mu}}(\nn_{>0})$ and there is a structural mapping
$$\tau_n: \theta^{-1}(n)\to\nn_{>0}$$
induced by $s_n$.  For any $m\in\nn_{>0}$, 
$$\tau_{n}^{-1}(m)=\MV\left([\{(x_0, x_+,x_-)\in \mathit{R}_1| \val(x_+)+\val(x_-)=\frac{n}{m}\}]\right)=0.$$
By \cite[Theorem 3.7]{Jiang4}, there is a map 
$$M: \C^{\hat{\mu}}(\nn_{>0})\to \mM_{\kappa,\loc}^{\hat{\mu}}[\![T]\!]_{\Gamma}$$
which is an isomorphism of rings, so
$$M(\theta^{-1}(n))=\sum_{m\geq 1}\tau_{n}^{-1}(m)T^m=0$$ and 
$$\theta^{-1}(n)=0\in\C^{\hat{\mu}}(\nn_{>0}).$$ Hence
$$\theta^{-1}(n)=0\in\mM_{\kappa,\loc}^{\hat{\mu}}.$$
Then 
$$M(\MV([\mathit{R}_1]))=\sum_{m\geq 1}\theta^{-1}(n)T^n=0.$$
So 
$$\MV([\mathit{R}_1])=0.$$

Let us do the case $\MV([\mathit{R}_0])$.  
Let $m_\pm=\dim(\UU_\pm)$. 
Similar to \cite[Theorem 3.8]{Jiang4}, we write 
$$\mathit{R}_0=\mathit{Y}_0\times \ZZ_\eta$$
where 
\[
\mathit{Y}_{0}= 
\left\{
(x_+,x_-)\in\aaa_{\kk}^{m_++m_-,\an} \Big|
 \begin{array}{l}
  \text{$ \val(x_+)>0;$} \\
  \text{$ \val(x_-)\geq 0;$} 
 \end{array}\right\}
\] 
and 
\[
\ZZ_{\eta}= 
\left\{
(x_0,0,0)\in\aaa_{\kk}^{m,\an} \Big|
 \begin{array}{l}
  \text{$ \val(x_0)>0;$} \\
  \text{$f(x_0,0,0)=t$.} 
 \end{array}\right\}
\]
Then let $dx_+\wedge dx_-:=dx_{r+1}\wedge \cdots\wedge dx_{m}$ be the standard gauge form on the open and closed ball 
$\mathit{Y}_0$ (open on $x_+$ and close on $x_-$). From \cite[Theorem 7.3]{Nicaise} or \cite[Theorem 2.6]{Jiang4}, we calculate 
$$\int_{\mathit{Y}_0(m)}|dx_+\wedge dx_-|=\ll^{-m_+}$$
since close ball has motivic volume $1$, and open ball has motivic volume $\ll^{-\dim}$. 
So 
\begin{align*}
\int_{\RR_s}\mMF_{\UU,f}&=\ll^{\dim(\UU)}\cdot \MV([\RR_\eta])\\
&=\ll^{\dim(\UU)}\cdot\MV([\mathit{R}_0]+ [\mathit{R}_1])\\
&=\ll^{\dim(\UU)}\cdot \MV([\mathit{R}_0]) \\
&=\ll^{\dim(\UU)}\cdot \MV([\mathit{Y}_0\times \ZZ_\eta])\\
&=-\ll^{\dim(\UU)}\cdot \lim_{T\to \infty}\sum_{m\geq 1}\left(\int_{\mathit{Y}_0(m)\times \ZZ_\eta(m)}|dx_+\wedge dx_-\wedge \omega(m)|\right)T^m \\
&=-\ll^{\dim(\UU)}\cdot \ll^{-m_+} \cdot\lim_{T\to \infty}\sum_{m\geq 1}\left(\int_{\ZZ_\eta(m)}|\omega(m)|\right)T^m \\
&=\ll^{\dim(\UU)-m_+}\cdot \ll^{-\dim(\UU^{\Gm})}\cdot \mMF_{\UU^{\Gm},f|_{\UU^{\Gm}}}\\
&=\ll^{m_-}\cdot \mMF_{\UU^{\Gm},f|_{\UU^{\Gm}}}.
\end{align*}
Hence 
\begin{equation}\label{proof_key_1}
\ll^{-\frac{1}{2}m}\cdot \int_{\RR_s}\mMF_{\UU,f}=\ll^{-\frac{1}{2}(m_+-m_-)}\cdot \ll^{-\frac{1}{2}m_0}\cdot \mMF_{\UU^{\Gm},f|_{\UU^{\Gm}}}.
\end{equation}
For the trivial motive $[\UU]$, it is clear that 
\begin{align}\label{proof_key_2}
\int_{\RR_s}[\UU]|_{\RR_s}&=\ll^{\dim(\UU)}\cdot \MV([\UU_\eta])\\
&=\int_{\RR_s}([\mathit{U}_0]+ [\mathit{U}_1]) \nonumber \\
&=\int_{\RR_s}([\mathit{U}_0]) \nonumber \\
&=\ll^{m_-}\cdot [\UU^{\Gm}], \nonumber
\end{align}
where we use the same calculation as above such that 
$$\mathit{U}_0=\{(x_0,x_+,x_-)\in\UU_\eta| x_+=0 \text{~or~} x_-=0\}$$
$$\mathit{U}_1=\{(x_0,x_+,x_-)\in\UU_\eta| x_+\neq 0, x_-\neq 0\}$$
and 
$\MV([\mathit{U}_1])=0$ as in \cite[Theorem 3.9]{Jiang4}.  So from (\ref{proof_key_2}), 
\begin{equation}\label{proof_key_3}
\ll^{-\frac{1}{2}m}\int_{\RR_s}[\UU]|_{\RR_s}=\ll^{-\frac{1}{2}(m_+-m_-)}\cdot \ll^{-\frac{1}{2}m_0}\cdot \int_{(\RR_i^{\Gm})}[\UU^{\Gm}]
\end{equation}
So from (\ref{proof_key_1}) and (\ref{proof_key_3}), we have
\begin{equation}\label{thm_proof_Key1}
\mMF_{\UU,f}^{\phi}=\sum_{i\in J}\ll^{-\frac{1}{2}\ind^{\virt}(\RR_i^{\Gm},\RR)/2}\cdot \int_{(\RR_i^{\Gm})_s}\mMF_{\UU_i^{\Gm},f|_{\UU_i^{\Gm}}}^{\phi}.
\end{equation}\\

Step 3:  Finally we need to glue the formulas in (\ref{thm_proof_Key1}) to get a global formula. 
The $\Gm$-fixed formal subschemes $\XX^{\Gm}$ is covered by formal $d$-critical charts $(\RR_i^{\Gm}, \UU_i^{\Gm}, f|_{\UU_i^{\Gm}}, i|_{\RR_i^{\Gm}})$, and from Proposition \ref{prop_Gm_fixed_d_critical_formal_scheme}, 
$(\XX^{\Gm}, s^{\Gm})$ is also an oriented $d$-critical formal scheme over $R$.  For different formal $d$-critical charts 
$(\RR,\UU,f,i)$ and $(\SS,\VV,g,j)$ such that 
$(\RR_i^{\Gm}, \UU_i^{\Gm}, f|_{\UU_i^{\Gm}}, i|_{\RR_i^{\Gm}})$ and $(\SS_i^{\Gm}, \VV_i^{\Gm}, g|_{\VV_i^{\Gm}}, j|_{\SS_i^{\Gm}})$  are open formal $d$-critical charts of components of $\XX_i^{\Gm}$,  the global motive 
$\mMF_{\XX,s}^{\phi}$ agrees on the overlap 
$\RR\cap\SS$ and $\mMF_{\XX^{\Gm},s^{\Gm}}^{\phi}$ agrees on the intersection 
$\RR^{\Gm}\cap \SS^{\Gm}$.  So to show that the local formula (\ref{thm_proof_Key1}) glue to give the result in the theorem, we need to explain the role of  the extra term 
$\Upsilon(Q_{\RR,\UU,f,i})$ in 
$$\mMF_{\XX,s}^{\phi}|_{\RR}=i^\star(\mMF_{\UU,f}^{\phi}\odot \Upsilon(Q_{\RR,\UU,f,i}))$$
where $\Upsilon(Q_{\RR,\UU,f,i})$ is the principal $\zz_2$-bundle over $\RR$ such that it parametrizes the local isomorphisms:
$$\alpha: K_{\XX,s}^{\frac{1}{2}}|_{\RR}\stackrel{\sim}{\rightarrow}i^\star(K_{\UU})|_{\RR}$$
and $\alpha\otimes\alpha=\iota_{\RR,\UU,f,i}$.  Under the torus $\Gm$-action, 
$$\Upsilon(Q_{\RR,\UU,f,i})=\sum_{i\in J}\Upsilon(Q_{\RR_i^{\Gm}, \UU_i^{\Gm}, f|_{\UU_i^{\Gm}}, i|_{\RR_i^{\Gm}}})$$
since 
$\Upsilon(Q_{\RR,\UU,f,i})=\ll^{-\frac{1}{2}}\odot ([\RR,\hat{i}]-[Q, \hat{\rho}])$ and the motivic volume of the non-fixed locus of $\RR$ is zero.  The proof is similar to the proof in Step 2. 
Hence similar to the proof as in Theorem \ref{thm_global_motive},  we need to show:  on $(\XX_i^{\Gm}, s_i^{\Gm})$, for any formal $d$-critical charts 
$$(\RR_i^{\Gm}, \UU_i^{\Gm}, f|_{\UU_i^{\Gm}}, i|_{\RR_i^{\Gm}}); \quad (\SS_i^{\Gm}, \VV_i^{\Gm}, g|_{\VV_i^{\Gm}}, j|_{\SS_i^{\Gm}})$$
we have
\begin{multline*}
\Big[(i_{\RR_i^{\Gm}})^\star(\mMF_{\UU_i^{\Gm},f|_{\UU_i^{\Gm}}}^{\phi})\odot \Upsilon(Q_{\RR_i^{\Gm}, \UU_i^{\Gm}, f|_{\UU_i^{\Gm}}, i|_{\RR_i^{\Gm}}})\Big]|_{\RR_i^{\Gm}\cap\SS_i^{\Gm}}= \\
\Big[(j_{\SS_i^{\Gm}})^\star(\mMF_{\VV_i^{\Gm},g|_{\VV_i^{\Gm}}}^{\phi})\odot \Upsilon(Q_{\SS_i^{\Gm}, \VV_i^{\Gm}, g|_{\VV_i^{\Gm}}, j|_{\SS_i^{\Gm}}})\Big]|_{\RR_i^{\Gm}\cap\SS_i^{\Gm}}.
\end{multline*}
This is a similar argument as in the proof of Theorem \ref{thm_global_motive}. Hence 
the data 
$\mMF_{\UU_i^{\Gm},f|_{\UU_i^{\Gm}}}^{\phi}\odot \Upsilon(Q_{\RR_i^{\Gm}, \UU_i^{\Gm}, f|_{\UU_i^{\Gm}}, i|_{\RR_i^{\Gm}}})$
glue to give $\mMF_{\XX_i^{\Gm}, s_i^{\Gm}}^{\phi}$ and Formula (\ref{thm_motivic_localization_key_formula}) follows. 
\end{proof}

\begin{rmk}
If $(X,s)$ is an oriented $d$-critical scheme in \cite{Joyce}, the result in Theorem \ref{thm_Maulik} is originally due to D. Maulik \cite{Maulik}, where he used the vanishing cycle of regular function 
$f: U\to\cc$ on a $d$-critical chart $(R,U,f,i)$, and the gluing of the local data in \cite[\S 2.6]{Joyce}.   Maulik's paper is still unavailable, but we believe that the method he uses is different from ours by the motivic integration on the formal $d$-critical chart $(\RR,\UU,f,i)$. 
\end{rmk}

\subsubsection{Motivic localization of Donaldson-Thomas invariants}\label{sec_motivic_DT_invariants}

In this section we apply the above result to the motivic Donaldson-Thomas invariants. 

Let $Y$ be a smooth Calabi-Yau threefold.  Fixing a curve class $\beta\in H_2(Y,\zz)$ and an integer $n$, let $X:=M_n(Y,\beta)$ be the 
moduli space of stable sheaves  $F$ with topological data 
$(1,0,\beta,n)\in H^*(Y,\zz)$.  Then from \cite{BJM}, \cite{PTVV}, $X$ can be lifted to a (-1)-shifted symplectic derived scheme, and hence  a $d$-critical scheme structure $(X,s)$.  Thus there exists a unique  coherent sheaf $\FP_{X}$ on 
$X$.   Also from \cite{BJM},  there exists a canonical line bundle 
$K_{X,s}$, and if a square root $K_{X,s}^{\frac{1}{2}}$ exists, then there is a unique global motive $\mMF_{X,s}^{\phi}$
in $\overline{\mM}_{X}^{\hat{\mu}}$.

Take $X$ as a scheme over $\spec(\kappa[t])$. 
Let 
$$\XX:=\varprojlim_{n}(X/(t^n))$$
be the formal $t$-adic completion of $X$. Then $\XX$ is a {\em stft} formal $R$-scheme over $R$. 
Let $\FP_{\XX}=\widehat{\FP_{X}}$ be the formal completion of the coherent sheaf $\FP_{X}$ on $X$, then 
$(\XX, s\in\FP_{\XX})$ is a formal $d$-critical scheme. 

\begin{prop}\label{prop_localization_motivic_Donaldson-Thomas_Invariants}
Let $Y$ be a smooth Calabi-Yau threefold over $\kappa$ of character zero, and $X=M_n(Y,\beta)$ the moduli scheme of stable coherent sheaves in $\Coh(Y)$ with topological data $(1,0,\beta,n)$.  Then the $t$-adic formal completion 
$\XX=\widehat{X}$ of $X$ and its generic fiber $\XX_\eta$ have a formal $d$-critical scheme structure $(\XX,s)$ and 
a $d$-critical non-archimedean analytic space structure $(\XX_\eta,s)$.  
The canonical line bundle 
$K_{\XX,s}$ is isomorphic to the formal completion of the canonical line bundle 
$K_{X,s}$ where $(X,s)$ is the $d$-critical scheme in \cite{Joyce}.  Moreover, if there exists an orientation $K_{\XX,s}^{\frac{1}{2}}$, then there exists a unique 
$\mMF_{\XX,s}^{\phi}\in\overline{\mM}_{X}^{\hat{\mu}}$ such that if $X$ admits a good circle-compact $\Gm$-action which preserves the orientation $K_{\XX,s}^{\frac{1}{2}}$, then
$$\int_{\XX_s}\mMF_{\XX,s}^{\phi}=\sum_{i\in J}\ll^{-\ind^{\virt}(\XX_i^{\Gm}, X)/2}\odot \int_{\XX_i}\mMF_{\XX_i^{\Gm}, s_i^{\Gm}}^{\phi}$$
where $\XX^{\Gm}=\bigsqcup_{i\in J}\XX_i^{\Gm}$ is the fixed locus of $\XX$ under the $\Gm$-action. 
\end{prop}
\begin{proof}
We take $X$ as a scheme over $\spec(\kappa[t])$. 
The $t$-adic formal completion $\XX$ of $X$ is a {\em stft} formal scheme and it has a $d$-critical formal scheme structure $(\XX,s)$ from above arguments.  Hence from Proposition \ref{prop_generic_fiber_of_d_formal_scheme}, its generic fiber $\XX_\eta$ also has a 
$d$-critical non-archimedean analytic space structure $(\XX_\eta,s)$.

The good and circle-compact $\Gm$-action on $X$ will induce a good and circle-compact $\Gm$-action on $\XX$, and the 
formal completion $\XX_i^{\Gm}$ of the fixed locus $X_i^{\Gm}$ has a $d$-critical formal scheme structure $(\XX_i^{\Gm}, s_i^{\Gm})$. Hence the result just follows from Theorem \ref{thm_Maulik}. 
\end{proof}
\begin{rmk}
For a $d$-critical scheme $(X,s)$, the canonical line bundle $K_{X,s}$ is isomorphic to 
$\det(E^\bullet_{X})$, where $E_X^\bullet\to L_{X}^{\bullet}$ is the symmetric obstruction theory of $X$ determined by the $d$-critical scheme $(X,s)$ in the sense of \cite{Behrend}.  Recall that the $d$-critical scheme $(X,s)$ is the underlying classical scheme of a $(-1)$-shifted symplectic derived scheme $\mathbf{X}$, and  roughly speaking the  $(-1)$-shifted symplectic derived scheme $\mathbf{X}$ is the scheme $(X, E_X^\bullet)$ together with its cotangent complex $E^\bullet_{X}$. 
\end{rmk}

We  have the following interesting corollary.

\begin{cor}\label{cor_motivic_isolated_points}
Let $(\XX,s)$ be an oriented formal $d$-critical scheme over $R$ such that it is the formal completion of an oriented $d$-critical scheme of \cite{Joyce}.  Suppose that $\XX$ admits a good and circle-compact $\Gm$-action such that the $\Gm$-action  preserves the orientation $K_{\XX,s}^{\frac{1}{2}}$, and the torus fixed points consist of finitely many isolated points, then 
$$\int_{\XX_s}\mMF_{\XX,s}^{\phi}=\sum_{P\in\XX^{\Gm}}\ll^{-\frac{1}{2}\ind^{\virt}(\{P\}, \XX)}.$$
\end{cor}

\begin{rmk}
B. Szendroi pointed out that the formula in Corollary \ref{cor_motivic_isolated_points} is useful for the calculation of motivic Donaldson-Thomas invariants for some compact DT or PT moduli spaces, see \cite{CKK}. 
\end{rmk}

\begin{example}\textbf{Hilbert scheme of points on $\aaa_{\kappa}^3$.}
In the last section we talk about an interesting example, such that we can not find a $\Gm$-action satisfying the condition in Corollary \ref{cor_motivic_isolated_points}. 

Let  $X:=\Hilb^n(\aaa_\kappa^3)$ be the Hilbert scheme of $n$-points on $\spec(\kappa[x,y,z])=\aaa_{\kappa}^3$. 
When $n\geq 4$, $X$ is a singular variety. 
Let $M(n\times n)$ be the space of all $n\times n$ matrices over $\kappa$. Let 
$V$ be an $n$-dimensional complex vector space, and matrices $B_1,B_2,B_3\in \mbox{End}(V)$.
Let $v\in V$ and suppose that $B_1,B_2,B_3,v$ generate the vector space $V$. We say that 
the $5$-tuple $(V,B_1,B_2,B_3,v)$ satisfies the $stability$ condition if
there is no proper subspace $V_0\subset V$ such that $V_0$ is stable under 
$B_1,B_2,B_3$. Define an action of $GL_n$ on the set of $5$-tuples by
\begin{equation}\label{action}
P\cdot(V,B_1,B_2,B_3,v)=(V,PB_1P^{-1},PB_2P^{-1},PB_3P^{-1},Pv).
\end{equation}
Then we have a statement about the Hilbert scheme 
$\mbox{Hilb}^n(\mathbb{C}^3)$ of $n$-points on $\aaa_\kappa^3$.

\begin{prop}\label{hilbertscheme}
\[
\mbox{Hilb}^n(\aaa_{\kappa}^3)= 
\left\{
(V,B_i,v)\in M(n\times n)^3\times \aaa_{\kappa}^n \left|
 \begin{array}{l}
  \text{$dim V=n$, $B_i\in \mbox{End}(V)$,$v\in V$,} \\
  \text{(stability), and $v,B_i$ generate $V$,} \\
  \text{$B_1,B_2,B_3$ commute.}
 \end{array}
 \right\}\right/\cong,
\]
where $(V,B_1,B_2,B_3,v)\cong (V,B_1^{'},B_2^{'},B_3^{'},v^{'})$ if 
there is a matrix $P\in GL_n$ such that 
$P\cdot(V,B_1,B_2,B_3,v)=(V,B_1^{'},B_2^{'},B_3^{'},v^{'})$.
\end{prop}

Let 
\begin{equation}\label{m}
M:= 
\left\{
(V,B_i,v)\in M(n\times n)^3\times \aaa_{\kappa}^n \left|
 \begin{array}{l}
  \text{$dim V=n$, $B_i\in \mbox{End}(V)$,$v\in V$,} \\
  \text{(stability), and $v,B_i$ generate $V$.}
   \end{array}
 \right\}\right/\cong.
\end{equation}
Then in \cite{BBS}, the authors prove that 
the space $M$ is a smooth scheme.

Define a function 
\begin{equation}\label{function}
f: M\longrightarrow \aaa_{\kappa}^1
\end{equation}
by 
$$f((V,B_i,v))=\tr(B_1[B_2,B_3]),$$
where $\tr$ is the trace. From \cite{BBS}, 
\begin{equation}\label{hilbert-df}
\Hilb^n(\aaa_{\kappa}^3)=\Crit(f).
\end{equation}

Behrend, Bryan and Szendroi \cite{BBS} directly calculated the motivic generating function of the Hilbert scheme of points on 
$\aaa_{\kappa}^3$.   Define 
\begin{equation}\label{motivic_generating_series}
Z_{\aaa_{\kappa}^3}(T)=\sum_{n=0}^{\infty}[\Hilb^n(\aaa_{\kappa}^3)]^{\virt}T^n\in\mM_{\kappa}[\![T]\!]
\end{equation}
where $[\Hilb^n(\aaa_{\kappa}^3)]^{\virt}=\ll^{-\frac{1}{2}\dim(M)}[\mMF_{M,f}^{\phi}]$ and $\mMF_{M,f}^{\phi}$ is the motivic vanishing cycle of $f$, here 
$$\mMF_{M,f}^{\phi}=[M_0]-[\mMF_{M,f}]$$
and 
$\mMF_{M,f}$ is the motivic nearby cycle of $f$.  In \cite{BBS},   the scheme $X$ admits a circle-compact $\Gm$-action, but this action does not satisfy the conditions in Corollary \ref{cor_motivic_isolated_points}. This is because the $\Gm$-action on $X$ is induced from the action on $\aaa^3_{\kappa}$, which has all positive weights on $\aaa_\kappa^3$.  So it does not preserve the potential function $f$. 
But using this action, in the Appendix B of \cite{BBS},  Behrend-Bryan-Szendroi prove 
$$[\mMF_{M,f}]=[f^{-1}(1)]; \quad [\mMF_{M,f}^{\phi}]=[f^{-1}(1)]-[f^{-1}(0)].$$
Then one can directly calculate $f^{-1}(1)-f^{-1}(0)$ using the matrix space representation of the Hilbert scheme $X$.  The  beautiful calculation yields the following formula:
\begin{equation}\label{motivic_generating_series_formula}
Z_{\aaa_{\kappa}^3}(T)=\prod_{m=1}^{\infty}\prod_{k=0}^{m-1}(1-\ll^{k+2-\frac{m}{2}}T^{m})^{-1}.
\end{equation}

The torus $\Gm$-action on $X$ has isolated fixed points, which correspond to monomial ideals of $\kappa[x,y,z]$.  The monomial ideals of length $n$ are one-to-one correspondence to 3D plane partitions $P$.  Although the conditions in Corollary \ref{cor_motivic_isolated_points} do not hold, we still conjecture that 
\begin{conjecture}\label{motivic_generating_series_formula_cor}
$$Z_{\aaa_{\kappa}^3}(T)=\sum_{n=0}^{\infty}\left(\ll^{-\frac{1}{2}(\dim(T_PX)_+-\dim(T_PX)_-)}\right)T^n$$
\end{conjecture}
So Conjecture (\ref{motivic_generating_series_formula_cor}) must be the formula (\ref{motivic_generating_series_formula}). 

As pointed out in \cite[\S 8.2.5]{NO}, the index 
$$\ind^{\virt}(P,X)=\dim(T_PX)_+-\dim(T_PX)_-)$$
is a complicated function of 3D partitions.  In some special cases, this may be calculated by certain sum of boxes in the 3D partition $P$.  Bryan and Szendroi also have some calculations on the pattens of the boxes in the 3D partitions and also in some special cases, they can calculate the index. 

Of course it is really interesting if we can get the global formula (\ref{motivic_generating_series_formula}) from motivic localization formula (\ref{motivic_generating_series_formula_cor}).  At the moment, we can not achieve this goal. 
But at least the formula (\ref{motivic_generating_series_formula_cor}) gives some information of single value of the Behrend function 
on the isolated fixed points as:
$$\nu_{X}(P)=\lim_{\ll^{\frac{1}{2}}\to (-1)}\ll^{-\frac{1}{2}\ind^{\virt}}=(-1)^{n}. $$
We can not get the Behrend function information directly from the formula (\ref{motivic_generating_series_formula})  of Behrend-Bryan-Szendroi \cite{BBS}.
\begin{rmk}
As mentioned in \cite{NO}, D. Maulik in \cite{Maulik} used the motivic localization formula to prove the formulas in \cite[Theorem]{NO} is actually the refined Donaldson-Thomas invariants. 
\end{rmk}

\end{example}


\subsection*{}

\end{document}